\documentclass{article}
\usepackage[utf8]{inputenc}
\usepackage[english]{babel}

\usepackage{amsthm}
\usepackage{mathrsfs}
\usepackage{amssymb}
\usepackage{mathtools}
\usepackage[dvipsnames]{xcolor}
\usepackage[hidelinks,colorlinks]{hyperref}
\hypersetup{
	linkcolor = MidnightBlue, citecolor = BrickRed, urlcolor = ForestGreen
}
\usepackage{microtype}
\DisableLigatures[f]{encoding = *, family = * }
\usepackage{authblk}

\newcommand{\nocontentsline}[3]{}
\newcommand{\toclesslab}[3]{\bgroup\let\addcontentsline=\nocontentsline#1{#2\label{#3}}\egroup}
\newcommand{\beq}{\begin{equation}}
\newcommand{\eeq}{\end{equation}}

\newcommand{\pp}{\partial}
\newcommand{\vol}{\,\mbox{vol}}
\newcommand{\dvol}{\mathrm{d}{\rm{vol}}}
\newcommand{\inj}{{\rm{inj}}}
\newcommand{\diam}{{\rm{diam}}}
\newcommand{\ep}{\epsilon}
\newcommand{\vep}{\varepsilon}
\newcommand{\uvep}{u_{\vep}}
\newcommand{\veep}{v_{\ep}}
\newcommand{\tveep}{\tilde{v}_{\ep}}
\newcommand{\evep}{e_{\vep}}
\newcommand{\calE}{{\mathcal{E}}}
\newcommand{\calEvep}{\calE_{\vep}}
\newcommand{\Vep}{V_{\vep}}
\newcommand{\Kap}{K_{ap}}
\newcommand{\Kapg}{K_{ap,g}}
\newcommand{\Evep}{\widetilde{E}_{\vep}}

\newcommand{\Eepgr}{\widetilde{E}_{\ep,\gR}}
\newcommand{\muvep}{\mu_{\vep}}
\newcommand{\mus}{\mu_{*}}
\newcommand{\vphi}{\varphi}
\newcommand{\mrmd}{\mathrm{d}}
\newcommand{\gR}{g_{R_{1}}}
\newcommand{\ceq}{{\coloneqq}}

\newcommand{\N}{{\mathbb{N}}}
\newcommand{\Z}{{\mathbb{Z}}}

\newcommand{\R}{{\mathbb{R}}}
\newcommand{\C}{{\mathbb{C}}}
\newcommand{\bbS}{{\mathbb{S}}}
\newcommand{\calK}{{\mathcal{K}}}
\newcommand{\calM}{{\mathcal{M}}}

\newcommand{\calH}{{\mathcal{H}}}
\newcommand{\calL}{{\mathcal{L}}}
\def\rest{\hskip 1pt{\hbox to 10.8pt{\hfill
\vrule height 7pt width 0.4pt depth 0pt\hbox{\vrule height 0.4pt width 7.6pt depth 0pt}\hfill}}}

\definecolor{darkgreen}{rgb}{0,0.55,0} 

\numberwithin{equation}{section}

\theoremstyle{plain}\newtheorem{theorem}{Theorem}[section]
\theoremstyle{plain}\newtheorem{proposition}[theorem]{Proposition}
\theoremstyle{plain}\newtheorem{lemma}[theorem]{Lemma}
\theoremstyle{plain}\newtheorem{corollary}[theorem]{Corollary}
\theoremstyle{definition}\newtheorem{remark}[theorem]{Remark}
\theoremstyle{definition}\newtheorem{definition}[theorem]{Definition}

\title{Structural descriptions of limits of the parabolic Ginzburg-Landau equation on closed manifolds}

\begin{document}

\author{Andrew Colinet}
\affil{Department of Mathematics, University of Toronto, Toronto, ON M5S 2E4 Canada}

\maketitle

\abstract{
In the setting of a compact Riemannian manifold of dimension $N\ge3$ we provide a
structural description of the limiting behaviour of the energy measures of
solutions to the parabolic Ginzburg-Landau equation.
In particular, we provide a decomposition of the limiting energy measure into a
diffuse part, which is absolutely continuous with respect to the volume
measure, and a concentrated part supported on a codimension $2$
rectifiable subset.
We also demonstrate that the time evolution of the diffuse part is
determined by the heat equation while the concentrated part evolves according to
a Brakke flow.
This paper extends the work of Bethuel, Orlandi, and Smets from \cite{BOS2}.
}

\section{Introduction}
\hspace{15pt}In this paper we extend the work of Bethuel, Orlandi, and Smets on
the parabolic Ginzburg-Landau equation from \cite{BOS2} to
the setting of a compact Riemannian manifold $(M,g)$ of dimension
$N\ge3$.
More specifically, we are interested in providing a
detailed description of the limiting behaviour
as $\vep\to0^{+}$ of solutions of the PDE initial value problem
\beq\label{PGLOriginal}
\begin{cases}\tag*{$\rm{(PGL)}_{\vep}$}
\pp_{t}\uvep=\Delta\uvep+\frac{1}{\vep^{2}}\uvep\bigl(1-|\uvep|^{2}\bigr)&\forall{}x\in{}M\text{ and }\forall{}t>0\\
\uvep(x,0)=\uvep^{0}(x)&\forall{}x\in{}M
\end{cases}
\eeq
for a given $\uvep^{0}$ which, throughout this paper, we assume satisfies
\beq\label{H0}
\calEvep(\uvep^{0})\le{}M_{0}\mathopen{}\left|\log(\vep)\right|\mathclose{}\hspace{20pt}\text{where }M_{0}\text{ is a fixed positive constant}
\tag{${\rm{H}}_{0}$}
\eeq
and where
\beq\label{GL:Energy}
\calEvep(u)\ceq\int_{M}\!{}\evep(u)\dvol_{g},
\qquad
\evep(\uvep)\ceq\frac{1}{2}|\nabla{}u|^{2}+\Vep(u)
\eeq
with
\[
\Vep(x)\ceq\frac{1}{4\vep^{2}}\bigl(1-|x|^{2}\bigr)^{2}.
\]

The asymptotics of solutions to the equation \ref{PGLOriginal}
has been extensively studied in the setting of Euclidean space.
For $N\ge3$ it was shown in  \cite{JS2, Lin2},
in a variety of settings, including
$\R^{N}$ and bounded open subsets of $\R^{N}$, that for well-prepared initial data,
the energy of solutions to \ref{PGLOriginal} concentrates around a codimension $2$ mean curvature flow, as long as that flow remains smooth.
It was then shown in \cite{AS} that {\em if} the limiting energy measure satisfies a lower density 
bound, then this result may be extended past the formation of singularities, thereby giving a conditional proof of
convergence of rescaled energy measures, globally in time, 
to a codimension $2$ Brakke flow -- a measure-theoretic weak solution of the mean curvature flow.\\

Following this work results were obtained in \cite{LiRi2},
for $N=3$ and on a bounded domain, relating a local energy
condition with the local absence of vortex behaviour and using
this to demonstrate energy concentration on a rectifiable
$1$-varifold.
The relationship between the local energy condition and the
absence of vortex behaviour was shown to hold for $\R^{4}$ in
\cite{Wang} where the energy is weighted by a Gaussian function.\\

Finally, this line of research concluded with \cite{BOS2} which,
among other improvements, removed the lower density bound imposed
in \cite{AS},  giving an unconditional proof that the concentrated part of a limiting energy measure evolves via a Brakke flow in $\R^N$, globally in $t$, for every $N\ge 3$, and without requiring well-prepared initial data.\\

The description of the dynamics of the limiting energy measure
over $\R^{N}$ in \cite{BOS2} raised the question of possible
extensions to other settings.
One such extension is found in \cite{Liu} who demonstrated the
conclusions of \cite{BOS2} for the parabolic
Ginzburg-Landau equation with magnetic potential in $\R^{3}$.
Related work in the Riemannian setting includes
\cite{PP1} and \cite{PP2} for the Allen-Cahn equation over a
compact Riemannian manifold without boundary as well as
\cite{PCC} which extends the Monotonicity formula to a suitably
restricted class of compact Riemannian manifolds, possibly with
boundary.
Despite these efforts, an extension of the results of
\cite{BOS2} to the case of a compact Riemannian manifold without
boundary has not been shown.\\

The main result of this paper is, in the setting of a compact
smooth Riemannian manifold $(M,g)$ without boundary,
a careful study of the family of energy measures
\[
\mus^{t}(x)\ceq\frac{\evep(\uvep(x,t))}
{\mathopen{}\left|\log(\vep)\right|\mathclose{}}
\dvol_{g}(x)
\]
for $t>0$ as $\vep\to0^{+}$.
Of particular note is that we impose no topological or curvature
restrictions on $M$ beyond what is guaranteed by compactness.
As a result, our analysis applies to compact manifolds with
possibly non-trivial topology.
The result of this analysis, stated in Theorem
\ref{BOSTheorem}, is that the limiting energy decomposes into
a diffuse energy and a concentrated vortex energy which
do not interact.
The evolution of the diffuse energy will be governed by
the heat equation while the vortex energy evolves according to
a Brakke flow, a measure theoretic formulation of
mean curvature flow.
More specifically, we have:
\begin{theorem}\label{BOSTheorem}
Let $M$ be of dimension $N\ge3$ and suppose that $\{\uvep\}_{\vep\in(0,1)}$ are a family of solutions to \ref{PGLOriginal} for
corresponding $\vep$ and with respective initial data $\{\uvep^{0}\}_{\vep\in(0,1)}$.
Let $\muvep^{t}$ be, for each $t>0$,
the measure on $M$ defined by
\[
\muvep^{t}\ceq\frac{\evep(\uvep(\cdot,t))}{\mathopen{}\left|\log(\vep)\right|\mathclose{}}\dvol_{g}.
\]
Then, after perhaps passing to a subsequence $\{u_{\vep_{n}}\}_{n\in\N}$, there exists a
family of
limiting measures $\{\mus^{t}\}_{t>0}$ and
subsets $\{\Sigma_{\mus}^{t}\}_{t>0}$ in $M$, as well as a
function
$\Phi_{*}\colon{}M\times(0,\infty)\to\R\slash2\pi\Z$ such that the
following properties hold:
\begin{enumerate}
\item\label{BOSTheorem:Item1}
$\mu_{\vep_{n}}^{t}\rightharpoonup\mus^{t}$ in $\calM(M)$ for each $t>0$.
\item\label{BOSTheorem:Item2}
$\Phi_{*}$ satisfies the heat equation on
$M\times(0,\infty)$.
\item\label{BOSTheorem:Item3}
For each $t>0$, the measure $\mus^{t}$ can be exactly decomposed as
\beq\label{Decomposition}
\mus^{t}=\frac{|\nabla\Phi_{*}|^{2}}{2}\calH^{N}+\nu_{*}^{t}
\eeq
where
\beq\label{nustar.form}
\nu_{*}^{t}=\Theta_{*}(x,t)\calH^{N-2}\rest\Sigma_{\mus}^{t}
\eeq
and where $\Theta_{*}(\cdot,t)$ is a bounded measurable function.
\item\label{BOSTheorem:Item4}
There exists a positive function $\eta$ defined on $(0,\infty)$ such that, for $\calL^{1}$-almost every $t>0$, the set $\Sigma_{\mus}^{t}$ is
$(N-2)$-rectifiable and
\beq\label{LowerBound}
\Theta_{*}(x,t)=\Theta_{N-2}(\mus^{t},x)=\lim_{r\to0^{+}}\frac{\mus^{t}(B_{r}(x))}{\omega_{N-2}r^{N-2}}\ge\eta(t),
\eeq
for $\calH^{N-2}$-almost every $x\in\Sigma_{\mus}^{t}$.
\item\label{BOSTheorem:Item5}
The family of measures $t\mapsto\Theta_{*}(x,t)\calH^{N-2}\rest\Sigma_{\mus}^{t}$ forms a Brakke flow.
\end{enumerate}
\end{theorem}

These conclusions were first demonstrated in \cite{BOS2}
for the, non-compact, smooth manifold $\R^{N}$ paired with the
standard metric.\\

In general we follow the strategy developed in
\cite{BOS2}.
However, a number of details need to be adapted in order for the
strategy to extend to the more general setting.
\begin{itemize}
\item
When defining the weighted energy, which is used to
establish a monotonicity formula, we use an approximation to
the heat kernel as a weight.
The form of the alteration that we employ differs from the
earlier works \cite{PP1}, \cite{PP2} and is designed to facilitate
a comparison of the weighted energy at distinct points in
space-time, see Lemma \ref{WeightedEnergyComparison}.
\item
A consequence of modifying the weighted energy is that new
error terms $\Phi$ and $\Psi$ arise, see \eqref{Phi:Def} and
\eqref{Psi:Def} for definitions.
The error term $\Phi$, as seen in Theorem $1.1$ of \cite{Ham2},
corresponds to the fact that we are not working over Euclidean
space while $\Psi$, as seen in \cite{PCC}, reflects the
fact that we have replaced the heat kernel on $M$ with an
approximation.
These error terms are handled by appealing to the Hessian
Comparison Theorem which is discussed in
\eqref{HessianComparisonInequality}.
\item
When following the Hodge de Rham decomposition strategy from
Subsection $3.6$ of \cite{BOS2} we need to solve a Poisson
problem over $M$.
Since we do not impose any topological restrictions on $M$
some care is needed to ensure that a solution exists.
Specifically, we needed to modify the argument from
\cite{BOS2} to account for the harmonic part of the data as well
as provide additional estimates for the resultant error terms.
\item
When decomposing the solution to \ref{PGLOriginal}, as in Theorem
$3$ of \cite{BOS2}, we now have to account for the fact that no
topological restrictions were placed on $M$.
This, in particular, has the effect of adding an additional
term, $u_{h,\vep}\colon{}M\times(0,\infty)\to\bbS^{1}$, which
corresponds to the harmonic part of the Hodge de Rham
decomposition of $\uvep\times\mrmd\uvep$ at time $t=0$.
The presence of this additional term also has consequences
on how we are able to express the limiting energy density in
Theorem \ref{BOSTheorem}.
\end{itemize}

The use of the Hessian Comparison Theorem gives rise to curvature-dependent constants in many of our estimates. 
In our arguments, it is often convenient to rescale the metric $g$ to a dilated metric $g/a$ with $a\in (0,1)$. 
All estimates that we need continue to hold with the same, often better, constants after such rescaling. Indeed, such a rescaling decreases bounds on the curvature and hence improves all curvature-dependent constants.

Inevitably, there are numerous arguments in the proof of Theorem \ref{BOSTheorem},
that are very
similar to corresponding points in \cite{BOS2}.
We omit  discussions that would essentially duplicate prior arguments.
However, we have taken a couple of steps to explain these points and
and to document their correctness.
First, we attempt to sketch these proofs well enough to 
make it clear that no significant new
subtleties arise in the Riemannian case. As a result, in places our exposition
resembles a sort of reader's guide to parts of \cite{BOS2}.
This seems to us necessary for a reasonably complete account of the proof of Theorem \ref{BOSTheorem}.
Second, the author's Ph.D. thesis \cite{Col} contains an expanded version of this paper, and it includes an appendix in which
we discuss in detail a number of the points omitted here.
These are arguments that involve few novel ingredients, but for which some documentation may be useful.
We refer to this appendix often.

While Theorem \ref{BOSTheorem} is interesting in its own right
it is worth noting that this result is a key ingredient
in demonstrating the existence of solutions to the elliptic
Ginzburg-Landau equation over $(M,g)$, when $N=3$, for which the energy and
a quantity associated to vorticity concentrate about a non-length minimizing
geodesic as $\vep\to0^{+}$.
This is shown in \cite{CJS} which
improves on earlier work such as \cite{JSt} and \cite{Mes}.\\

We conclude this introduction by describing some issues in 
the proof of Theorem \ref{BOSTheorem}.
First, as in \cite{BOS2}, an important intermediate result is the
following ``clearing out" theorem.
It involves a weighted energy, $\Evep$, whose definition
is provided in \eqref{WeightedEnergyIntegral}.

\begin{theorem}\label{Theorem1}
For any $\sigma\in(0,1)$ and $T>0$ there exists positive numbers $\vep_{0}$,
$R(\sigma)$, and $\eta(\sigma)$ such that if $\uvep$ is a solution to
\ref{PGLOriginal} on $M\times(0,T)$ satisfying \eqref{H0} for $0<\vep<\vep_{0}$, $R$
satisfies $\sqrt{2\vep}<R<\min\bigl\{R(\sigma),\sqrt{T}\bigr\}$,
and $x_{T}$ is a point such that
\[
\Evep(\uvep,(x_{T},T),R)
\le\eta(\sigma)\mathopen{}\left|\log(\vep)\right|\mathclose{},
\]
then
\[
|\uvep(x_{T},T)|\ge1-\sigma.
\]
\end{theorem}

The overall strategy of the proof follows that of Theorem $1$
in \cite{BOS2}, on which Theorem \ref{Theorem1} is modelled.
We start by presenting an overview in Section \ref{Sec:Outline}, drawing on the work of \cite{BOS2}.
In this overview we highlight elements of the proof in which substantial new considerations arise. All such points are treated in detail in Section \ref{Sec::SteppingStones}. The overview of Section \ref{Sec:Outline} also identifies many aspects of the proof
that carry over to the Riemannian setting with only superficial changes. 
Detailed verification of these points can be found in Appendix A of
\cite{Col}.
In addition, for such points we attempt in Section \ref{Sec:Outline} to describe the underlying ideas 
in sufficient detail to explain why the arguments
of \cite{BOS2} do not involve any substantial changes 
in the Riemannian context.

The next result is an adaptation of Theorem $3$ from \cite{BOS2}.
New issues arise from the possibly non-trivial topology of $M$.
This is reflected in the presence of the $\bbS^{1}$-valued
map $u_{h,\vep}$.
We refer the reader to \eqref{jnotation} for the definition of $ju$, where $u\colon{}M\to\C$, which is used in the statement of the next theorem.
\begin{theorem}\label{Theorem3}
Suppose $\uvep$ satisfies \emph{\ref{PGLOriginal}} and
\eqref{H0}.
Then there exists an $\bbS^{1}$-valued
function $u_{h,\vep}$, depending only on the initial data of $\uvep$ such that,
for any compact set $\calK\subset{}M\times(0,\infty)$ and $\vep$
sufficiently small,
there is a real-valued function $\phi_{\vep}$ and a complex-valued function $w_{\vep}$ defined on a neighbourhood of $\calK$, such that
\begin{enumerate}
\item\label{Theorem3Item1}
$\uvep=w_{\vep}e^{i\phi_{\vep}}u_{h,\vep}
\hspace{5pt}\text{on }\calK$,
\item\label{Theorem3Item2}
$\phi_{\vep}\text{ verifies the heat equation on }\calK$,
\item\label{Theorem3Item3}
$|\nabla\phi_{\vep}(x,t)|\le{}C(\calK)\sqrt{(M_{0}+1)
\mathopen{}\left|\log(\vep)\right|\mathclose{}}
\text{ for all }(x,t)\in\calK$,
\item\label{Theorem3Item4}
$\left\|\nabla{}w_{\vep}\right\|_{L^{p}(\calK)}\le{}C(p,\calK)\text{, for any }1\le{}p<\frac{N+1}{N}$,
\item\label{Theorem3Item5}
$u_{h,\vep}$ does not depend on $t$, $ju_{h,\vep}$ is a harmonic $1$-form on $M$,
and
\[
|\nabla{}u_{h,\vep}(x,t)|\le{}K_{M}\sqrt{M_{0}\mathopen{}\left|\log(\vep)\right|\mathclose{}}\text{ for all }(x,t)\in\calK.
\]
\end{enumerate}
Here, $C(\calK)$ and $C(p,\calK)$ are constants depending only on $\calK$ and $\calK$,p (and $M_{0}$) respectively and $K_{M}$ is a constant depending only on $M$.
\end{theorem}

This is proved in Section \ref{Sec::Decompositions}.
Finally, the proof of Theorem \ref{BOSTheorem} is completed in Section \ref{Sec::Limiting}.\\

\emph{Acknowledgements.}
I would like to thank Prof. Robert Jerrard for his assistance, support, and
guidance in the preparation of this paper.
I would like to thank Prof. Peter Sternberg for helpful comments made to improve
the paper.
I would also like to thank Prof. Giandomenico Orlandi for the very helpful
references on Green's functions.

\section{Preliminaries}\label{Prelim}

In this section we record some of the specialized notation and definitions used throughout this paper.
\newline

At each $x\in{}M$ we use $\mathopen{}\left<\cdot,\cdot\right>\mathclose{}_{g}$ and
$|\cdot|_{g}$ to
denote, respectively, the inner product and norm on $T_{x}M$ given by $g$.
For $x,y\in{}M$ we use $d_{g}(x,y)$ to denote the distance between $x$ and $y$
induced by the metric $g$.
For $p\in{}M$ and $r>0$ we use the notation $B_{r,g}(p)$ to denote the geodesic ball
about $p$ of radius $r$ in the metric $g$ which is defined
by
\[
B_{r,g}(p)\ceq\{x\in{}M:d_{g}(x,p)<r\}.
\]
We will write $\vol_{g}$ to denote the unique Radon measure on $M$ with the
property that 
$\vol_{g}(A)$ is the Riemannian volume of $A$ for all sufficiently
regular $A$,
and for non-negative $f\in L^1(M;\vol_{g})$, we write $f\vol_{g}$ to denote the
measure defined by
\[
f\vol_{g}(A)\ceq \int_{A}\!{}f\dvol_{g}.
\]
We define the \emph{injectivity radius of $M$ according to the metric $g$}, denoted $\inj_{g}(M)$, by
\beq\label{InjectivityRadius}
\inj_{g}(M)\ceq\sup\biggl\{r>0\biggm\vert
\begin{aligned}
&\exp_{x}\colon{}T_{x}M\to{}M
\text{ is a diffeomorphism}\\
&\hspace{25pt}\text{onto }B_{r,g}(x)\text{ for all }x\in{}M
\end{aligned}
\biggr\}.
\eeq

We define the \emph{diameter of $M$ according to the metric $g$}, denoted $\diam_{g}(M)$, by
\beq\label{Diameter}
\diam_{g}(M)\ceq\sup\left\{d_{g}(x,y):\forall{}x,y\in{}M\right\}.
\eeq
In the above notation we may, for convenience, remove the subscript $g$.
We note for $p\in{}M$ and $0<s<\inj_{g}(M)$ that the function
\beq\label{dsquared}
r(x)\ceq\frac{1}{2}(d_{g}(x,p))^{2}
\eeq
satisfies
\begin{equation}\label{DistanceSquaredGradient}
\nabla{}r(x)=-\text{exp}_{x}^{-1}(p)
\end{equation}
on $B_{s}(p)$, see Theorem $6.6.1$ of \cite{Jos}.
Also, if the sectional curvature, $K$, of $M$ satisfies
\begin{equation*}
\lambda\le{}K\le\mu,\hspace{10pt}\text{with }\lambda\le0\le\mu
\end{equation*}
then, for
$0<\rho<\min\Bigl\{\frac{\pi}{2\sqrt{}\mu},\inj_{g}(M)\Bigr\}$ if $\mu>0$ and $0<\rho<\inj_{g}(M)$ otherwise, we have
\begin{equation}\label{HessianComparisonInequality}
\sqrt{\mu}d(x,p)\cot\bigl(\sqrt{\mu}d(x,p)\bigr)|v|^{2}
\le\text{Hess}(r)(v,v)
\le\sqrt{|\lambda|}d(x,p)\coth\bigl(\sqrt{|\lambda|}d(x,p)\bigr)|v|^{2}
\end{equation}
for $x\in{}B_{\rho}(p)$ and $v\in{}T_{x}M$, see
Theorem $6.6.1$ of \cite{Jos}.
This is referred to as the Hessian Comparison Theorem.\\

We use the notation $\Lambda_{\alpha}(x_{0},T,R,\Delta{}T)$ for $0<\alpha\le1$, $x_{0}\in{}M$, $T\ge0$, $\Delta{}T>0$, and
$R>0$ to refer to
\begin{equation}\label{ParabolicCylinder}
\Lambda_{\alpha}(x_{0},T,R,\Delta{}T)
\ceq{}B_{\alpha{}R}(x_{0})\times[T+(1-\alpha^{2})\Delta{}T,T+\Delta{}T].\end{equation}
We also use the abbreviations $\Lambda_{\alpha}$ for
\eqref{ParabolicCylinder} and $\Lambda\ceq\Lambda_{1}(x_{0},T,R,\Delta{}T)$ when the
other parameters are understood.\\

For $y\in{}M$ we define the \emph{approximate heat
kernel about $y$ evaluated at $(x,t)\in{}M\times(0,\infty)$},
denoted $\Kap(x,t;y)$, by
\beq\label{AHK}
\Kap(x,t;y)\ceq\frac{1}{(4\pi{}t)^{\frac{N}{2}}}
\exp\Biggl[\frac{-(d_{+,g}(x,y))^{2}}{4t}\Biggr]
\eeq
where $d_{+,g}\colon{}M\times{}M\to[0,\infty)$ is a smooth function defined so that
\beq\label{dplus}
d_{+,g}(x,y)\ceq\inj_{g}(M)f\left(\frac{d_{g}(x,y)}{\inj_{g}(M)}\right)
\eeq
where $f\colon[0,\infty)\to[0,\infty)$ is a smooth function chosen so that
\begin{enumerate}
\item\label{f1}
$f(s)=s \hspace{10pt}\text{ for }s\in{}\bigl[0,\frac{1}{2}\bigr]$,
\item\label{f2}
$f(s)=1 \hspace{10pt}\text{ for }s\ge1$,
\item\label{f3}
$f(s)\ge{}s\hspace{10pt}\text{ for }0\le{}s\le1$,
\item\label{f4}
$f \text{ is non-decreasing}$,
\item\label{f5}
$\left\|f'\right\|_{L^{\infty}(\R)}<\sqrt{2}$.
\end{enumerate}
We note that $d_{+,g}$ satisfies
\beq\label{dInequality}
c_{*}d_{g}(x,y)\le{}d_{+,g}(x,y)\le{}2d_{g}(x,y)
\eeq
where
\[
c_{*}\ceq\frac{\inj_{g}(M)}{\diam_{g}(M)}.
\]
We will use the notation $\Kapg(x,t;x_{*})$ when we wish to explicitly indicate the dependence of $\Kap$ on the metric
$g$.
Also, for a fixed point $p\in{}M$ we use the notation
$r_{+}$ to denote
\beq\label{rplus}
r_{+}(x)\ceq\frac{1}{2}(d_{+}(x,p))^{2}.
\eeq

Next we introduce notation for energy weighted by the approximate heat kernel on
$M$.
For $z_{*}=(x_{*},t_{*})\in{}M\times(0,\infty)$ and $0<R\le{}\sqrt{t_{*}}$ we use the notation
\beq\label{WeightedEnergyIntegral}
\Evep(z_{*},R)\ceq{}R^{2}\int_{M}\!{}\evep(u(x,t_{*}-R^{2}))\Kap(x,R^{2};x_{*})\dvol_{g}(x).
\eeq
We may also use variations of this notation which include $g$ in the subscript to emphasize particular dependence on the metric.

For a given $u\colon{}M\to\C$ we introduce the notation $ju$ for the $1$-form
\beq\label{jnotation}
ju\ceq{}u\times\mrmd{}u
\eeq
which in coordinates can be expressed as
\[
u\times\mrmd{}u\ceq\sum_{i=1}^{N}u\times\frac{\pp{}u}{\pp{}x_{i}}\mrmd{}x^{i}.
\]

Now we provide a series of definitions related to Brakke
flows.
\begin{definition}\label{RectifiableMeasure}
A Radon measure $\nu$ on $M$ is said to be \emph{$k$-rectifiable} if there exists a $k$-rectifiable set $\Sigma$, and a density function
$\Theta\in{}L_{loc}^{1}(\calH^{k}\rest\Sigma)$ such that
\begin{equation}\label{Def:RectMeas}
\nu=\Theta(\cdot)\calH^{k}\rest\Sigma.
\end{equation}
\end{definition}
Next, we define the distributional first variation of a rectifiable Radon measure.
To do this, we remark that if $\Sigma$ is $k$-rectifiable then at $\calH^{k}$-almost every point $x\in\Sigma$ there is a unique tangent space
$T_{x}\Sigma$ belonging to the Grassmannian $G_{N,k,x}$.
Similar to \cite{BOS2} we associate $G_{N,k,x}$ to projection operators onto $k$-dimensional subspaces of $T_{x}M$.
\begin{definition}\label{DistributionalFirstVariation}
Let $\nu$ be a $k$-rectifiable Radon measure.
Then we define the \emph{distributional first variation of} $\nu$ to be the distribution, $\delta{}v$, defined by
\begin{equation}\label{DistributionalVariationMeasure}
\delta\nu(X)\ceq\int_{\Sigma}\!{}\text{div}_{T_{x}\Sigma}(X)\mrmd\nu\hspace{5pt}\text{for all }X\in{}\chi(M)
\end{equation}
where $\chi(M)$ denotes the space of smooth vector fields over $M$ and, following Section $2$ of \cite{PP2}, we
define
\beq\label{Def:div}
\text{div}_{T_{x}\Sigma}(X)\ceq
\sum_{k=1}^{N-2}\mathopen{}
\left<D_{e_{i}}X(x),e_{i}\right>\mathclose{}
\eeq
where $\{e_{1},e_{2},\ldots,e_{N-2}\}$ denote any
orthonormal basis of $T_{x}\Sigma$ and $D_{e_{i}}X$ denote
the associated covariant derivatives.
When $|\delta\nu|$ is absolutely continuous with respect to $\nu$, we say that $\nu$ has a \emph{first variation} and we may write
\begin{equation*}
\delta\nu=H\nu
\end{equation*}
where $H$ is the Radon-Nikodym derivative of $\delta\nu$ with respect to $\nu$.
In this case, \eqref{DistributionalVariationMeasure} becomes
\begin{equation}\label{FirstVariationAbsolutelyContinuous}
\int_{\Sigma}\!{}\text{div}_{T_{x}\Sigma}(X)\mrmd\nu=\int_{\Sigma}\!{}\mathopen{}\left<H,X\right>\mathclose{}\mrmd\nu.
\end{equation}
\end{definition}
Next, we let $\{\nu^{t}\}_{t\ge0}$ be a family of Radon measures on $M$.
For $\chi\in{}C^{2}(M;(0,\infty))$, we define
\begin{equation*}
\overline{D}_{t}\nu_{0}^{t}(\chi)\ceq\limsup_{t\to{}t_{0}}\frac{\nu^{t}(\chi)-\nu^{t_{0}}(\chi)}{t-t_{0}}.
\end{equation*}
If $\nu^{t}\rest\{\chi>0\}$ is a $k$-rectifiable measure which has a first variation verifying
$\chi|H|^{2}\in{}L^{1}(\nu^{t})$, then we set
\begin{equation*}
\mathcal{B}\bigl(\nu^{t},\chi\bigr)\ceq-\int{}\chi|H|^{2}\mrmd\nu^{t}+\int\!{}\mathopen{}\left<\nabla\chi,P(H)\right>\mathclose{}\mrmd\nu^{t},
\end{equation*}
where $P$, as in Section $2$ of \cite{PP2} and consistent
with our identification of the Grassmannian with
projections, denotes $\calH^{k}$-almost everywhere the orthogonal projection onto the tangent space to $\nu^{t}$, otherwise, we set
\begin{equation*}
\mathcal{B}(\nu^{t},\chi)=-\infty.
\end{equation*}
We are now in a position to give the definition of a Brakke flow.
\begin{definition}\label{BrakkeFlow}
Let $\{\nu^{t}\}_{t\ge0}$ be a family of Radon measures on $M$.
We say that $\{\nu^{t}\}_{t\ge0}$ is a $k$\emph{-dimensional Brakke flow} if and only if
\begin{equation}\label{BrakkeFlowIdentity}
\overline{D}_{t}\nu^{t}(\chi)\le\mathcal{B}(\nu^{t},\chi),
\end{equation}
for every $\chi\in{}C^{\infty}(M;(0,\infty))$ and for all $t\ge0$.
\end{definition}

\section{Toolbox}\label{Toolbox}
We record a few helpful results that will be needed for the proof of Theorem \ref{BOSTheorem}.
These are generalizations of corresponding results found in
\cite{BOS2}.

\begin{lemma}\label{DerivativeOfEnergy}
Let $\chi$ be a Lipschitz function on $M$.
Then, for any $T\ge0$, at $t=T$,
\beq\label{OriginalEnergyTestFunction}
\frac{\mrmd}{\mrmd{}t}\int_{M\times\{t\}}\!{}
\evep(\uvep)\chi(x)
=-\int_{M\times\{T\}}\!{}|\pp_{t}\uvep|^{2}\chi(x)
-\int_{M\times\{T\}}\!{}
\pp_{t}\uvep\cdot\mathopen{}\left<\nabla{}\uvep,\nabla\chi\right>\mathclose{}
\eeq
and
\beq\label{EnergyTestFunction}
\frac{1}{2}\int_{M\times\{t\}}\!{}|\pp_{t}\uvep|^{2}\chi^{2}
+\frac{\mrmd}{\mrmd{}t}\int_{M\times\{t\}}\!{}\evep(\uvep)\chi^{2}
\le4\left\|\nabla\chi\right\|_{L^{\infty}}^{2}\int_{\emph{supp}(\chi)}\!{}\evep(\uvep).
\eeq
In particular, for any $0\le{}T_{1}\le{}T_{2}$,
\begin{align}\label{OriginalEnergyTestFunctionIntegrated}
\int_{M\times\{T_{2}\}}\!{}&\evep(\uvep)\chi(x)
-\int_{M\times\{T_{1}\}}\!{}\evep(\uvep)\chi(x)\\
&=-\int_{M\times[T_{1},T_{2}]}\!{}|\pp_{t}\uvep|^{2}\chi(x)
-\int_{M\times[T_{1},T_{2}]}\!{}\pp_{t}\uvep\cdot
\mathopen{}\left<\nabla\uvep,\nabla\chi\right>\mathclose{}.
\nonumber
\end{align}
\end{lemma}
\begin{proof}
The proof of \eqref{OriginalEnergyTestFunction} follows from differentiation under the integral while
\eqref{OriginalEnergyTestFunctionIntegrated} follows by integrating \eqref{OriginalEnergyTestFunction} in $t$.
To see \eqref{EnergyTestFunction} we replace $\chi$ with $\chi^{2}$ in Lemma \ref{DerivativeOfEnergy} and use standard estimates.
\end{proof}

The next result, the basis for a monotonicity formula, will play a fundamental role in the proof of Theorem \ref{Theorem1}.

\begin{lemma}\label{MonotonicityFormulaApproximateHeatKernel}
Suppose $(M,g)$ is an $N$-dimensional compact Riemannian manifold without boundary
and suppose that $\uvep$ solves \ref{PGLOriginal} on $M$.
Let $\Kap$ be the approximate heat kernel as in \eqref{AHK}.
Then for $0<R<\sqrt{T}$ and $y\in{}M$:
\begin{align}
Z'(R)
&=2R\int_{M}\!{}\Bigl[\Vep(\uvep(x,T-R^{2}))
+\Xi(\uvep,(y,T))(x,T-R^{2})\Bigr]
\label{Diff:AHK}\\
&+2R\int_{M}\!{}\Psi(\uvep,(y,T))(x,T-R^{2})    \nonumber\\
&+2R\int_{M}\!{}\Phi(\uvep,(y,T))(x,T-R^{2})    \nonumber
\end{align}
where
\beq\label{ZR}
Z(R)\ceq{}R^{2}\int_{M}\!{}\evep(\uvep(x,T-R^{2}))
\Kap\bigl(x,R^{2};y\bigr)\dvol_{g}(x)
\eeq
and where, for $0<t<T$, we have set
\begin{align}
\Xi(\uvep,(y,T))(x,t)
&\ceq(T-t)\biggl|\pp_{t}\uvep(x,t)
+\frac{\mathopen{}\left<\nabla\uvep(x,t),\nabla\Kap(x,T-t;y)\right>\mathclose{}}{\Kap(x,T-t;y)}\biggr|^{2},
\label{Xi:Def}\\
\Phi(\uvep,(y,T))(x,t)&\ceq(T-t)\mathopen{}\biggl[\emph{Hess}(\Kap(x,T-t;y))(\nabla\uvep(x,t),\nabla\uvep(x,t)),
\label{Phi:Def}\\
&-\frac{|\mathopen{}\left<\nabla\uvep(x,t),\nabla\Kap(x,T-t;y)\right>\mathclose{}|^{2}}{\Kap(x,T-t;y)}\mathclose{}
+\mathopen{}\frac{|\nabla\uvep|^{2}\Kap(x,T-t;y)}{2(T-t)}\biggr]\mathclose{},
\nonumber\\
\Psi(\uvep,(y,T))(x,t)&\ceq(T-t)
\evep(\uvep(x,t))[(\pp_{t}\Kap)(x,T-t;y)-(\Delta\Kap)(x,T-t;y)].    \label{Psi:Def}
\end{align}
We also have, for any $z_{T}=(x_{T},T)\in{}M\times(0,\infty)$ and $R_{*}=\sqrt{T}$, that
\begin{align}
\Evep(z_{T},R_{*})&=
\int_{M\times[0,T]}\!{}(\Vep(\uvep)
+\Xi(\uvep,z_{T}))\Kap(x,T-t;x_{T})
\dvol_{g}(x)\mrmd{}t  \label{TimeIntegratedMonotonicity}  \\
&+\int_{M\times[0,T]}\!{}\Psi(\uvep,z_{T})\dvol_{g}(x)\mrmd{}t
+\int_{M\times[0,T]}\!{}\Phi(\uvep,z_{T})\dvol_{g}(x)\mrmd{}t. \nonumber
\end{align}
\end{lemma}
\begin{proof}
Computations like \eqref{Diff:AHK} are quite standard,
and very similar ones can be found for example in the proof of
Theorem $2.1$ of \cite{PCC}.
Then \eqref{TimeIntegratedMonotonicity} follows by integrating \eqref{Diff:AHK} from
$R=0$ to $R=\sqrt{T}$ and changing variables.
For a detailed exposition see A.3.1.1 of \cite{Col}.
\end{proof}

As remarked in the introduction, the terms $\Phi$ and $\Psi$  reflect the non-Euclidean character of the metric and the use of the approximate, rather than exact, heat kernel. They are estimated using arguments that ultimately rely on the Hessian Comparison Theorem. We illustrate this first for $\Psi$.

\begin{lemma}\label{HeatKernelTermEstimate}
Let $(M,g)$ be an $N$-dimensional compact Riemannian manifold and suppose $y\in{}M$.
Let $\Kap$ be the approximate heat kernel from \eqref{AHK} and
$\Psi$ be as in \eqref{Psi:Def}.
Then there is $c_{0}>0$ such that for all $0<t<T$ we have
\beq\label{HeatErrorTermLowerBound}
\int_{M}\!{}\Psi(u,(y,T))(x,T-t)
\ge\frac{-N\mu{}t^{\frac{1}{2}}}{4}
\int_{M}\!{}\evep(u)\Kap
-c_{0}t\int_{M}\!{}\evep(u)
\eeq
where the constants remain bounded when dividing the metric by $0<a\le1$ and
we have used the abbreviations $\Kap$ for
$\Kap(x,t;y)$ and $u$ for $u(x,T-t)$.
Similarly, there is $c_{1}>0$ such that for all $0<t<T$ we have
\beq\label{HeatErrorTermUpperBound}
\int_{M}\!{}\Psi(u,(y,T))(x,T-t)
\le{}\frac{N|\lambda|t^{\frac{1}{2}}}{6}
\int_{M}\!{}\evep(u)\Kap
+c_{1}t\int_{M}\!{}\evep(u).
\eeq
where the constants remain bounded when dividing the metric by $0<a\le1$.
It is worth noting that we also have
\beq\label{HeatErrorTermDistanceBound}
\int_{M}\!{}\Psi(u,(y,T))(x,T-t)
\le\frac{N|\lambda|}{6}\int_{M}\!{}(d_{+}(x,y))^{2}\evep(u)\Kap
+C_{M}\int_{M}\!{}\evep(u)\Kap
+C_{0}E_{0}t
\eeq
where $C_{M},C_{0}$ remain bounded when dividing the metric by $0<a\le1$.
\end{lemma}
\begin{proof}
By computing $\pp_{t}\Kap-\Delta\Kap$ we obtain, using the notation from \eqref{rplus}, that
\[
\pp_{t}\Kap-\Delta\Kap=\frac{[\Delta{}r_{+}(x)-N]}{2t}\Kap
+\frac{r_{+}(x)-\frac{1}{2}|\nabla{}r_{+}(x)|^{2}}{2t^{2}}K_{ap}.
\]
First observe that if $s\ceq\min\Bigl\{\frac{\pi}{4\sqrt{\mu}},\frac{\inj(M)}{2}\Bigr\}$ and
$x\in{}B_{s}(y)$ then the rightmost term is zero and by using
the notation \eqref{dsquared} as well as \eqref{HessianComparisonInequality} 
we obtain
\beq\label{DistanceLaplacianEstimate}
\frac{[\Delta{}r_{+}(x)-N]}{2t}\Kap
=\frac{[\Delta{}r(x)-N]}{2t}\Kap\ge\frac{-N\mu(d(x,y))^{2}}{4t}\Kap.
\eeq
Next, observe that for $x\in{}M\setminus{}B_{s}(y)$ we have
\beq\label{ResidualLaplacianEstimate}
\pp_{t}\Kap-\Delta\Kap\ge-\frac{C_{M}\max\{t,1\}}{t^{2}}\Kap.
\eeq
Using \eqref{DistanceLaplacianEstimate} and \eqref{ResidualLaplacianEstimate} leads to
\[
\int_{M}\!{}\evep(u)[\pp_{t}\Kap-\Delta\Kap]
\ge-\frac{N\mu}{4t}\int_{B_{s}(y)}\!{}(d(x,y))^{2}\evep(u)\Kap
-\frac{C_{M}\max\{t,1\}}{t^{2}}
\int_{M\setminus{}B_{s}(y)}\!{}\evep(u)\Kap.
\]
Note that, since $d_{+}(\cdot,y)$ is a function of distance from $y$, we have
\begin{align*}
\frac{C_{M}\max\{t,1\}}{t^{2}}
\int_{M\setminus{}B_{s}(y)}\!{}\evep(u)\Kap
&\le\frac{C_{M}\max\{t,1\}e^{-\frac{s^{2}}{4t}}}{t^{2}(4\pi{}t)^{\frac{N}{2}}}
\int_{M\setminus{}B_{s}(y)}\!{}\evep(u)\\
&\le{}C_{M}'e^{-\frac{s^{2}}{8t}}\int_{M}\!{}\evep(u)\\
&\le{}C_{M}'\int_{M}\!{}\evep(u).
\end{align*}
Note that if we rescale the metric by dividing by $0<a\le1$ then the constant $C_{M}'$
only becomes smaller.
Observe that we either have $t^{\frac{1}{4}}\ge{}s$ or $0<t^{\frac{1}{4}}<s$.
If $t^{\frac{1}{4}}\ge{}s$ then
\[
-\frac{N\mu}{4t}\int_{B_{s}(y)}\!{}(d(x,y))^{2}\evep(u)\Kap
\ge-\frac{N\mu}{4t^{\frac{1}{2}}}
\int_{B_{s}(y)}\!{}\evep(u)\Kap.
\]
If $0<t^{\frac{1}{4}}<s$ then we have, using the notation
$A_{t^{\frac{1}{4}},s}(y)\ceq{}
B_{s}(y)\setminus{}B_{t^{\frac{1}{4}}}(y)$ for $y\in{}M$, that
\begin{align*}
&-\frac{N\mu}{4t}\int_{B_{s}(y)}\!{}(d(x,y))^{2}\evep(u)\Kap\\
&=-\frac{N\mu}{4t}\int_{B_{t^{\frac{1}{4}}}(y)}\!{}(d(x,y))^{2}\evep(u)\Kap
-\frac{N\mu}{4t}\int_{A_{t^{\frac{1}{4}},s}(y)}
\!{}(d(x,y))^{2}\evep(u)\Kap\\
&\ge-\frac{N\mu}{4t^{\frac{1}{2}}}\int_{B_{t^{\frac{1}{4}}}(y)}\!{}\evep(u)\Kap
-\frac{N\mu\left(\inj(M)\right)^{2}}{16t}
\int_{A_{t^{\frac{1}{4}},s}(y)}\!{}
\evep(u)\cdot\frac{e^{\frac{-(d(x,y))^{2}}{4t}}}{(4\pi{}t)^{\frac{N}{2}}}\\
&\ge-\frac{N\mu}{4t^{\frac{1}{2}}}\int_{B_{t^{\frac{1}{4}}}(y)}\!{}\evep(u)\Kap
-\frac{N\mu\left(\inj(M)\right)^{2}}{16}
\sup_{t>0}\left\{\frac{e^{\frac{-1}{8t^{\frac{1}{2}}}}}{t(4\pi{}t)^{\frac{N}{2}}}\right\}
\cdot{}e^{\frac{-1}{8t^{\frac{1}{2}}}}\int_{A_{t^{\frac{1}{4}},s}(y)}\!{}
\evep(u)\\
&\ge-\frac{N\mu}{4t^{\frac{1}{2}}}\int_{M}\!{}\evep(u)\Kap
-C_{M}''e^{\frac{-1}{8t^{\frac{1}{2}}}}\int_{M}\!{}\evep(u)\\
&\ge-\frac{N\mu}{4t^{\frac{1}{2}}}\int_{M}\!{}\evep(u)\Kap
-C_{M}''\int_{M}\!{}\evep(u).
\end{align*}
Notice that $C_{M}''$ is invariant under rescaling in the metric and $\mu$ only becomes smaller if we divide the metric by
$a$ for $0<a<1$.
Putting this altogether gives
\begin{align*}
\int_{M}\!{}\evep(u)[\pp_{t}\Kap-\Delta\Kap]
&\ge\frac{-N\mu}{4t^{\frac{1}{2}}}\int_{M}\!{}\evep(u)\Kap\\
&-2\max\{C_{M}',C_{M}''\}
\int_{M}\!{}\evep(u).
\end{align*}
Setting
\[
c_{0}\ceq2\max\{C_{M}',C_{M}''\}
\]
and multiplying by $t$ gives the desired result.
Observe that a similar proof holds for \eqref{HeatErrorTermUpperBound} and that
\eqref{HeatErrorTermDistanceBound} is demonstrated through the proof of the upper bound.
\end{proof}

We next record estimates of a similar character for $\Phi$.

\begin{lemma}\label{ApproximateHeatKernelMatrixHarnackPrinciple}
Suppose $(M,g)$ is an $N$-dimensional compact Riemannian manifold without boundary.
Let $\Kap$ be the approximate heat kernel from \eqref{AHK}.
Then there is $c_{2}>0$ such that for all $0<t<T$ that
\beq\label{WeightedEnergyDerivative}
\int_{M}\!{}\Phi(u,(y,T))(x,T-t)
\ge\frac{-|\lambda|t^{\frac{1}{2}}}{3}
\int_{M}\!{}\evep(u)\Kap
-c_{2}t\int_{M}\!{}\evep(u)
\eeq
where the constants remain bounded when dividing the metric by $0<a\le1$ and
where we have used the abbreviations $\Kap$ for $\Kap(x,t;y)$ and $u$ for
$u(x,T-t)$.
Similarly, there is $c_{3}>0$ such that for all $0<t<T$ that
\beq\label{HessianTermUpperBound}
\int_{M}\!{}\Phi(u,(y,T))(x,T-t)
\le\frac{\mu{}t^{\frac{1}{2}}}{2}
\int_{M}\!{}\evep(u)\Kap+c_{3}t\int_{M}\!{}\evep(u)
\eeq
where the constants remain bounded when dividing the metric by $0<a\le1$.
It is worth noting that we have
\beq\label{HessianTermDistanceBound}
\int_{M}\!{}\Phi(u,T)(x,T-t)
\le\mu\int_{M}\!{}\frac{(d_{+}(x,y))^{2}}{4}\evep(u)\Kap
+D_{M}\int_{M}\!{}\evep(u)\Kap
\eeq
where $D_{M}$ remains bounded when dividing the metric by $0<a\le1$.
\end{lemma}
\begin{proof}
The proof is similar to that of Lemma \ref{HeatKernelTermEstimate}.
More discussion is provided in A.3.1.2 of \cite{Col}.
\end{proof}

We now prove a monotonicity formula for solutions to
\ref{PGLOriginal}.
As noted before, this result will be instrumental to demonstrating
many of the estimates needed in the proof of Theorem
\ref{Theorem1}.

\begin{proposition}\label{MonotonicityFormula}
Let $\Kap$ be the approximate heat kernel and suppose that $y\in{}M$ and $T>0$.
Then there exists positive constants $C_{1}\ge1$ and $C_{2}$ such that if
$0\le{}R_{1}\le{}R_{2}\le\min\bigl\{\sqrt{T},1\bigr\}$ then
\beq\label{AlmostMonotonicity}
C_{1}E_{0}R_{1}+\exp[C_{2}R_{1}]Z(R_{1})
\le{}C_{1}E_{0}R_{2}+\exp[C_{2}R_{2}]Z(R_{2})
\eeq
where
\beq\label{E0}
E_{0}\ceq{}\int_{M}\!{}\evep(\uvep^{0}(x))\dvol_{g}(x).
\eeq
That is, the function $r\mapsto{}C_{1}E_{0}r+\exp[C_{2}r]Z(r)$
is non-decreasing on $\bigl[0,\min\bigl\{\sqrt{T},1\bigr\}\bigr]$.
\end{proposition}

\begin{proof}
Combining \eqref{HeatErrorTermLowerBound} and
\eqref{WeightedEnergyDerivative} for $u=\uvep$
with the expression for $Z'(R)$ from
Lemma \ref{MonotonicityFormulaApproximateHeatKernel}
gives an inequality of the form
\[
Z'(R)\ge{}-\tilde{C}Z(R)-\tilde{D}E_{0}
\]
where $\tilde{C}$ and $\tilde{D}$ are positive constants that
remain bounded when dividing the metric by $0<a\le1$.
Setting $C_{2}\ceq\tilde{C}$ and
$C_{1}\ceq\tilde{D}e^{\tilde{C}}$ as well as using that $R\le1$
leads to \eqref{AlmostMonotonicity}.
More details are provided in A.3.1.3 of \cite{Col}.
\end{proof}

\begin{remark}
As one might guess from the appeal to the Hessian Comparison
Theorem, the constants $C_{1}$ and $C_{2}$ from the above
proposition can all be estimated in terms of
upper and lower bounds on the sectional curvature.
As a result, all such constants are preserved by dividing the
metric $g$ by factors smaller than one.
As noted in the introduction, this is generally the case for all curvature-dependent constants appearing in this paper.
\end{remark}

The next result facilitates comparison of the weighted energy
centred about two different points in space-time.

\begin{lemma}\label{WeightedEnergyComparison}
Let $0<t_{*}<T$, and $z_{*}=(x_{*},t_{*})\in{}M\times(0,\infty)$.
Then,
\[
\Evep(z_{*},\sqrt{t_{*}})
\le\biggl(\frac{T}{t_{*}}\biggr)^{\frac{N}{2}-1}
\exp\biggl[\frac{C_{f}(d_{+}(x_{T},x_{*}))^{2}}{T-t_{*}}\biggr]\Evep\bigl(\uvep,(x_{T},T),\sqrt{T}\bigr)
\]
for all $x_{T}\in{}M$ where $C_{f}\ceq\max\bigl\{1,\left\|f'\right\|_{L^{\infty}([0,\infty))}^{2}\bigr\}$ and $f$ is as defined below \eqref{dplus}.
In particular,
\[
\Evep(z_{*},\sqrt{t_{*}})
\le\biggl(\frac{T}{t_{*}}\biggr)^{\frac{N}{2}-1}
\exp\biggl[\frac{4C_{f}(d(x_{T},x_{*}))^{2}}{T-t_{*}}\biggr]\Evep\bigl(\uvep,(x_{T},T),\sqrt{T}\bigr)
\]
for all $x\in{}M$ where $C_{f}$ is as above.
\end{lemma}

\begin{proof}
The proof proceeds in the same way as the proof of Lemma $2.3$ of \cite{BOS2} except a careful estimate of the supremum of the function
\[
x\mapsto\exp\biggl(\frac{(d_{+}(x,x_{T}))^{2}}{4T}-\frac{(d_{+}(x,x_{*}))^{2}}{4t_{*}}\biggr)
\]
is required.
The corresponding estimate in \cite{BOS2} is done completely explicitly.
Here it is carried out by considering several cases, depending on the relative size of $d_{g}(x,x_{T})$, $d_{g}(x,x_{*})$, and $\inj_{g}(M)$.
Details can be found in A.3.2.1 of \cite{Col}.
\end{proof}

The next proposition is an important localization method that
converts information about the energy density on a small ball
to information about the weighted energy.
This will be helpful when analyzing the structure of
the energy density measure in the proof of Theorem
\ref{BOSTheorem}.

\begin{proposition}\label{ManifoldEnergyLocalizedToBall}
Suppose $T>0$ and $\sqrt{2\vep}<R<1$.
Then for any $\lambda>0$ and $x_{T}\in{}M$ the following inequality holds
\begin{align*}
\int_{M}\!{}&\evep(\uvep(\cdot,T))e^{-\frac{(d_{+}(\cdot,x_{T}))^{2}}{4R^{2}}}
\le{}\int_{B_{\lambda{}R}(x_{T})}\!{}\evep(\uvep(\cdot,T))\\
&+M_{0}e^{-\frac{c_{*}^{2}\lambda^{2}}{8}}
\biggl[e^{C_{2}}\biggl(\frac{2R^{2}}{T+2R^{2}}\biggr)^{\frac{N-2}{2}}
+C_{1}(4\pi)^{\frac{N}{2}}(\sqrt{2}R)^{N-2}\sqrt{T}\biggr]\mathopen{}\left|\log(\vep)\right|\mathclose{}.
\end{align*}
\end{proposition}

\begin{proof}
The proof is essentially the same as that of Proposition $2.3$ of \cite{BOS2}.
The point is to estimate
\[
\int_{M\setminus{}B_{\lambda{}R}(x_{T})}\!{}\evep(\uvep)e^{-\frac{(d_{+}(x,x_{T}))^{2}}{4R^{2}}}.
\]
This is achieved by applying the monotonicity formula, see Proposition \ref{MonotonicityFormula}, in addition to the properties of $d_{+}$, as in \eqref{dInequality}.
See A.3.5.1 of \cite{Col} for details.
\end{proof}

In the next proposition we exploit the monotonicity formula to 
obtain good estimates of the solution of a nonhomogeneous heat equation
when the right-hand side is dominated by $\Vep(\uvep)$.

\begin{proposition}\label{HeatEstimate}
If $0<T<1$, $x_{T}\in{}M$, and
$\omega\colon{}M\times(0,\infty)\to\wedge^{2}M$ solves
\[
\begin{cases}
\pp_{t}\omega-\Delta{}\omega=h&\text{on }M\times(0,\infty)\\
\omega(x,0)=0&x\in{}M
\end{cases}
\]
where $h\in{}L^{\infty}(M\times[0,T];\wedge^{2}M)$ satisfies
\beq\label{PotentialUpperBound}
|h(x,t)|\le{}\Vep(\uvep(x,t))\hspace{10pt}\text{for }(x,t)\in{}M\times[0,T]
\eeq
then for any $z=(x,t)\in{}M\times[0,T]$, the following estimate holds:
\begin{align}
|\omega(z)|&\le{}C_{3}(T+1)
\biggl(\frac{T}{t}\biggr)^{\frac{N}{2}-1}e^{\frac{C_{f}(d_{+}(x_{T},x_{*}))^{2}}{T-t}}
\Bigl(\Evep\bigl(\uvep,(x_{T},T),\sqrt{T}\bigr)+C_{1}E_{0}T\Bigr)\\
&\le{}C_{3}(T+1)
\biggl(\frac{T}{t}\biggr)^{\frac{N}{2}-1}e^{\frac{4C_{f}(d(x_{T},x_{*}))^{2}}{T-t}}
\Bigl(\Evep\bigl(\uvep,(x_{T},T),\sqrt{T}\bigr)+C_{1}E_{0}T\Bigr)  \nonumber
\end{align}
where $C_{f}$ is as in Lemma \ref{WeightedEnergyComparison} and $C_{3}$ depends on $M$.
\end{proposition}

\begin{proof}
The proof proceeds in the same way as Proposition $2.2$ of \cite{BOS2}, the idea being to represent $\omega(z)$ by Duhamel's formula and then exploit
the fact that, since $|f|\le{}\Vep(\uvep)$, the right-hand side of Duhamel's formula is controlled by the weighted energy.
In our setting we must estimate the heat kernel for $2$-forms, appearing in Duhamel's formula, by the approximate heat kernel $\Kap$, appearing in $\Evep$.
This may be done using estimates on the heat kernel for differential forms which are provided in \cite{Lud}.
Details can be found in A.3.3.1 and Proposition A.3.2 of \cite{Col}.
\end{proof}
The next proposition is a localization method that originated from
\cite{LiRi} and was used in \cite{BOS2}.
As in \cite{BOS2} this result is vital to our proof of Theorem
\ref{Theorem1} as it permits us to localize our estimate of
the weighted energy to a small coordinate ball.
It is based on a Pohozaev type inequality.

\begin{proposition}\label{PohozaevEnergyIdentity}
Let $(M,g)$ be an $N$-dimensional compact Riemannian manifold without boundary and suppose $0<t<T$ is chosen so that $(T-t)$ is
small enough that
\[
1-C_{5}(T-t)\ge\frac{1}{2}
\]
where $C_{5}>0$ depends linearly on the sectional curvature of $M$.
Then, there is a constant $C_{6}>0$ invariant under dilations of the metric $g$
by factors larger than one and $D_{f}>0$ dependent only on $f$ such that if $z_{T}=(x_{T},T)\in{}M\times(0,\infty)$
then
\begin{align}\label{PohozaevLocalizationFirstInequalityProof}
\int_{M\times\{t\}}\!{}\evep&(\uvep)\frac{(d_{+}(x,x_{T}))^{2}}{4(T-t)}
e^{-\frac{(d_{+}(x,x_{T}))^{2}}{4(T-t)}}\\
&\le(4\pi)^{\frac{N}{2}}C_{6}(T-t)^{\frac{N-2}{2}}\Evep\bigl(z_{T},\sqrt{T-t}\bigr)
+2D_{f}C_{0}(4\pi)^{\frac{N}{2}}(T-t)^{\frac{N}{2}+1}E_{0}
\nonumber\\
&+2D_{f}[4\pi(T-t)]^{\frac{N}{2}}
\int_{M\times\{t\}}\!{}[\Vep(\uvep)+\Xi(\uvep,z_{T})]
\Kap\bigl(x,\sqrt{T-t};x_{T}\bigr)  \nonumber
\end{align}
and consequently
\begin{align}
\int_{M\times\{t\}}\!{}&\evep(\uvep)e^{-\frac{(d_{+}(x,x_{T}))^{2}}{4(T-t)}}
\le2\int_{A\times\{t\}}\!{}\evep(\uvep)e^{-\frac{(d_{+}(x,x_{T}))^{2}}{4(T-t)}}
+\frac{(4\pi)^{\frac{N}{2}}C_{0}(T-t)^{\frac{N}{2}+1}}{N}E_{0}
\label{DistSqWeight:Bound} \\
&+\frac{2[4\pi(T-t)]^{\frac{N}{2}}}{N}
\int_{M\times\{t\}}\!{}[\Vep(\uvep)+\Xi(\uvep,z_{T})]\Kap(x,T-t;x_{T})  \nonumber
\end{align}
where
\[
A\ceq\biggl\{x\in{}M:\frac{(d_{+}(x,x_{T}))^{2}}{8(T-t)}\le{}C_{6}\biggr\}.
\]
\end{proposition}

\begin{proof}
The proof is mostly similar to the one found in Proposition $2.4$
of \cite{BOS2}, the only exception being we replace usage of the
distance function $d$ with $d_{+}$ and use properties relating
to the definition of $d_{+}$.
We refer the reader to Lemma A.3.4 and A.3.5.2 of \cite{Col} for
more details.
First, for $0<T_{1}\le{}T_{2}<T$ and $x_{T}\in{}M$, we establish the inequality
\begin{align*}
\int_{T_{1}}^{T_{2}}\!\!\!{}\int_{M}&\!{}\frac{(d_{+}(x,x_{T}))^{2}}{4(T-t)}
\evep(\uvep)e^{-\frac{(d_{+}(x,x_{T}))^{2}}{4(T-t)}}\\
&\le[4\pi(T-T_{1})]^{\frac{N}{2}}D_{f}\Evep\bigl(z_{T},\sqrt{T-T_{1}}\bigr)
-[4\pi(T-T_{2})]^{\frac{N}{2}}D_{f}\Evep\bigl(z_{T},\sqrt{T-T_{2}}\bigr)
\end{align*}
where
\[
D_{f}\ceq
\bigl(2-\left\|f'\right\|_{L^{\infty}(\R)}^{2}\bigr)^{-1}.
\]
To do this, we take the dot product of \ref{PGLOriginal} with
$2(T-t)\pp_{t}\uvep{}e^{-\frac{(d_{+}(x,x_{T}))^{2}}{4(T-t)}}$,
integrate by parts in time, and apply elementary inequalities.
The only difference from Lemma $2.6$ of \cite{BOS2}
involves using properties of $d_{+}$, such as that
\[
|\nabla{}r_{+}(x)|\le\left\|f'\right\|_{L^{\infty}(\R)}d_{+}(x,y)
\]
where $r_{+}$ is as in \eqref{rplus} and $y\in{}M$.\\

Then, setting $T_{1}=t$, letting $T_{2}\searrow{}t$, and using
\eqref{TimeIntegratedMonotonicity} leads to
\begin{align}\label{DistanceWeightedInequality}
\int_{M\times\{t\}}\!{}\evep(\uvep)&\frac{(d_{+}(x,x_{T}))^{2}}{4(T-t)}
e^{-\frac{(d_{+}(x,x_{T}))^{2}}{4(T-t)}}
\le\frac{ND_{f}(4\pi)^{\frac{N}{2}}}{2}(T-t)^{\frac{N}{2}-1}\widetilde{E}_{\vep}(z_{T},\sqrt{T-t})\\
&+(4\pi)^{\frac{N}{2}}D_{f}(T-t)^{\frac{N}{2}}\int_{M\times\{t\}}\!{}[\Vep(\uvep)+\Xi(\uvep,z_{T})]
\Kap(x,T-t;x_{T})
\nonumber\\
&+(4\pi)^{\frac{N}{2}}D_{f}(T-t)^{\frac{N}{2}}\int_{M\times\{t\}}\!{}[\Phi(\uvep,z_{T})(x,t)
+\Psi(\uvep,z_{T})(x,t)]\dvol(x). \nonumber{}
\end{align}

Combining \eqref{HeatErrorTermDistanceBound} and
\eqref{HessianTermUpperBound} with
\eqref{DistanceWeightedInequality}
leads to
\begin{align*}
\Bigl[1-(4\pi)^{\frac{N}{2}}&D_{f}\bigl[\mu-\frac{2N\lambda}{3}\bigr](T-t)\Bigr]
\int_{M\times\{t\}}\!{}\evep(\uvep)\frac{(d_{+}(x,x_{T}))^{2}}{4(T-t)}
e^{-\frac{(d_{+}(x,x_{T}))^{2}}{4(T-t)}}\\
&\le(4\pi)^{\frac{N}{2}}D_{f}[N+C_{M}+D_{M}]
(T-t)^{\frac{N}{2}-1}\widetilde{E}_{\vep}(z_{T},\sqrt{T-t})\\
&+(4\pi)^{\frac{N}{2}}D_{f}(T-t)^{\frac{N}{2}}
\int_{M\times\{t\}}\!{}[\Vep(\uvep)+\Xi(\uvep,z_{T})]
\Kap(x,T-t;x_{T})\\
&+C_{0}(4\pi)^{\frac{N}{2}}D_{f}(T-t)^{\frac{N}{2}+1}E_{0}.
\end{align*}
Defining
$C_{5}\ceq
(4\pi)^{\frac{N}{2}}D_{f}\bigl[\mu-\frac{2N\lambda}{3}\bigr]$,
using that we assume $1-C_{5}(T-t)\ge\frac{1}{2}$, and
rearranging gives
\eqref{PohozaevLocalizationFirstInequalityProof}.
Setting $C_{6}\ceq{}2D_{f}[N+C_{M}+D_{M}]$ and arguing as in \cite{BOS2} then
gives \eqref{DistSqWeight:Bound}.
\end{proof}

The following permits us to find a good bound for
the weighted energy on a scale
$R_{1}$ by the weighted energy on a smaller scale $\delta_{0}R_{1}$.

\begin{proposition}\label{Averaging}
Fix $0<\delta_{0}<\frac{1}{16}$, $T>0$, and let $0<R\le\min\bigl\{\sqrt{T},1\bigr\}$.
There exists a constant $\vep_{1}>0$ depending only $R$, $T$, and $\delta_{0}$, such that, for $0<\vep\le\vep_{1}$,
there exists $R_{1}>0$, satisfying
\[
R_{1}\in\bigl(\sqrt{\vep},R\bigr)
\]
such that
\begin{align}\label{MontonicityAveragingInequality}
\{C_{7}E_{0}R_{1}+Z(R_{1})\}&-\{C_{7}E_{0}(\delta_{0}R_{1})+Z(\delta_{0}R_{1})\}\\
&\le\frac{4C_{7}e^{C_{8}}\mathopen{}\left|\log(\delta_{0})\right|\mathclose{}}{\mathopen{}\left|\log(\vep)\right|\mathclose{}}[RE_{0}+Z(R)]  \nonumber
\end{align}
where $C_{7}\ceq{}C_{1}e^{2C_{2}}$ and $C_{8}\ceq2C_{2}$.
We also have
\begin{align}\label{SmallEnergyAveraging}
\int_{T-R_{1}^{2}}^{T-(\delta_{0}R_{1})^{2}}\!\!\!\int_{M}\!{}
\mathopen{}\left[(\Xi(\uvep,z_{T}))(x,t)\right.\mathclose{}&\mathopen{}\left.+\Vep(\uvep(x,t))\right]\mathclose{}\Kap(x,T-t;x_{T})\dvol_{g}(x)\mrmd{}t\\
\le&\frac{4C_{7}e^{C_{8}}\mathopen{}\left|\log(\delta_{0})\right|\mathclose{}}{\mathopen{}\left|\log(\vep)\right|\mathclose{}}
[RE_{0}+Z(R)].  \nonumber
\end{align}
\end{proposition}

\begin{proof}
The proof is essentially the same as in Proposition $2.6$ of
\cite{BOS2}.
The idea in proving \eqref{MontonicityAveragingInequality}
is to average increments of
$r\mapsto{}C_{7}e^{C_{8}}E_{0}r+e^{C_{8}r}Z(r)$
over time intervals
$\bigl[\delta_{0}^{j}R,\delta_{0}^{j-1}R\bigr]$
for $j=2,3,\ldots,k_{0}$ where
\[
k_{0}\approx
\biggl\lfloor\frac{\mathopen{}\left|\log(\vep)\right|\mathclose{}}
{2\mathopen{}\left|\log(\delta_{0})\right|\mathclose{}}
\biggr\rfloor
\]
and to find an interval,
$\bigl[\delta_{0}^{k_{1}}R,\delta_{0}^{k_{1}-1}R\bigr]$,
for which the increment is small.
This is achieved by repeatedly making use of Proposition
\ref{MonotonicityFormula}.
The inequality \eqref{SmallEnergyAveraging} then follows from
\eqref{TimeIntegratedMonotonicity}, additional estimates on the
error terms $\Phi$ and $\Psi$ in terms of the weighted energy
and the initial energy due to \eqref{HeatErrorTermLowerBound},
\eqref{WeightedEnergyDerivative},
as well as our choice of constants $C_{7}$ and $C_{8}$.
For more details we refer the reader to A.3.6.1 of \cite{Col}.
\end{proof}

\section{Overview of the proof of Theorem \ref{Theorem1}}
\label{Sec:Outline}

\hspace{15pt}We start by presenting a detailed outline of the proof of Theorem \ref{Theorem1}.
We closely follow the proof presented in \cite{BOS2}, so much so that this section may be used as a reader's guide to the arguments found there.
As we proceed, we will distinguish between
\begin{enumerate}
\item
arguments that can be adapted from the Euclidean to the Riemannian setting with only cosmetic changes; we will describe these but not present them in detail;
and
\item
places where more effort is needed in order to adjust earlier arguments to the present setting.
These points will be discussed at greater length in Section
\ref{Sec::SteppingStones}.
\end{enumerate}
Note that, unless otherwise specified, all metric related
quantities will be associated to $\gR$ and the metric will be
suppressed from the notation.\\

\begin{flushleft}
{\it Reduction via rescaling:}\\
\end{flushleft}
\par{}Throughout the proof of the theorem, $0<\delta_{0}<\frac{1}{16}$
will denote a fixed parameter whose precise value will not
be specified until a late stage of the proof.
Applying Proposition \ref{Averaging} with this choice of $\delta_{0}$
and $R$, $T$ as in the statement of Theorem \ref{Theorem1}
we find a suitable
time interval, $[\delta_{0}R_{1},R_{1}]$, in which the weighted
energy of $\uvep$ satisfies
\eqref{MontonicityAveragingInequality} and
\eqref{SmallEnergyAveraging}.
Next we define $\veep\colon{}M\times(0,\infty)\to\C$,
a rescaling of $\uvep$ in
$\vep$ and in time, by
\beq\label{vep}
\veep(x,t)\ceq\uvep\bigl(x,T+R_{1}^{2}[t-1]\bigr)
\eeq
where $\ep\ceq\frac{\vep}{R_{1}^{2}}$.
We also introduce the rescaled metric
\beq\label{RescaledMetric}
\gR\ceq\frac{g}{R_{1}^{2}}.
\eeq
It follows from standard parabolic estimates that there is
$K>0$ such that for $x\in{}M$ and $t>0$ that
\beq\label{vepEstimates}
|\veep(x,t)|\le3,\hspace{15pt}
|\nabla\veep(x,t)|\le\frac{K}{\ep},\hspace{15pt}
|\pp_{t}\veep(x,t)|\le\frac{K}{\ep^{2}}.
\eeq
Rescaling \eqref{MontonicityAveragingInequality} to be written in
terms of $\veep$ as well as applying \eqref{H0} and the
assumptions of Theorem \ref{Theorem1} we obtain
\begin{align}
\bigl\{\widetilde{E}_{\ep,\gR}(\veep,(x_{T},1),1)+C_{7}E_{0}R_{1}\bigr\}
&-\bigl\{\widetilde{E}_{\ep,\gR}(\veep,(x_{T},1),\delta_{0})
+C_{7}E_{0}(\delta_{0}R_{1})\bigr\}  \label{RescaledAveragingInequality}\\
&\le4C_{7}e^{C_{8}}\mathopen{}\left|\log(\delta_{0})\right|\mathclose{}
[RM_{0}+\eta]. \nonumber
\end{align}
Finally, a change of variables applied to
\eqref{SmallEnergyAveraging} in addition to an application \eqref{H0} and the assumptions of Theorem \ref{Theorem1} leads to
\beq\label{SmallEnergyAveragingv}
\int_{M\times[0,1-\delta_{0}^{2}]}\!{}
\mathopen{}\left[V_{\ep}(v_{\ep})+
\Xi(v_{\ep},(x_{T},1))\right]\mathclose{}
K_{ap,\gR}(x,1-t;x_{T})
\le4C_{7}e^{C_{8}}\mathopen{}\left|\log(\delta_{0})\right|\mathclose{}
[RM_{0}+\eta].
\eeq
From here Theorem \ref{Theorem1} is reduced to demonstrating the
following result.

\begin{proposition}\label{Theorem1Reduction}
Let $T>0$ and $x_{T}\in{}M$.
Then there exists constants $0<\delta_{0}<\frac{1}{16}$, $0<\ep_{0}<\frac{1}{2}\min\{1,\inj_{g}(M)\}$, $0<R_{0}<1$,
and $\eta_{0}>0$ such that for $0<\eta\le\eta_{0}$,
$0<\ep<\ep_{0}$, and
$0<R<\min\bigl\{\sqrt{T},R_{0}\bigr\}$ the following inequality holds:
\begin{equation}
\Eepgr(\veep,(x_{T},1),\delta_{0})
\le\frac{1}{4}
\Bigl(\Eepgr(\veep,(x_{T},1),1)+C_{7}E_{0}R_{1}\Bigr)
+\mathcal{R}(\eta,R)
\end{equation}
where $\mathcal{R}(\eta,R)$ tends to zero as $\eta,R\to0^{+}$
and $R_{1}$ is as in Proposition \ref{Averaging}.
\end{proposition}
The proof of Theorem \ref{Theorem1}
using Proposition \ref{Theorem1Reduction} proceeds with only
minor differences to the argument in the Euclidean
setting from \cite{BOS2}.
We refer the reader to Section \ref{Sec::SteppingStones} for
additional details.
We then reduce proving Proposition \ref{Theorem1Reduction}
to demonstrating the existence of
$0<\delta_{0}<\frac{1}{16}$ for which there is
some $\delta\in[\delta_{0},2\delta_{0}]$ for which
\beq\label{Theorem1FurtherReduction}
\widetilde{E}_{\ep,\gR}(\veep,(x_{T},1),\delta)
\le\frac{e^{-C_{2}}}{8}
\Bigl(\widetilde{E}_{\ep,\gR}(\veep,(x_{T},1),1)
+C_{7}E_{0}R_{1}\Bigr)+\mathcal{R}(\eta,R)
\eeq
where $\mathcal{R}(\eta,R)\to0^{+}$ as $\eta,R\to0^{+}$.
The argument reducing the proof of
Proposition \ref{Theorem1Reduction} to
\eqref{Theorem1FurtherReduction} is essentially the same as in
\cite{BOS2} and is sketched in Subsection \ref{Subsec:Reduction}.\\

\begin{flushleft}
{\it Preliminary choice of good time slice:}\\
\end{flushleft}
\par{}By Chebyshev's inequality applied to \eqref{SmallEnergyAveragingv}
in time over the interval
$[1-4\delta_{0}^{2},1-\delta_{0}^{2}]$ we see that there are a large number of time
slices $t=1-\delta^{2}$, where $\delta\in[\delta_{0},2\delta_{0}]$,
for which
\begin{align}
\int_{M}\!{}V_{\ep}(v_{\ep})
K_{ap,\gR}(x,1-t;x_{T})\le{}C(\delta_{0})
[RM_{0}+\eta]    \label{WeightedVEnergy}\\
\int_{M}\!{}\Xi(v_{\ep},(x_{T},1))
K_{ap,\gR}(x,1-t;x_{T})\le{}C(\delta_{0})
[RM_{0}+\eta].     \label{WeightedXiEnergy}
\end{align}
The inequalities \eqref{WeightedVEnergy} and \eqref{WeightedXiEnergy}
will be used to determine $\mathcal{R}(\eta,R)$ from
\eqref{Theorem1FurtherReduction} as well as obtain the
coefficient of
$\widetilde{E}_{\ep,\gR}(\veep,(x_{T},1),1)+C_{7}E_{0}R_{1}$
from \eqref{Theorem1FurtherReduction}.
We use the notation $\Theta_{1}$ to denote
\beq\label{Def:Theta1}
\Theta_{1}\ceq
\Bigl\{t\in\bigl[1-4\delta_{0}^{2},1-\delta_{0}^{2}\bigr]
:\eqref{WeightedVEnergy}\text{ and }\eqref{WeightedXiEnergy}
\text{ both hold at }t\Bigr\}.
\eeq
In particular, we provide a more explicit estimate on the number of slices in
Lemma A.4.1 of \cite{Col}.
The strategy in proving \eqref{Theorem1FurtherReduction} will be to
decompose $\widetilde{E}_{\ep,\gR}(\veep,(x_{T},1),\delta)$ into suitable
components and estimate the resultant terms using PDE techniques
by showing that the data can be controlled by
\eqref{WeightedVEnergy} and \eqref{WeightedXiEnergy}.\\

\begin{flushleft}
{\it Localization and decomposition:}\\
\end{flushleft}
\par{}We make use of Proposition \ref{PohozaevEnergyIdentity}, applied through $\uvep$, in
combination with \eqref{WeightedVEnergy} and \eqref{WeightedXiEnergy}
to obtain
\begin{align}
\Eepgr(\veep,(x_{T},1),\delta)
&\le\frac{2}{(4\pi)^{\frac{N}{2}}\delta^{N-2}}
\int_{B_{\delta\sqrt{8C_{6}}}(x_{T})}\!{}e_{\ep}(\veep)\dvol
\label{LocalizedEnergyError}\\
&+K_{M}\delta_{0}^{4}R^{3}\bigl[\Eepgr(\veep,(x_{T},1),1)+C_{7}E_{0}R_{1}\bigr]
+C(\delta_{0})[RM_{0}+\eta].    \nonumber
\end{align}
Since $C(\delta_{0})[RM_{0}+\eta]$ can be included in
$\mathcal{R}(\eta,R)$ and $\delta_{0}^{3}R^{4}$ can be chosen
suitably small it suffices to estimate
the remaining term from \eqref{LocalizedEnergyError}.
To do this we decompose $\evep(\uvep)$.
We first observe that there is a constant $K>0$ such that
\beq\label{CurrentDecomposition}
\evep(\veep)\le{}K\Bigl(|\veep\times\mrmd\veep|^{2}
+|\veep|^{2}\bigl|\nabla|\veep|\bigr|^{2}+\Vep(\veep)\Bigr).
\eeq
We further decompose $\veep\times\mrmd\veep$ by using a Hodge
de Rham decomposition.
To do this we first define $H(\omega)$ to be the harmonic part of a
$2$-form $\omega$.
Explicitly,
\beq\label{H:Def}
H(\omega)
\ceq\sum_{i=1}^{\beta_{2}(M)}\mathopen{}\left<\omega,
\gamma_{i,\gR}\right>_{L^{2}}\mathclose{}\gamma_{i,\gR}
\eeq
where $\{\gamma_{i,\gR}\}_{i=1}^{\beta_{2}(M)}$ is an
$L^{2}$-orthonormal basis for the space of harmonic $2$-forms on
$M$ in the metric $\gR$, obtained by rescaling an
$L^{2}$-orthonormal basis $\{\gamma_{i,g}\}_{i=1}^{\beta_{2}(M)}$
for the space of harmonic $2$-forms on $M$ in the metric $g$, and
$\beta_{2}(M)$ denotes the $2^{\text{nd}}$ Betti number of $M$.
We may then use a Hodge de Rham decomposition to find
$\vphi_{t}$, $\psi_{t}$, and $\xi_{t}$ satisfying
\begin{align}
\veep\times\mrmd\veep&=\mrmd\vphi_{t}+\mrmd^{*}\psi_{t}+\xi_{t}
\hspace{10pt}\text{on }B_{3r\slash2}(x_{T})\times\{t\} \label{LocalHodgeDecomposition}\\
\mrmd^{*}\xi_{t}&=0
\hspace{10pt}\text{on }B_{3r\slash2}(x_{T})\times\{t\}
\label{XiPDE1}\\
\mrmd\xi_{t}&=\mrmd^{*}\mrmd\psi_{t}
+H(\mrmd[\veep\times\mrmd\veep]\chi)
\hspace{10pt}\text{on }B_{3r\slash2}(x_{T})\times\{t\}
\label{XiPDE2}\\
-\Delta\psi_{t}
&=\mrmd[\veep\times\mrmd\veep]\chi
-H(\mrmd[\veep\times\mrmd\veep]\chi),
\hspace{10pt}\text{on }M\times\{t\}
\label{LocalizedCurrent}
\end{align}
where $r>0$ is chosen sufficiently small toward the end
of the argument and $\chi$ is a smooth cutoff function supported in
$B_{4r}(x_{T})$ which is identically $1$ on $B_{2r}(x_{T})$, $0\le\chi\le1$, and
$\left\|\nabla\chi\right\|_{L^{\infty}}\le\frac{2}{r}$.
To make the notation more compact we set
\beq\label{HarmonicPerp}
H^{\perp}(\omega)\ceq\omega-H(\omega)
\eeq
where $\omega$ is a $2$-form over $M\times\{t\}$.
We note that $H$ is a new consideration for the Hodge de
Rham decomposition that does not appear in the corresponding
identities from \cite{BOS2}.
This term arises because we impose no topological restrictions on
$M$.\\

From \eqref{LocalizedEnergyError},
\eqref{CurrentDecomposition}, and \eqref{LocalHodgeDecomposition}
it will suffice
to estimate each of the following:
\begin{align}
&\int_{B_{\delta\sqrt{8C_{6}}}(x_{T})}\!{}
\Bigl[|\veep|^{2}\bigl|\nabla|\veep|\bigr|^{2}
+V_{\ep}(\veep)\Bigr]    \label{ModulusComponent}\\
&\int_{B_{\delta\sqrt{8C_{6}}}(x_{T})}\!{}
|\mrmd\vphi_{t}|^{2}    \label{PsiComponent}\\
&\int_{B_{\delta\sqrt{8C_{6}}}(x_{T})}\!{}
\bigl\{|\mrmd\psi_{t}|^{2}+|\mrmd^{*}\psi_{t}|^{2}\bigr\}
\label{PhiComponent}\\
&\int_{B_{\delta\sqrt{8C_{6}}}(x_{T})}\!{}
|\xi_{t}|^{2}.    \label{XiComponent}
\end{align}
Since $H\bigl(\mrmd\bigl[\veep\times\mrmd\veep\bigr]\chi\bigr)$ is
present in both \eqref{XiPDE2} and \eqref{LocalizedCurrent} then to
achieve our goal we will also need to provide estimates related to $H$.
Below we outline the strategy for estimating
\eqref{ModulusComponent}--\eqref{XiComponent}.
Where necessary, we will provide additional details in Section
\ref{Sec::SteppingStones}.\\

\begin{flushleft}
{\it Modulus Estimate:}\\
\end{flushleft}
\par{}As in \cite{BOS2} the goal is to demonstrate that \eqref{ModulusComponent}
satisfies an estimate of the form
\beq\label{ModulusTermSmall}
\int_{B_{r}(x_{T})}\!{}\Bigl\{|\veep|^{2}\bigl|\nabla|\veep|\bigr|^{2}+V_{\ep}(\veep)\Bigr\}
\le{}C(\delta_{0},r)[RM_{0}+\eta]^{\frac{1}{2}}
\biggl[\widetilde{E}_{\ep,\gR}(\veep,(x_{T},1),\delta)+CRE_{0}+1\biggr].
\eeq
The proof of \eqref{ModulusTermSmall}
follows the same procedure as Section $3.5$ of \cite{BOS2}.
As such, we will describe it briefly and refer the
reader to Subsection A.4.2 of \cite{Col} additional details.
We first define $\sigma_{\ep}\ceq1-|\veep|^{2}$ and observe that $\sigma_{\ep}$ solves the PDE
\beq\label{ModulusPDE}
\pp_{t}\sigma_{\ep}-\Delta\sigma_{\ep}=2|\nabla\veep|^{2}-\frac{2}{\ep^{2}}\sigma_{\ep}(1-\sigma_{\ep}).
\eeq
By moving $\pp_{t}\sigma_{\ep}$ to the right-hand side we 
can treat this as a Poisson problem so that elliptic techniques
can be applied to obtain an interior estimate for $\nabla\sigma_{\ep}$.
We then apply various algebraic manipulations to estimate
the terms on the right-hand side of the Poisson problem
by quantities involving $\Xi$ and $\Vep$.
From there we make use of the assumption that $t\in\Theta_{1}$.\\

\begin{flushleft}
{\it Estimate of $\xi_{t}$:}\\
\end{flushleft}
\par{}As in \cite{BOS2} we use that $\xi_{t}$ solves
\eqref{XiPDE1} and \eqref{XiPDE2} in addition to
elliptic estimates, see Lemma $5.2$ of
\cite{Am1} and note the correction found in \cite{Am2}, to obtain
\beq\label{XiEstimate}
\left\|\xi_{t}\right\|_{L^{2}(B_{\frac{3r}{2}}(x_{T}))}
\le{}K\Bigl[\left\|\mrmd\psi_{t}\right\|_{L^{2}(B_{2r}(x_{T}))}
+\left\|H(\mrmd[\veep\times\mrmd\veep]\chi)\right\|_{L^{2}(M)}\Bigr].
\eeq
The argument for \eqref{XiEstimate} is the
same as Lemma $3.4$ of \cite{BOS2}.\\

\begin{flushleft}
{\it Estimate of $\vphi_{t}$:}\\
\end{flushleft}
\par{}As in \cite{BOS2} the goal in estimating $\vphi_{t}$ is to obtain an
inequality of the form
\begin{align}\label{vphiEstimate}
\int_{B_{\delta\sqrt{8C_{6}}}(x_{T})\times\{t\}}\!{}&|\nabla\vphi_{t}|^{2}e^{\frac{-(d(x,x_{T}))^{2}}{4\delta^{2}}}
\le{}\frac{K_{M}\delta_{0}^{N}}{r}\biggl[\Eepgr(\veep,(x_{T},1),1)+C_{7}R_{1}E_{0}\biggr]\\
&+C(\delta_{0},r)\biggl[(RM_{0}+\eta)+(RM_{0}+\eta)^{\frac{1}{2}}\delta^{\frac{N}{2}}\biggl[\Eepgr(\veep,(x_{T},1),1)+C_{7}R_{1}E_{0}\biggr]^{\frac{1}{2}}+R_{2}(t)\biggr]  \nonumber
\end{align}
where
\begin{align*}
R_{2}(t)&\ceq\int_{B_{\frac{3r}{2}}(x_{T})}\!{}\bigl(|\mrmd^{*}\psi_{t}|^{2}+|\xi_{t}|^{2}\bigr)\\
&+\biggl(\int_{B_{\frac{3r}{2}}(x_{T})}\!{}
\bigl(|\mrmd^{*}\psi_{t}|^{2}+|\xi_{t}|^{2}\bigr)\biggr)^{\frac{1}{2}}
\bigl(\widetilde{E}_{\ep,\gR}(v_{\ep},(x_{T},1),1)+C_{7}R_{1}E_{0}\bigr)^{\frac{1}{2}}.
\end{align*}
Much of the proof extends with little
change to the Riemannian setting with the exception of a
computation done in coordinates.
We present this new ingredient in Section \ref{subsec:vphit}
and provide a brief outline of the general argument below.\\

To achieve \eqref{vphiEstimate}, we introduce an elliptic PDE that $\vphi_{t}$ solves over
$B_{s}(x_{T})$ for $s\in\bigl[r,\frac{3r}{2}\bigr]$ where
$s$ will be carefully chosen to ensure good properties.
As shown in Section \ref{subsec:vphit} we have that for each $s\in\bigl[r,\frac{3r}{2}\bigr]$, $\vphi_{t}$ solves
\beq\label{vphiPDE}
\begin{cases}
L_\delta{}\vphi=h& \text{on }B_{s}(x_{T})\times\{t\}\\
\frac{\pp\vphi}{\pp{}r}=g&\text{on }\pp{}B_{s}(x_{T})\times\{t\},
\end{cases}
\eeq
where $L_\delta$ is defined,
using the abbreviation $\Kap\ceq{}K_{ap,\gR}(x;\delta^{2};x_{T})$, by
\[
L_{\delta}\ceq
\frac{-1}{\Kap}
\text{div}[\Kap\nabla],
\]
and
\begin{align*}
h&\ceq\veep\times\biggl(\frac{-\mathopen{}\left<\nabla\Kap,
\nabla\veep\right>\mathclose{}}{\Kap}-\pp_{t}\veep\biggr)
+\mathopen{}\left<\frac{\mrmd\Kap}{\Kap},
\mrmd^{*}\psi_{t}+\xi_{t}\right>\mathclose{}
\hspace{10pt}
\text{on }B_{s}(x_{T})\times\{t\}\\
g&\ceq\veep\times\frac{\pp\veep}{\pp{}r}-(\mrmd^{*}\psi_{t}+\xi_{t})_{N}
\hspace{10pt}
\pp{}B_{s}(x_{T})\times\{t\},
\end{align*}
where $\omega_{N}$ denotes the normal part of $\omega$ and we recall that
$t=1-\delta^2$ is assumed to be an element of
$\Theta_{1}$.
Next, we make use of elliptic estimates for \eqref{vphiPDE} to obtain
\begin{align}\label{WeigtedInequalityOnSolutio}
&\int_{B_{s}(x_{T})}\!{}|\nabla\vphi|^{2}
e^{\frac{-(d(x,x_{T}))^{2}}{4\delta^{2}}}\\
\le&C(\delta,r)\biggl[\int_{B_{s}(x_{T})}\!{}h^{2}
e^{-\frac{(d(x,x_{T}))^{2}}{4\delta^{2}}}
+\biggl(\int_{B_{s}(x_{T})}\!{}h^{2}
e^{\frac{-(d(x,x_{T}))^{2}}{4\delta^{2}}}\biggr)^{\frac{1}{2}}
\biggl(\int_{\pp{}B_{s}(x_{T})}\!{}g^{2}
e^{\frac{-(d(x,x_{T}))^{2}}{4\delta^{2}}}\biggr)^{\frac{1}{2}}\biggr]
\nonumber\\
+&K_{M}r\int_{\pp{}B_{s}(x_{T})}\!{}g^{2}
e^{\frac{-\left(d(x,x_{T})\right)^{2}}{4\delta^{2}}}.
\nonumber
\end{align}
Finally, to obtain \eqref{vphiEstimate} we estimate the data, $h$ and $g$,
in terms of $\Xi$, $V_{\ep}(\veep)$, $\mrmd^{*}\psi_{t}$, and $\xi_{t}$ with one
exception in which we obtain an estimate in terms of $\veep$.
In particular, in estimating $g$ an averaging process is used for $s\in\bigl[r,\frac{3r}{2}\bigr]$ to estimate integrals over
$\pp{}B_{s}(x_{T})$ in terms of integrals over
$B_{\frac{3r}{2}}(x_{T})\setminus{}B_{r}(x_{T})$.
We note that the described exception results in the first term on the
right-hand side of \eqref{vphiEstimate}.
This is the only term where $\delta$ will be needed to manufacture the leading
coefficient of \eqref{Theorem1FurtherReduction}.
We also note that the procedure for estimating $g$ is the same as in \cite{BOS2}
except an application of \eqref{HessianComparisonInequality} is needed.
We refer the reader to Section \ref{Sec::SteppingStones} for more details.\\

\begin{flushleft}
{\it Estimate of $\psi_{t}$:}\\
\end{flushleft}
\par{}Estimating $\psi_{t}$ is more involved and will be outlined
through a number of steps.
The goal of each step will be to successively decompose
$\psi_{t}$ into terms with more specific information
that can be utilized.
Unlike previous estimates, subterms of $\psi_{t}$ are not
always estimated by appealing to PDE techniques.
Instead we may make use of detailed pointwise information as well
as the the work of \cite{JS3}.\\

\begin{flushleft}
{\it Step $1$: Decomposition of $\psi_{t}$}\\
\end{flushleft}
\par{}Before proceeding with the decomposition we introduce some
notation as well as some useful pointwise estimates.
We introduce a real-valued function defined on
$M\times(0,\infty)$ in terms of $|\veep|$ so that
if $\tveep=\tau\veep$ then
\begin{align}
\bigl|1-\tau^{2}(x,t)\bigr|
&\le{}K\bigl|1-|\veep(x,t)|^{2}\bigr|
\label{TauPotentialUpperBound}\\
\tveep=\veep&\hspace{20pt}\text{if }\,|\veep|\le\frac{1}{4}
\label{EqualLocal}\\
|\tveep|=1&\hspace{20pt}\text{if }\,|\veep|\ge\frac{1}{2}.
\label{UnitModulus}
\end{align}
Next we decompose $\psi_{t}$, which solves
\eqref{LocalizedCurrent}, as $\psi_{t}=\psi_{1,t}+\psi_{2,t}$
where $\psi_{1,t}$ and $\psi_{2,t}$ solve
\begin{align}
-\Delta\psi_{1,t}
&=H^{\perp}(\mrmd[\tveep\times\mrmd\tveep]\chi)
\hspace{55pt}
\text{on }M\times\{t\}\label{psi1t:PDE}\\
-\Delta\psi_{2,t}
&=H^{\perp}(\mrmd[(1-\tau^{2})\veep\times\mrmd\veep]\chi)
\hspace{20pt}
\text{on }M\times\{t\}.
\end{align}

The smallness \eqref{TauPotentialUpperBound} of $1-\tau^2$
will aid in estimates of $\psi_{2,t}$. 
A key point in
estimates of $\psi_{1,t}$ is that 
\beq\label{PointwiseJacobian}
|\mrmd(\tveep\times\mrmd\tveep)|
\le{}K\frac{(1-|\veep|^{2})^{2}}{4\ep^{2}}=KV_{\ep}(\veep)
\hspace{10pt}
\text{on }M\times(0,\infty).
\eeq
To prove this, note that
\beq\label{jacobian.structure}
\mrmd[\tveep\times\mrmd\tveep ]
= \sum_{i\ne j} \frac{\pp\tveep}{\pp{}x_{i}}\times
\frac{\pp\tveep}{\pp{}x_{j}}\mrmd{}x^{i}\wedge\mrmd{}x^{j}.
\eeq
By \eqref{vepEstimates}, the right-hand side is always bounded by $K/\ep^2$ and vanishes when $|\tveep|=1$, that is, when $|\veep|\ge 1/2$, so
\eqref{PointwiseJacobian} follows from the definition of $V_{\ep}$. We also see from \eqref{jacobian.structure} that $ \mrmd(\tveep\times\mrmd\tveep)$ has a Jacobian structure which we will exploit to apply \cite{JS3}.\\

\begin{flushleft}
{\it Step $2$: Decomposition and estimate of $\psi_{2,t}$}\\
\end{flushleft}
\par{}We further decompose $\psi_{2,t}$ as
$\psi_{2,t}=\psi_{2,t}^{1}+\psi_{2,t}^{2}$ where
$\psi_{2,t}^{1}$, $\psi_{2,t}^{2}$ solve
\begin{align}
-\Delta\psi_{2,t}^{1}&=\mrmd[(1-\tau^{2})(\veep\times\mrmd\veep)\chi]
\hspace{47pt}\text{on }M\times\{t\}    \label{Psi2t1:PDE}\\
-\Delta\psi_{2,t}^{2}&=
H^{\perp}((1-\tau^{2})[\veep\times\mrmd\veep]\wedge\mrmd\chi)
\hspace{20pt}\text{on }M\times\{t\}   \label{Psi2t2:PDE}
\end{align}
where in \eqref{Psi2t1:PDE} we have used that $H^{\perp}$ is the
identity on exact forms.
The argument to estimate
$\psi_{2,t}^{1}$ and $\psi_{2,t}^{2}$ is similar in style
to that of Lemma $3.8$ of \cite{BOS2} though executed
differently.
In addition, the data of \eqref{Psi2t2:PDE} requires
additional estimates due to the harmonic projection term.
We provide more details in Section \ref{Sec::SteppingStones}.\\

Appealing to, among other things, elliptic regularity,
the pointwise estimates
\eqref{TauPotentialUpperBound}, \eqref{vepEstimates}, and that
$t\in\Theta_{1}$ we can estimate both terms in the decomposition
to find that
\beq\label{psi2t:Estimate}
\int_{M\times\{t\}}\!{}
\bigl\{|\mrmd\psi_{2,t}|^{2}
+|\mrmd^{*}\psi_{2,t}|^{2}\bigr\}
\le{}C(\delta_{0},r)[RM_{0}+\eta].
\eeq

\begin{flushleft}
{\it Step $3$: Decomposition of $\psi_{1,t}$}\\
\end{flushleft}
\par{}The decomposition and estimates presented
here represent an additional step required to extend the argument
of \cite{BOS2} to the manifold setting.
This arises due to the presence of the harmonic part in
\eqref{psi1t:PDE}.
More details are provided in Section
\ref{Sec::SteppingStones}.
Since $\psi_{1,t}$ solves \eqref{psi1t:PDE} then we may write
\begin{align}\label{psi1tconvolution}
\psi_{1,t}(x)
&=\int_{M}\!{}\mathopen{}
\left<G(x,y),
H^{\perp}(\mrmd\bigl[\tveep\times\mrmd\tveep\bigr]\chi)
\right>\mathclose{}\\
&=\int_{M}\!{}\mathopen{}
\left<G(x,y),
\mrmd\bigl[\tveep\times\mrmd\tveep\bigr]\chi\right>\mathclose{}
-\int_{M}\!{}\mathopen{}
\left<G(x,y),
H(\mrmd\bigl[\tveep\times\mrmd\tveep\bigr]\chi)
\right>\mathclose{}
\nonumber\\
&\eqqcolon\widetilde{\psi}_{1,t}(x)
-(G\star{}H(\mrmd[\tveep\times\mrmd\tveep]\chi))(x)    \nonumber
\end{align}
where $G$ is the Dirichlet kernel for $2$-forms on $M$.
In addition, we can use \eqref{psi1t:PDE} to establish
\beq\label{psi1tidentity}
\int_{M}\!{}\bigl\{
|\mrmd\psi_{1,t}|^{2}+|\mrmd^{*}\psi_{1,t}|^{2}\bigr\}
=\int_{M}\!{}\mathopen{}\left<\psi_{1,t},
H^{\perp}(\mrmd[\tveep\times\mrmd\tveep]\chi)\right>\mathclose{}.
\eeq
Combining \eqref{psi1tconvolution}, \eqref{psi1tidentity},
and estimates on the harmonic projection gives
\beq\label{psi1treduction}
\int_{M}\!{}\bigl\{
|\mrmd\psi_{1,t}|^{2}+|\mrmd^{*}\psi_{1,t}|^{2}\bigr\}
\le\int_{M}\!{}\mathopen{}\left<\widetilde\psi_{1,t},
\mrmd[\tveep\times\mrmd\tveep]\chi\right>\mathclose{}
+C(\delta_{0},r)[RM_{0}+\eta]^{2}
\eeq
and so it suffices to estimate
\beq\label{tpsi1t:Goal}
\int_{M}\!{}\mathopen{}\left<\widetilde{\psi}_{1,t},
\mrmd[\tveep\times\mrmd\tveep]\chi\right>\mathclose{}.
\eeq
The advantage of \eqref{tpsi1t:Goal} is that both the
``convolution" defining $\widetilde{\psi}$ as well as
the integral defining \eqref{tpsi1t:Goal} can be taken
over small geodesic ball.
This facilitates the use of normal coordinates and allows for
the use of a detailed coordinate description of the Green's
function.\\

\begin{flushleft}
{\it Step $4$: Decomposition of $\widetilde{\psi}_{1,t}$}\\
\end{flushleft}
\par{}Here we follow the decomposition
technique presented in \cite{BOS2} which adapts to the setting
of a Riemannian manifold with only minor differences.
A more complete discussion can be found in Section
\ref{Sec::SteppingStones}.\\

Following \cite{BOS2} we choose an appropriate $\alpha>0$ and
carefully construct a Lipschitz function
$l\colon[0,\infty)\to[0,\infty)$ supported on
$[\ep^{\alpha}r,32r]$ such that
$l\equiv1$ on $[2\ep^{\alpha}r,16r]$.
From this we define $m\colon[0,\infty)\to[0,\infty)$ by
\begin{equation}\label{mFunction}
m(s)\ceq\begin{cases}
1-l(s)& \text{for }s\in[0,16r]\\
0& \text{for }s\in(16r,\infty)
\end{cases}
\end{equation}
and set
\begin{equation*}
G(x,y)=m(d(x,y))G(x,y)
+(1-m(d(x,y)))G(x,y)\eqqcolon{}G^{i}(x,y)+G^{e}(x,y).
\end{equation*}
This decomposition allows us to define
\begin{align}
\widetilde{\psi}_{1,t}^{i}&\ceq\int_{B_{2r}(x_{T})}\!{}
\mathopen{}\left<G^{i}(x,y),\mrmd[\tilde{v}_{\ep}\times\mrmd\tilde{v}_{\ep}]\chi\right>\mathclose{}
\label{Psi1tIJiDefinition}\\
\widetilde{\psi}_{1,t}^{e}&\ceq
\int_{B_{32r}(x)}\!{}
\mathopen{}\left<G^{e}(x,y),\mrmd[\tilde{v}_{\ep}\times\mrmd\tilde{v}_{\ep}]\chi\right>\mathclose{}.
\label{Psi1tIJeDefinition}
\end{align}
Using \eqref{Psi1tIJiDefinition} and \eqref{Psi1tIJeDefinition}
we reduce estimating \eqref{tpsi1t:Goal} to estimating
\begin{align}
\int_{M}\!{}\mathopen{}\left<\widetilde{\psi}_{1,t}^{e},
\mrmd[\tveep\times\mrmd\tveep]\chi\right>\mathclose{}
\label{Psi1te:Goal}\\
\int_{M}\!{}\mathopen{}\left<\widetilde{\psi}_{1,t}^{i},
\mrmd[\tveep\times\mrmd\tveep]\chi\right>\mathclose{}.
\label{Psi1ti:Goal}
\end{align}
\newpage{}

\begin{flushleft}
{\it Step $5$: Estimating $\widetilde{\psi}_{1,t}^{e}$}\\
\end{flushleft}
\par{}Extending the argument to the setting of
a Riemannian manifold retains much of the style of the
corresponding one from \cite{BOS2}.
The main difference is that we need to make use
of detailed local coordinate expressions for the Green's
function on $2$-forms, $G$, as well as of the integrand of
\eqref{Psi1te:Goal}.
We highlight the main ideas below and refer the reader to
Section \ref{Sec::SteppingStones} for a more complete account.\\

Here, working in normal coordinates, we exploit the Jacobian structure of $\mrmd[\tveep\times\mrmd\tveep]$, see \eqref{jacobian.structure}, together with estimates from \cite{JS3}
to obtain an $L^{\infty}$ estimate on $\widetilde{\psi}_{1,t}^{e}$
which can be paired with pointwise estimates of
$\mrmd[\tveep\times\mrmd\tveep]\chi$ to show
\beq\label{Psi1te:Complete}
\int_{M}\!{}\mathopen{}\left<\widetilde{\psi}_{1,t}^{e},
\mrmd[\tveep\times\mrmd\tveep]\chi\right>\mathclose{}
\le{}C(\delta_{0},r)
\Bigl(\Eepgr(\veep,(x_{T},1),1)+C_{7}R_{1}E_{0}+1\Bigr)
[RM_{0}+\eta].
\eeq
These computations use detailed information on the form of the Green's function in normal coordinates.
\\

\begin{flushleft}
{\it Step $6$: Auxiliary parabolic problem}\\
\end{flushleft}
\par{}Unfortunately, the obtainable estimates for
$\widetilde{\psi}_{1,t}^{i}$
are insufficient to make use of duality in
\eqref{Psi1ti:Goal}.
As a result we, following \cite{BOS2}, introduce $\psi_{1}^{*}$
the solution to the parabolic problem
\beq\label{Psi1star:PDE}
\begin{cases}
\pp_{t}\psi_{1,t}^{*}-\Delta\psi_{1,t}^{*}=\mrmd[\tveep\times\mrmd\tveep]\chi&
\text{on }M\times(0,\infty)\\
\psi_{1,t}^{*}(\cdot,0)=0&
\text{on }M\times\{0\},
\end{cases}
\eeq
and use this to replace $\mrmd[\tveep\times\mrmd\tveep]\chi$
with $\pp_{t}\psi_{1}^{*}-\Delta\psi_{1}^{*}$ in
\eqref{Psi1ti:Goal}.
Thus, it suffices to estimate each of
\begin{align}
&\int_{M}\!{}\mathopen{}\left<\widetilde{\psi}_{1,t}^{i},
\pp_{t}\psi_{1}^{*}\right>\mathclose{}
\label{Psi1ti:partial}\\
&\int_{M}\!{}\mathopen{}\left<\widetilde{\psi}_{1,t}^{i},
-\Delta\psi_{1}^{*}\right>\mathclose{}.
\label{Psi1ti:Laplacian}
\end{align}
The arguments to estimate
\eqref{Psi1ti:partial} and \eqref{Psi1ti:Laplacian} extend to
the Riemannian setting with a bit of additional work.
For $\widetilde{\psi}_{1,t}^{i}$ the main obstacle is the
need for local coordinate expressions of the Green's function and
direct computations of the distributional Laplacian of $G^{i}$.
The techniques for estimating $\psi_{1}^{*}$ extend to the Riemannian setting but a
bit of care is needed.
We provide more details in Section \ref{Sec::SteppingStones}.\\

Using \eqref{Psi1star:PDE} it is possible to find a
time slice for which we have $L^{2}$ estimates on
$\pp_{t}\psi_{1}^{*}$ as well as be a member of $\Theta_{1}$.
For such $t$ we can estimate \eqref{Psi1ti:partial} by
Cauchy-Schwarz.
We estimate \eqref{Psi1ti:Laplacian} by making use of
information about the distributional laplacian of $G^{i}$, the
integral kernel defining $\widetilde{\psi}_{1,t}^{i}$.
These arguments closely follow \cite{BOS2}, but some adjustments are needed to adapt them to the Riemannian setting.

\section{Clearing Out Proof}\label{Sec::SteppingStones}

\hspace{15pt}In this section we present the details omitted from the outline
presented in Section \ref{Sec:Outline}.
As in Section \ref{Sec:Outline}, unless otherwise specified, all
metric related quantities will be associated to $\gR$ and the metric will
be suppressed from the notation.

\subsection{Reduction via rescaling}\label{Subsec:Reduction}

\hspace{15pt}Following \cite{BOS2} we first reduce the proof of Theorem
\ref{Theorem1} to that of Proposition \ref{Theorem1Reduction},
which is stated in Section \ref{Sec:Outline}.
We refer the reader to Section \ref{Sec:Outline} for the relevant
definitions surrounding the rescaled solution $\veep$ as well as to
A.4.5.1 of \cite{Col} for a detailed account of this reduction.\\

\emph{Proof of Theorem \ref{Theorem1}, assuming Proposition \ref{Theorem1Reduction}.}
Using the conclusion of Proposition \ref{Theorem1Reduction} as well as \eqref{RescaledAveragingInequality} leads to
\[
\widetilde{E}_{\ep,g_{R_{1}}}(\veep,(x_{T},1),1)+C_{7}E_{0}R_{1}
\le\mathcal{R}_{1}(\eta,R)
\]
where $\mathcal{R}_{1}(\eta,R)\to0^{+}$ as $\eta,R\to0^{+}$.
It is then possible to choose $T_{\ep}$, also dependent on $\sigma$, such that $T_{\ep}=1+O_{\sigma}(\ep^{2})$ and,
by an extension of Lemma $({\rm{III}}.3)$ of \cite{BBO} to our
setting,
\beq\label{PreliminaryEstimate}
1-|\veep(x_{T},T_{\ep})|
\le{}C_{M}\biggl[\frac{1}{\ep^{N}}\int_{B_{\ep}(x_{T})}\!{}(1-|\veep(x,T_{\ep})|^{2})^{2}\biggr]^{\frac{1}{N+2}}
\le{}D_{M}\mathcal{R}_{1}(\eta,R)
\eeq
where $C_{M}$ and $D_{M}$ are constants that depend on $M$.
Next, using the time derivative estimate from \eqref{vepEstimates}
and our choice of $T_{\ep}$ we have
\beq\label{MeanValueEstimate}
|\veep(x_{T},T_{\ep})-\veep(x_{T},1)|\le\frac{\sigma}{2}.
\eeq
Hence, after combining \eqref{PreliminaryEstimate} and
\eqref{MeanValueEstimate}, as well as choosing $R_{0}$ and $\eta_{0}$
small enough that we can ensure
$D_{M}\mathcal{R}_{1}(\eta,R)\le\frac{\sigma}{2}$,
we will have the desired result.
\null\nobreak\hfill\ensuremath{\square}\newline

As remarked in Section \ref{Sec:Outline}, 
to prove Proposition
\ref{Theorem1Reduction} it suffices to demonstrate
that for some $0<\delta_{0}<\frac{1}{16}$ there is
$\delta\in[\delta_{0},2\delta_{0}]$ such that
\[
\widetilde{E}_{\ep,\gR}(\veep,(x_{T},1),\delta)
\le\frac{e^{-C_{2}}}{8}
\Bigl(\widetilde{E}_{\ep,\gR}(\veep,(x_{T},1),1)+
C_{7}E_{0}R_{1}\Bigr)
+\mathcal{R}(\eta,R)
\]
where $\mathcal{R}(\eta,R)$ tends to zero as $\eta,R\to0^{+}$
and $R_{1}$ is as in Proposition \ref{Averaging}.
To see this, observe that by applying Proposition
\ref{MonotonicityFormula} through $\uvep$, using
\eqref{Theorem1FurtherReduction}, as well as that
$C_{1}\le{}C_{7}$ and $\delta<2\delta_{0}<\frac{1}{8}$ we have
\begin{align*}
\widetilde{E}_{\ep,\gR}(\veep,(x_{T},1),\delta_{0})
&\le{}e^{C_{2}}\widetilde{E}_{\ep,\gR}(\veep,(x_{T},1),\delta)
+\delta[C_{1}E_{0}R_{1}]\\
&\overset{\eqref{Theorem1FurtherReduction}}\le\frac{1}{8}\Bigl(\widetilde{E}_{\ep,\gR}(\veep,(x_{T},1),1)
+C_{7}E_{0}R_{1}\Bigr)+\delta[C_{7}E_{0}R_{1}]
+e^{C_{2}}\mathcal{R}(\eta,R)\\
&\le\frac{1}{4}\Bigl(\widetilde{E}_{\ep,\gR}(\veep,(x_{T},1),1)
+C_{7}E_{0}R_{1}\Bigr)
+e^{C_{2}}\mathcal{R}(\eta,R).
\end{align*}

\subsection{Estimate of \texorpdfstring{$\vphi_{t}$}{}}\label{subsec:vphit}
\hspace{15pt}To obtain the PDE \eqref{vphiPDE} we
begin by taking the cross product of
$({\rm PGL})_{\ep}$ with $\veep$ we obtain
\beq\label{PDE:CrossProduct}
0=\veep\times\pp_{t}\veep+\mrmd^{*}(\veep\times\mrmd\veep)
\hspace{20pt}
\text{on }M\times(0,\infty).
\eeq
By applying $\mrmd^{*}$ to \eqref{LocalHodgeDecomposition} we see that
\beq\label{VarphitPoisson}
\mrmd^{*}(\veep\times\mrmd\veep)=-\Delta\vphi_{t}
\hspace{20pt}
\text{on }B_{\frac{3r}{2}}(x_{T})\times\{t\}.
\eeq
Rewriting $\veep\times\pp_{t}\veep$ as
\[
\veep\times\pp_{t}\veep=\veep\times
\biggl(\frac{\mathopen{}\left<\nabla{}\Kap,
\nabla\veep\right>\mathclose{}}{\Kap}+\pp_{t}\veep\biggr)
-\mathopen{}\left<\frac{\mrmd\Kap}{\Kap},
\mrmd\vphi_{t}+\mrmd^{*}\psi_{t}+\xi_{t}\right>\mathclose{}
\]
and then applying \eqref{PDE:CrossProduct} and \eqref{VarphitPoisson}
leads to
\beq\label{ProvisionalvphiPDE}
-\Delta\vphi_{t}
-\mathopen{}\left<\frac{\nabla\Kap}{\Kap},\nabla\vphi_{t}\right>\mathclose{}
=\veep\times\biggl(\frac{-\mathopen{}\left<\nabla\Kap,
\nabla\veep\right>\mathclose{}}
{\Kap}-\pp_{t}\veep\biggr)
+\mathopen{}\left<\frac{\mrmd\Kap}{\Kap},
\mrmd^{*}\psi_{t}+\xi_{t}\right>\mathclose{}.
\eeq
Finally, noting that
\[
-\Delta\vphi_{t}
-\mathopen{}\left<\frac{\nabla\Kap}{\Kap},\nabla\vphi_{t}\right>\mathclose{}
=\frac{-1}{\Kap}\text{div}(\Kap\nabla\vphi_{t})
\]
we may rewrite \eqref{ProvisionalvphiPDE} as well as introduce
boundary conditions coming from \eqref{LocalHodgeDecomposition} to obtain
\eqref{vphiPDE}.
We use elliptic PDE techniques to estimate the $L^{2}$ norm of
$\nabla\vphi_{t}$ on $B_{s}(x_{T})$, where $s\in\bigl[r,\frac{3r}{2}\bigr]$ will be chosen
later, in terms of the data from \eqref{vphiPDE}.
First, we  multiply the PDE from \eqref{vphiPDE} by
$-2\delta^{2}\mathopen{}\left<\nabla{}v,\nabla{}K_{ap}\right>\mathclose{}$
and integrate by parts to obtain
\begin{align}
-2\delta^{2}\int_{B_{s}(x_{T})}\!{}h\mathopen{}\left<\nabla\vphi,\nabla\Kap\right>\mathclose{}
&=2\delta^{2}\int_{B_{s}(x_{T})}\!{}
\text{div}(\Kap\nabla{}\vphi)
\frac{\mathopen{}\left<\nabla{}\vphi,\nabla\Kap\right>\mathclose{}}{\Kap}  \label{EllipticIntegralIdentity}\\
&=-2\delta^{2}\int_{B_{s}(x_{T})}\!{}\Kap\mathopen{}\left<\nabla\vphi,\nabla\biggl(
\frac{\mathopen{}\left<\nabla\vphi,\nabla\Kap\right>\mathclose{}}{\Kap}\biggr)\right>\mathclose{}  \nonumber\\
&\hspace{8pt}+2\delta^{2}\int_{\pp{}B_{s}(x_{T})}\!{}
g\mathopen{}\left<\nabla\vphi,\nabla\Kap\right>\mathclose{}.    \nonumber
\end{align}
Then one can verify, for example by a pointwise computation in normal coordinates,
that
\beq\label{CoordinateComputation}
\Kap\mathopen{}\left<\nabla\vphi,\nabla\biggl(
\frac{\mathopen{}\left<\nabla\vphi,
\nabla\Kap\right>\mathclose{}}{\Kap}\biggr)\right>\mathclose{}
=
\text{Hess}(\Kap)(\nabla\vphi,\nabla\vphi)
-\frac{\mathopen{}\left<\nabla\vphi,\nabla\Kap\right>\mathclose{}^{2}}{\Kap}
+\frac{1}{2}\mathopen{}\left<\nabla\bigl(|\nabla\vphi|^{2}\bigr),\nabla\Kap\right>\mathclose{}.
\eeq
Using \eqref{CoordinateComputation} in \eqref{EllipticIntegralIdentity}
followed by integrating by parts leads to
\begin{align}
2\delta^{2}\int_{B_{s}(x_{T})}&\!{}
\biggl[\frac{-\Delta\Kap}{2}|\nabla\vphi|^{2}+{\rm{Hess}}(\Kap)(\nabla\vphi,\nabla\vphi)
-\frac{\mathopen{}\left<\nabla\vphi,\nabla\Kap\right>\mathclose{}^{2}}{\Kap}\biggr]
-2\delta^{2}\int_{B_{s}(x_{T})}\!{}h\mathopen{}\left<\nabla\vphi,\nabla{}G\right>\mathclose{}    \label{ProvisionalIntegralIdentity}\\
&=-\delta^{2}\int_{\pp{}B_{s}(x_{T})}\!{}|\nabla\vphi|^{2}\frac{\pp\Kap}{\pp{}r}
+2\delta^{2}\int_{\pp{}B_{s}(x_{T})}\!{}
g\mathopen{}\left<\nabla\vphi,\nabla{}\Kap\right>\mathclose{}.    \nonumber
\end{align}
We then compute the integrand of the braced term of
\eqref{ProvisionalIntegralIdentity} more
explicitly in terms of the function \eqref{dsquared} and
apply \eqref{HessianComparisonInequality} to obtain
\begin{align}\label{IntermediateVarphiInequality}
\int_{B_{s}(x_{T})}&\!{}\frac{(N-2)}{2}|\nabla\vphi|^{2}\Kap
-\int_{B_{s}(x_{T})}
\frac{(d(x,x_{T}))^{2}}{4\delta^{2}}\biggl[1+\frac{2N\lambda\delta^{2}}{3}-2\mu\delta^{2}\biggr]
|\nabla\vphi|^{2}\Kap
-2\delta^{2}\int_{B_{s}(x_{T})}\!{}h\mathopen{}\left<\nabla\vphi,\nabla\Kap\right>\mathclose{}\\
&\ge-\delta^{2}\int_{\pp{}B_{s}(x_{T})}\!{}|\nabla\vphi|^{2}
\frac{\pp\Kap}{\pp{}r}
+2\delta^{2}\int_{\pp{}B_{s}(x_{T})}\!{}
g\mathopen{}\left<\nabla\vphi,\nabla\Kap\right>\mathclose{}.    \nonumber
\end{align}
More details can be found in Lemma A.4.4 of \cite{Col}.
From there we appeal to standard estimates and choose
$\delta_{0}$ such that
$0<2\delta_{0}\le\frac{1}{2\sqrt{\frac{1}{2}-\frac{N\lambda}{6}+\frac{\mu}{2}}}$ to obtain
\begin{align*}
&-2\delta^{2}\int_{B_{s}(x_{T})}\!{}
h\mathopen{}\left<\nabla\vphi,\nabla\Kap\right>\mathclose{}
-\int_{B_{s}(x_{T})}
\frac{(d(x,x_{T}))^{2}}{4\delta^{2}}\biggl[1+\frac{2N\lambda\delta^{2}}{3}-2\mu\delta^{2}\biggr]
|\nabla\vphi|^{2}\Kap\\
&\le\frac{1}{2}\int_{B_{s}(x_{T})}\!{}h^{2}\Kap
\end{align*}
which, when combined with \eqref{IntermediateVarphiInequality},
leads to
\begin{align}\label{PrelimVarphiInequality}
\int_{B_{s}(x_{T})}&\!{}\frac{(N-2)}{2}|\nabla\vphi|^{2}\Kap
+\frac{1}{2}\int_{B_{s}(x_{T})}\!{}h^{2}\Kap
\\
&\ge-\delta^{2}\int_{\pp{}B_{s}(x_{T})}\!{}|\nabla\vphi|^{2}
\frac{\pp\Kap}{\pp{}r}
+2\delta^{2}\int_{\pp{}B_{s}(x_{T})}\!{}
g\mathopen{}\left<\nabla\vphi,\nabla\Kap\right>\mathclose{}.    \nonumber
\end{align}
Then an explicit computation of $\frac{\pp\Kap}{\pp{}r}$ and
$\mathopen{}\left<\nabla\vphi,\nabla\Kap\right>\mathclose{}$,
analogous to \cite{BOS2}, give
\begin{align}\label{BoundaryInequality}
\int_{\pp{}B_{s}(x_{T})}\!{}|\nabla_{\perp}\vphi|^{2}K_{ap}
&\le\int_{B_{s}(x_{T})}\!{}\frac{(N-2)}{s}|\nabla{}\vphi|^{2}
K_{ap}
+\frac{1}{s}\int_{B_{s}(x_{T})}\!{}h^{2}K_{ap}\\
&+\int_{\pp{}B_{s}(x_{T})}\!{}g^{2}K_{ap}  \nonumber
\end{align}
for each $s\in[r,3r\slash2]$ where
\[
\nabla_{\perp}\vphi\ceq\nabla\vphi-
\frac{\pp\vphi}{\pp{}r}\frac{\pp}{\pp{}r}.
\]
More details can be found in Corollary A.4.5 of \cite{Col}.
This estimate will permit us to obtain control over the $L^{2}$ norm of $\vphi$ on $\pp{}B_{s}(x_{T})$ in terms of $h$ and $g$.
Next if we multiply the PDE of \eqref{vphiPDE} by $\vphi$, integrate by parts, and make use of \eqref{BoundaryInequality}, the Poincar\'e-Wirtinger, and
Young's inequality we obtain
\begin{align}\label{vphiWeightedInequality}
\int_{B_{s}(x_{T})}\!{}
&|\nabla\vphi|^{2}e^{\frac{-(d(x,x_{T}))^{2}}{4\delta^{2}}}\\
&\le{}C(\delta,r)\biggl[\int_{B_{s}(x_{T})}\!{}h^{2}e^{-\frac{(d(x,x_{T}))^{2}}{4\delta^{2}}}
+\biggl(\int_{B_{s}(x_{T})}\!{}h^{2}e^{\frac{-(d(x,x_{T}))^{2}}{4\delta^{2}}}\biggr)^{\frac{1}{2}}
\biggl(\int_{\pp{}B_{s}(x_{T})}\!{}
g^{2}e^{\frac{-(d(x,x_{T}))^{2}}{4\delta^{2}}}\biggr)^{\frac{1}{2}}\biggr]  \nonumber\\
&+K_{M}r\int_{\pp{}B_{s}(x_{T})}\!{}g^{2}e^{\frac{-(d(x,x_{T}))^{2}}{4\delta^{2}}}  \nonumber
\end{align}
for each $s\in[r,3r\slash2]$.
We notice that when $h$ is the data from \eqref{vphiPDE} then we have the
pointwise estimate
\beq\label{hData:PointwiseBound}
h^{2}\le{}
C(\delta_{0},r)\Bigl[\Xi(\veep,(x_{T},1)
+|\mrmd^{*}\psi_{t}|^{2}+|\xi_{t}|^{2})\Bigr].
\eeq
As a result of \eqref{hData:PointwiseBound} and the assumption that
$t\in\Theta_{1}$ we have
\beq\label{hData:L2Bound}
\int_{B_{s}(x_{T})}\!{}h^{2}e^{\frac{-(d(x,x_{T}))^{2}}{4\delta^{2}}}
\le{}C(\delta_{0},r)\biggl[RM_{0}+\eta
+\int_{B_{\frac{3r}{2}}(x_{T})}\!{}
\bigl(|\mrmd^{*}\psi_{t}|^{2}+|\xi_{t}|^{2}\bigr)\biggr].
\eeq
A similar pointwise estimate for the data $g$ from \eqref{vphiPDE} leads to
\beq\label{gData:PrelimL2Bound}
\int_{\pp{}B_{s}(x_{T})}\!{}g^{2}e^{\frac{-(d(x,x_{T}))^{2}}{4\delta^{2}}}
\le{}K_{M}
\int_{\pp{}B_{s}(x_{T})}\!{}\Bigl[|\nabla\veep|^{2}+|\mrmd^{*}\psi_{t}|^{2}
+|\xi_{t}|^{2}\Bigr]e^{\frac{-(d(x,x_{T}))^{2}}{4\delta^{2}}}.
\eeq
Averaging to find a suitable
$s\in\bigl[r,\frac{3r}{2}\bigr]$ and
manipulating Gaussian functions gives
\begin{align}
\int_{\pp{}B_{s}(x_{T})}\!{}g^{2}e^{\frac{-(d(x,x_{T}))^{2}}{4\delta^{2}}}
\le{}\frac{K_{M}}{r}
\int_{B_{s}(x_{T})}\!{}|\nabla\veep|^{2}e^{\frac{-(d(x,x_{T}))^{2}}{4\delta^{2}}}
+\frac{K_{M}}{r}\int_{B_{s}(x_{T})}\!{}\Bigl[|\mrmd^{*}\psi_{t}|^{2}+|\xi_{t}|^{2}\Bigr]
\label{gData:L2Bound}\\
=\frac{(4\pi)^{\frac{N}{2}}K_{M}\delta^{N}}{r}\int_{B_{s}(x_{T})}\!{}|\nabla\veep|^{2}K_{ap,\gR}
+\frac{K_{M}}{r}\int_{B_{s}(x_{T})}\!{}\Bigl[|\mrmd^{*}\psi_{t}|^{2}+|\xi_{t}|^{2}\Bigr].
\nonumber
\end{align}
More details are provided in A.4.3.1 of \cite{Col}.
Combining \eqref{vphiWeightedInequality} with \eqref{hData:L2Bound} and
\eqref{gData:L2Bound} gives \eqref{vphiEstimate} if we choose $\delta_{0}$ small
enough that $2\delta_{0}\sqrt{8C_{6}}<r$.

\subsection{Estimate of \texorpdfstring{$\psi_{t}$}{}}
\hspace{15pt}The strategy from \cite{BOS2} to estimate $\psi_{t}$ extends to our setting with
a few modifications mostly caused by the use of coordinates and the
possibility of non-trivial homology.
In particular, some additional work is due to the presence of the harmonic
projection $H$.\\

We first observe that since $\psi_{t}$ solves \eqref{LocalizedCurrent} then we can represent $\psi_{t}$ as
\beq\label{PsiRepresentation}
\psi_{t}(x)=\int_{M}\!{}\mathopen{}\left<G(x,y),H^{\perp}(\mrmd[\veep(y)\times\mrmd\veep(y)]\chi(y))\right>\mathclose{}\dvol(y)
\eeq
where $G$ is the Dirichlet kernel for $2$-forms on $M$ with the metric $\gR$,
which can be constructed in coordinates using the results of \cite{Au} and
\cite{BdR}.
In particular, for each $x,y\in{}M$ we have $G(\cdot,y),G(x,\cdot)\in{}W^{1,1}(M;\wedge^{(2,2)}M)$,
$-\Delta_{x}G(\cdot,y)\in{}L^{1}(M;\wedge^{2}M)$,
and when $x,y$ are elements of a normal coordinate neighbourhood
then
\begin{align}
|G(x,y)|\le{}&K(d(x,y))^{2-N},
\quad
|DG(x,y)|\le{}K(d(x,y))^{1-N}  \label{GreensFunction:Properties1}\\
&\left|-\Delta_{x}G(x,y)\right|\le{}K(d(x,y))^{2-N}.  \label{GreensFunction:Properties2}
\end{align}
In particular, if we were to rescale the metric by a factor $a^{2}$ for $a>0$ we would
find
\beq\label{GreenFunctionScaling}
G_{a^{2}g}=a^{6-N}G_{g}
\eeq
where $G_{g}$ denotes the Green's function constructed using the metric $g$.
Next we introduce a smooth function $\rho$ such that
\beq\label{rho.def}
\rho(s)=1\hspace{5pt}\text{for }s\in[0,1\slash4],\qquad
\rho(s)=\frac{1}{s}\hspace{5pt}\text{for }s\ge1\slash2,\qquad
\left\|\rho'\right\|_{L^{\infty}(\R)}\le4
\eeq
as well as the localized version of $\veep$
\beq\label{Localizedu}
\tveep(x,t)\ceq\tau(x,t)\veep(x,t),\qquad\qquad  \tau(x,t)\ceq\rho(|\veep(x,t)|).
\eeq
Noting that
$\tveep\times\mrmd\tveep=\tau^{2}\veep\times\mrmd\veep$,
we split \eqref{PsiRepresentation} into
\begin{align}
\psi_{t}(x)
&=\int_{M}\!{}\mathopen{}\left<G(x,y),H^{\perp}(\mrmd[\tveep\times\mrmd\tveep]\chi)\right>\mathclose{}
+\int_{M}\!{}\mathopen{}\left<G(x,y),H^{\perp}(\mrmd[(1-\tau^{2})\veep\times\mrmd\veep]\chi)\right>\mathclose{}
\label{Psi12t:IntegralRepresentation}\\
&\eqqcolon\psi_{1,t}+\psi_{2,t}.    \nonumber
\end{align}
From the definitions of $\psi_{1,t}$ and $\psi_{2,t}$ we see that
\begin{align}
-\Delta\psi_{1,t}&=H^{\perp}(\mrmd[\tveep\times\mrmd\tveep]\chi)
\hspace{10pt}\text{on }M\times\left\{t\right\}
\label{Psi1tPDE}\\
-\Delta\psi_{2,t}&=H^{\perp}(\mrmd[(1-\tau^{2})\veep\times\mrmd\veep]\chi)
\hspace{10pt}\text{on }M\times\{t\}.
\label{Psi2tPDE}
\end{align}
We observe that there is $K>0$ such that
\begin{align}
\bigl|1-\tau^{2}(x,t)\bigr|&\le{}K\bigl|1-|\veep(x,t)|^{2}\bigr|
\label{TauEstimate}\\
\bigl|\mrmd[\tveep\times\mrmd\tveep]\chi\bigr|&\le{}KV_{\ep}(\veep)
\label{JacobianEstimate}
\end{align}
over $M\times\{t\}$.

\toclesslab\subsubsection{Estimate of Harmonic Projection}{Subsec:HarmProj}
Before we begin estimating $\psi_{1,t}$ and $\psi_{2,t}$ we
first obtain estimates for
$H(\mrmd[\veep\times\mrmd\veep]\chi)$.
Note that
\beq\label{CurrentSplitting}
\mrmd[\veep\times\mrmd\veep]\chi
=\mrmd[\tveep\times\mrmd\tveep]\chi
+\mrmd[(1-\tau^{2})(\veep\times\mrmd\veep)\chi]
+(1-\tau^{2})[\veep\times\mrmd\veep]\wedge\mrmd\chi
\eeq
and so the definition of $H$ implies that
\beq\label{CurrentSplittingHarmonic}
H(\mrmd[\veep\times\mrmd\veep]\chi)
=
H(\mrmd[\tveep\times\mrmd\tveep]\chi)
+H((1-\tau^{2})[\veep\times\mrmd\veep]\wedge\mrmd\chi).
\eeq
We now estimate each of the terms on the
right-hand side of \eqref{CurrentSplittingHarmonic}.
By straightforward estimates 
using \eqref{H:Def}, the definition of $H$,
we may find a constant $K$ such that
\begin{align}
\left\|H(\omega)\right\|_{L^{2}(\wedge^{2}M)}&\le{}KR_{1}^{\frac{N}{2}}
\left\|\omega\right\|_{L^{1}(\wedge^{2}M)}    \label{HCont:L2}\\
\left\|H(\omega)\right\|_{L^{\infty}(\wedge^{2}M)}&\le{}KR_{1}^{N}
\left\|\omega\right\|_{L^{1}(\wedge^{2}M)}    \label{HCont:LInf}
\end{align}
for all $2$-forms $\omega$, where the exponents on $R_{1}$ are due to
the scaling properties of the basis $\{\gamma_{\gR,i}\}_{i=1}^{\beta_{2}(M)}$ appearing in \eqref{H:Def}.
Next observe that by \eqref{JacobianEstimate}
and manipulations with Gaussian functions using that $\chi$ is
supported on $B_{4r}(x_{T})$, we have
\beq\label{vtildecurrent:L1}
\left\|\mrmd(\tilde{v}_{\ep}\times\mrmd\tilde{v}_{\ep})\chi\right\|
_{L^{1}(\wedge^{2}M)}
\le{}C(\delta_{0},r)\int_{M}\!{}\Vep(\veep)K_{ap,\gR}.
\eeq
Now we consider $(1-\tau^{2})[\veep\times\mrmd\veep]\wedge\mrmd\chi$.
Using \eqref{vepEstimates}, Cauchy-Schwarz, \eqref{TauEstimate}, and
manipulations with Gaussian functions using that $\chi$ is supported on
$B_{4r}(x_{T})$ leads to
\beq\label{vtilderemainder:L1}
\left\|(1-\tau^{2})[\veep\times\mrmd\veep]\wedge\mrmd\chi\right\|
_{L^{1}(\wedge^{2}M)}
\le{}C(\delta_{0},r)\biggl(\int_{M}\!{}\Vep(\veep)K_{ap,\gR}\biggr)^{\frac{1}{2}}.
\eeq
We will make use of various combinations of \eqref{HCont:L2} and
\eqref{HCont:LInf} with \eqref{vtildecurrent:L1} and
\eqref{vtilderemainder:L1}.\\

\toclesslab\subsubsection{Estimate of $\psi_{2,t}$}{Subsec:Psi2t}
The aim of this subsection is to establish the following estimate:
\beq\label{Psi2tEstimate}
\int_{M\times\{t\}}\!{}
\bigl[|\mrmd\psi_{2,t}|^{2}+|\mrmd^{*}\psi_{2,t}|^{2}\bigr]
\le{}C(\delta_{0},r)[RM_{0}+\eta].
\eeq
This will be achieved by making use of the Poisson problem \eqref{Psi2tPDE},
appealing to elliptic regularity, and applying our assumption that
$t\in\Theta_{1}$.
We note that, as in previous estimates, the goal is to show that the data from
the PDE \eqref{Psi2tPDE} can be estimated in terms of
\eqref{WeightedVEnergy} and \eqref{WeightedXiEnergy}.\\

We notice that $\psi_{2,t}$ can be further decomposed as
$\psi_{2,t}=\psi_{2,t}^{1}+\psi_{2,t}^{2}$ where $\psi_{2,t}^{1}$, $\psi_{2,t}^{2}$
satisfy
\begin{align}
-\Delta\psi_{2,t}^{1}&=\mrmd[(1-\tau^{2})(\veep\times\mrmd\veep)\chi]
\hspace{5pt}\text{on }M\times\{t\}    \label{Psi2t1PDE}\\
-\Delta\psi_{2,t}^{2}&=
H^{\perp}((1-\tau^{2})[\veep\times\mrmd\veep]\wedge\mrmd\chi)
\hspace{5pt}\text{on }M\times\{t\}   \label{Psi2t2PDE}
\end{align}
where in \eqref{Psi2t1PDE} we have used that $H^{\perp}$ is the
identity on exact forms.
Since $\psi_{2,t}^{1}$ solves \eqref{Psi2t1PDE} elliptic regularity gives
\beq\label{Psi2t1InitialEstimate}
\int_{M\times\{t\}}\!{}
\bigl\{|\mrmd\psi_{2,t}^{1}|^{2}+|\mrmd^{*}\psi_{2,t}^{1}|^{2}\bigr\}
\le{}K\left\||1-\tau^{2}|
\bigl|(\veep\times\mrmd\veep)\chi\bigr|\right\|_{L^{2}(\wedge^{2}M)}
^{2}.
\eeq
Using \eqref{vepEstimates}, \eqref{TauEstimate}, and manipulations with Gaussian
functions using that the support of $\chi$ is $B_{4r}(x_{T})$ gives
\beq\label{Psi2t1:DataEstimate}
\int_{M\times\{t\}}\!{}|1-\tau^{2}|^{2}|(\veep\times\mrmd\veep)\chi|^{2}
\le{}C(\delta_{0},r)\int_{M\times\{t\}}\!{}\Vep(\veep)K_{ap,\gR}.
\eeq
Combining \eqref{Psi2t1InitialEstimate}, \eqref{Psi2t1:DataEstimate}, and using
that $t\in\Theta_{1}$ leads to
\beq\label{Psi2t1:Estimate}
\int_{M\times\{t\}}\!{}
\bigl\{|\mrmd\psi_{2,t}^{1}|^{2}+|\mrmd^{*}\psi_{2,t}^{1}|^{2}\bigr\}
\le{}C(\delta_{0},r)[RM_{0}+\eta].
\eeq
Next, since $\psi_{2,t}^{2}$ solves \eqref{Psi2t2PDE} then elliptic regularity
gives
\beq\label{Psi2t2PreliminaryEstimate}
\int_{M\times\{t\}}\!{}\bigl\{|\mrmd\psi_{2,t}^{2}|^{2}
+|\mrmd^{*}\psi_{2,t}^{2}|^{2}\bigr\}
\le{}K\left\|
H^{\perp}((1-\tau^{2})[\veep\times\mrmd\veep]\wedge\mrmd\chi)\right\|
_{L^{2}(\wedge^{2}M)}
^{2}.
\eeq
It then follows from \eqref{HCont:L2}, \eqref{vtilderemainder:L1},
and similar considerations as in \eqref{Psi2t1:DataEstimate} that we have
\beq\label{Psi2t2:DataEstimate}
\left\|
H^{\perp}((1-\tau^{2})[\veep\times\mrmd\veep]\wedge\mrmd\chi)\right\|
_{L^{2}\wedge^{2}(M)}
^{2}\le{}
C(\delta_{0},r)\int_{M\times\{t\}}\!{}V_{\ep}(\veep)K_{ap,\gR}.
\eeq
From \eqref{Psi2t2PreliminaryEstimate}, \eqref{Psi2t2:DataEstimate}, and
the assumption that $t\in\Theta_{1}$ it follows that
\beq\label{Psi2t2:Estimate}
\int_{M\times\{t\}}\!{}\bigl\{|\mrmd\psi_{2,t}^{2}|^{2}
+|\mrmd^{*}\psi_{2,t}^{2}|^{2}\bigr\}
\le{}C(\delta_{0},r)[RM_{0}+\eta].
\eeq
Finally, combining \eqref{Psi2t1:Estimate} and \eqref{Psi2t2:Estimate} gives
\eqref{Psi2tEstimate}.\\

\toclesslab\subsubsection{Estimate of $\psi_{1,t}$}{Subsec:Psi1t}
Next we estimate $\psi_{1,t}$.
We proceed with the a slightly modified version of the strategy applied in
\cite{BOS2}.
The main difference is the need to estimate terms relating to the harmonic
projection $H$.
In addition, some computations need to be done in coordinates, for example the
estimate of the low frequency term $\widetilde{\psi}_{1,t}^{e}$.\\

Taking the inner product of \eqref{Psi1tPDE} with $\psi_{1,t}$ and integrating by
parts we obtain
\beq\label{Psi1tEstimate}
\int_{M\times\{t\}}\!{}
\bigl[|\mrmd\psi_{1,t}|^{2}+|\mrmd^{*}\psi_{1,t}|^{2}\bigr]
=\int_{M\times\{t\}}\!{}\mathopen{}\left<\psi_{1,t},
H^{\perp}(\mrmd[\tveep\times\mrmd\tveep]\chi)\right>\mathclose{}.
\eeq
Thus, to obtain control of the $L^{2}$ norms of
the differential and codifferential of $\psi_{1,t}$ it
suffices to estimate
\beq\label{Psi1tInnerProduct}
\int_{M\times\{t\}}\!{}\mathopen{}\left<\psi_{1,t},
H^{\perp}(\mrmd[\tveep\times\mrmd\tveep]\chi)\right>\mathclose{}.
\eeq
As a result, we focus on obtaining an upper bound of \eqref{Psi1tInnerProduct}.
We proceed through a series of steps.\\

\begin{flushleft}
{\it Step $1$: Localization}\\
\end{flushleft}
\par{}Due to the presence of the harmonic projection we will need a few additional
estimates not needed in \cite{BOS2}.
We begin by noting that $\psi_{1,t}$ can be decomposed as
\begin{align}\label{KernelRepresentation}
\psi_{1,t}&=\int_{M}\!{}\mathopen{}\left<G,H^{\perp}(\mrmd[\tveep\times\mrmd\tveep]\chi)\right>\mathclose{}\\
&=\int_{M}\!{}\mathopen{}\left<G,\mrmd[\tveep\times\mrmd\tveep]\chi\right>
\mathclose{}
-\int_{M}\!{}\mathopen{}\left<G,H(\mrmd[\tveep\times\mrmd\tveep]\chi)\right>
\mathclose{}    \nonumber\\
&\eqqcolon\widetilde{\psi}_{1,t}-G\star{}H(\mrmd[\tveep\times\mrmd\tveep]\chi).
\nonumber
\end{align}
We then use \eqref{KernelRepresentation} to rewrite \eqref{Psi1tInnerProduct} as
\begin{align}\label{Psi1tLocalization}
\int_{M}\!{}
\mathopen{}\left<\psi_{1,t},
H^{\perp}(\mrmd[\tveep\times\mrmd\tveep]\chi)\right>\mathclose{}
&=\int_{M}\!{}
\mathopen{}\left<\widetilde{\psi}_{1,t},
\mrmd[\tveep\times\mrmd\tveep]\chi\right>\mathclose{}\\
&-\int_{M}\!{}
\mathopen{}\left<G\star{}H\bigl(\mrmd[\tveep\times\mrmd\tveep]\chi\bigr),
\mrmd[\tveep\times\mrmd\tveep]\chi\right>\mathclose{}    \nonumber\\
&-\int_{M}\!{}
\mathopen{}\left<\psi_{1,t},
H\bigl(\mrmd[\tveep\times\mrmd\tveep]\chi\bigr)\right>\mathclose{}.
\nonumber
\end{align}
We now estimate the last two terms of \eqref{Psi1tLocalization}.
We will start with
\[
\int_{M\times\{t\}}\!{}\mathopen{}\left<\psi_{1,t},
H(\mrmd[\tveep\times\mrmd\tveep]\chi)\right>\mathclose{}.
\]
Observe that by the integral representation of $\psi_{1,t}$,
standard integral kernel estimates, as well as
$W^{1,1}$ estimates on $G$ we obtain
\beq\label{Pst1tLocalization:Prelim}
\int_{M\times\{t\}}\!{}|\psi_{1,t}|
\le{}KR_{1}^{-2}
\biggl\{
\int_{M}\!{}\bigl|\mrmd[\tilde{v}_{\ep}\times\mrmd\tilde{v}_{\ep}]\chi\bigr|
+\left\|H(\mrmd[\tilde{v}_{\ep}\times\mrmd\tilde{v}_{\ep}]\chi)\right\|_{L^{\infty}(\wedge^{2}M)}\vol(M)\biggr\}.
\eeq
Combining \eqref{HCont:LInf}, \eqref{vtildecurrent:L1}, \eqref{JacobianEstimate},
manipulations of Gaussians that use that $\chi$ is supported on $B_{4r}(x_{T})$,
as well as the assumption that $t\in\Theta_{1}$ with
\eqref{Pst1tLocalization:Prelim} we obtain
\beq\label{Psi1tLocalization:L1}
\int_{M}\!{}|\psi_{1,t}|
\le{}C(\delta_{0},r)R_{1}^{-2}[RM_{0}+\eta].
\eeq
Finally, by \eqref{Psi1tLocalization:L1}, \eqref{HCont:LInf}, and
\eqref{vtildecurrent:L1} we have
\beq\label{psi1tHarmonicPart}
\biggl|\int_{M}\!{}\mathopen{}\left<\psi_{1,t},H(\mrmd[\tveep\times\mrmd\tveep]\chi)\right>\mathclose{}\biggr|
\le{}C(\delta_{0},r)R_{1}^{N-2}\bigl[RM_{0}+\eta\bigr]^{2}.
\eeq
Next we estimate $G\star{}H(\mrmd[\tveep\times\mrmd\tveep]\chi)$.
Observe that by \eqref{HCont:LInf}, \eqref{vtildecurrent:L1}, and
manipulations with Gaussian functions which use that $\chi$ is supported on
$B_{4r}(x_{T})$ we have
\[
\left\|G\star{}H(\mrmd[\tveep\times\mrmd\tveep]\chi)\right\|
_{L^{\infty}(\wedge^{2}M)}
\le{}C(\delta_{0},r)R_{1}^{N-2}\int_{M\times\{t\}}\!{}\Vep(\veep)K_{ap,\gR}.
\]
Thus, combining this with \eqref{JacobianEstimate}, similar Gaussian function
manipulations to those in \eqref{Psi1tLocalization:L1},
and the assumption that $t\in\Theta_{1}$, we have
\beq\label{KernelHarmonicEstimate}
\biggl|\int_{M}\!{}
\mathopen{}\left<G\star{}H(\mrmd[\tilde{v}_{\ep}\times\mrmd\tilde{v}_{\ep}]\chi),\mrmd[\tilde{v}_{\ep}\times\mrmd\tilde{v}_{\ep}]\chi\right>\mathclose{}\biggr|
\le{}C(\delta_{0},r)R_{1}^{N-2}[RM_{0}+\eta]^{2}.
\eeq
\begin{flushleft}
{\it Step $2$: Decomposition of $\widetilde{\psi}_{1,t}$}\\
\end{flushleft}
\par{}As a result of the previous step we focus on estimating
\begin{equation}\label{TildePsiDefinition}
\int_{M\times\{t\}}\!{}\mathopen{}\left<\widetilde{\psi}_{1,t},
\mrmd[\tveep\times\mrmd\tveep]\chi\right>\mathclose{}.
\end{equation}
To achieve this, we proceed as in \cite{BOS2} and decompose
$\widetilde{\psi}_{1,t}$ by splitting $G$ into its the high and low frequency
parts.
For the low frequency term, $\widetilde{\psi}_{1,t}^{e}$, we will be interested in
establishing an $L^{\infty}$ estimate by appealing to the work of \cite{JS3}.
For the high frequency term, $\widetilde{\psi}_{1,t}^{i}$, we will be interested
in $L^{2}$ estimates in addition to an operator norm bound on the distributional
Laplacian of $G^{i}$, the integral kernel of $\widetilde{\psi}_{1,t}^{i}$.\\

Given $\alpha\in(2\slash3,3\slash4)$ and assuming that
$36r<\frac{\inj_{g}(M)}{2}$ we consider the function
$l\colon[0,\infty)\to[0,\infty)$ defined by
\begin{equation}\label{lFunction}
l(s)\ceq\begin{cases}
0& \text{if }s\le\ep^{\alpha}r\\
\Bigl(\Bigl(\frac{s}{\ep^{\alpha}r}\Bigr)^{N-1}-1\Bigr)\bigl(2^{N-1}-1\bigr)^{-1}&\text{if }\ep^{\alpha}r\le{}s\le{}2\ep^{\alpha}r\\
1& \text{if }2\ep^{\alpha}r\le{}s\le16r\\
\Bigl(2^{N-1}-\bigl(\frac{s}{16r}\bigr)^{N-1}\Bigr)\bigl(2^{N-1}-1\bigr)^{-1}& \text{if }
16r\le{}s\le{}32r\\
0& \text{if }s\ge{}32r.
\end{cases}
\end{equation}
From this we define $m\colon[0,\infty)\to[0,\infty)$ by
\begin{equation}\label{m:Function}
m(s)\ceq\begin{cases}
1-l(s)& \text{for }s\in[0,16r]\\
0& \text{for }s\in(16r,\infty)
\end{cases}
\end{equation}
and note that $m$ satisfies
\[
\begin{cases}
m(s)\equiv1& \text{for }s\in(0,\ep^{\alpha}r)\\
m(s)\equiv0& \text{for }s\in(2\ep^{\alpha}r,\infty)\\
|m'(s)|\le{}\frac{K}{\ep^{\alpha}r}.
\end{cases}
\]
Then we set
\begin{equation*}
G(x,y)=m(d(x,y))G(x,y)
+(1-m(d(x,y)))G(x,y)\eqqcolon{}G^{i}(x,y)+G^{e}(x,y).
\end{equation*}
This decomposition allows us to define
\begin{align}
\widetilde{\psi}_{1,t}^{i}&\ceq\int_{B_{2\ep^{\alpha}r}(x_{T})}\!{}
\mathopen{}\left<G^{i}(x,y),\mrmd[\tilde{v}_{\ep}\times\mrmd\tilde{v}_{\ep}]\chi\right>\mathclose{}
\label{Psi1tIJi:Definition}\\
\widetilde{\psi}_{1,t}^{e}&\ceq
\int_{B_{32r}(x)}\!{}
\mathopen{}\left<G^{e}(x,y),\mrmd[\tilde{v}_{\ep}\times\mrmd\tilde{v}_{\ep}]\chi\right>\mathclose{}.
\label{Psi1tIJe:Definition}
\end{align}
In the above $\widetilde{\psi}_{1,t}^{i}$ represents the high frequencies of $\widetilde{\psi}_{1,t}$ while $\widetilde{\psi}_{1,t}^{e}$ represents the low frequencies of $\widetilde{\psi}_{1,t}$.\\

Next, we note that by \eqref{GreensFunction:Properties1},
the definition of \eqref{m:Function},
and computations in normal coordinates we obtain
\begin{align}
\left\|\left\|G^{i}\right\|_{L_{x}^{1}(\wedge^{2}M)}\right\|_{L_{y}^{\infty}(\wedge^{2}M)},
\left\|\left\|G^{i}\right\|_{L_{y}^{1}(\wedge^{2}M)}\right\|_{L_{x}^{\infty}(\wedge^{2}M)}
&\le{}K_{M}\ep^{\alpha}r
\label{CompactGreenEstimate}\\
\left\|\left\|DG^{i}\right\|_{L_{x}^{1}(M)}\right\|
_{L_{y}^{\infty}(M)},
\left\|\left\|DG^{i}\right\|_{L_{y}^{1}(M)}\right\|
_{L_{x}^{\infty}(M)}&\le{}K_{M}\ep^{\alpha}r,
\label{GradientCompactGreenEstimate}
\end{align}
where $K_{M}$ is a constant depending only on $M$.
Similar computations in normal coordinates for the more delicate
estimate of $\widetilde{\psi}_{1,t}^{e}$ are presented in
detail below, see for example \eqref{CoordinateGreens}.
We refer the reader to Lemma A.4.9 of \cite{Col} for more details.

In addition, through direct computations related to the
distributional Laplacian of $G^{i}$ we obtain
\begin{equation}\label{LaplacianGreenEstimate}
\Bigl|\mathopen{}\left<G^{i}(\cdot,y),-\Delta{}h\right>\mathclose{}\Bigr|
\le{}K_{M}\left\|h\right\|_{L^{\infty}(\wedge^{2}M)}
\end{equation}
for all $h\in{}C^{2}\left(M;\wedge^{2}M\right)$ and each $y\in{}M$.
We refer the reader to Lemma A.4.9 of \cite{Col} for more details.
Estimates \eqref{CompactGreenEstimate} and \eqref{GradientCompactGreenEstimate}
along with the integral kernel expression for $\widetilde{\psi}_{1,t}^{i}$ allow us
to obtain
\beq\label{WidetildePsi1ti:L2}
\int_{M}|\widetilde{\psi}_{1,t}^{i}|^{2}
\le{}KC(\delta_{0},r)\ep^{2\alpha}
\Bigl(\widetilde{E}_{\ep,\gR}(v_{\ep},(x_{T},1),1)+C_{7}R_{1}E_{0}\Bigr)
\eeq
for $t\in\Theta_{1}$.
We refer the reader to Lemma A.4.11 of \cite{Col} for more details.\\

Finally, we use the decomposition of $\widetilde{\psi}_{1,t}$ in
\eqref{TildePsiDefinition} to conclude that it suffices to
estimate
\begin{align}
&\int_{M\times\{t\}}\!{}\mathopen{}\left<\widetilde{\psi}_{1,t}^{e},
\mrmd[\tveep\times\mrmd\tveep]\chi\right>\mathclose{}    \label{TildePsi1te:Pair}\\
&\int_{M\times\{t\}}\!{}\mathopen{}\left<\widetilde{\psi}_{1,t}^{i},
\mrmd[\tveep\times\mrmd\tveep]\chi\right>\mathclose{}.    \label{TildePsi1ti:Pair}
\end{align}

\begin{flushleft}
{\it Step $3$: Estimate of $\widetilde{\psi}_{1,t}^{e}$}\\
\end{flushleft}
\par{}We now focus on estimating the $L^{\infty}$ norm of
\begin{equation*}
\widetilde{\psi}_{1,t}^{e}=\int_{M}\!{}\mathopen{}\left<G^{e},
\mrmd[\tveep\times\mrmd\tveep]\chi\right>\mathclose{}.
\end{equation*}
Doing this will permit us to provide an upper bound on \eqref{TildePsi1te:Pair}.
We proceed in the same way as in \cite{BOS2} except we need to make use of
normal coordinates in order to have an explicit expression for the integrand.
The idea is to rewrite $\widetilde{\psi}_{1,t}^{e}$ into a distributional pairing
of the coordinate components of $\mrmd[\tveep\times\mrmd\tveep]$ and a Lipschitz
function.
Then, the work of \cite{JS3} is able to provide an $O(1)$ estimate for the
Lipschitz dual norm of the coordinate components of
$\mrmd[\tveep\times\mrmd\tveep]$.\\

We will estimate $\widetilde{\psi}_{1,t}^{e}$ at a fixed
$x\in {}B_{4r}(x_{T})$.
We let $y$ denote normal coordinates centered at $x$.
In these coordinates, of course $x$ corresponds to zero,
and if $p\in M$ is the point corresponding to the coordinate $y$,
then $d(x,p) = |y|$.
In the coordinates, $y$, we will write a $2$-form as
\beq\label{CoordinateForm}
\omega = \sum_{j_{1}\ne j_{2}}\omega_{j_{1} j_{2}}(y)
\mrmd{}y^{j_{1}}\wedge \mrmd{}y^{j_{2}}
\eqqcolon\sum_{J} \omega_{J}(y)\mrmd{}y^{J}.
\eeq
We recall that the Green's function $G$ is a tensor of type $(2,2)$ such that,
if the first and second components $z$ and $p$ are written respectively in
normal coordinates $\bar{y}$ and $y$ centred on $x$, then $G$ acts on $2$-forms
$\omega$ via
\[
\mathopen{}\left<G(z,p),\omega(p)\right>\mathclose{}
=\sum_{I}\biggl(\sum_{J}
G_{I}^{J}(\bar{y},y)\omega_{J}(y)\biggr)\mrmd\bar{y}^{I}.
\]
In particular,
\[
\mathopen{}\left<G(x,p),\omega(p)\right>\mathclose{}
=\sum_{I}\biggl(\sum_{J}
G_{I}^{J}(0,y)\omega_{J}(y)\biggr)\mrmd\bar{y}^{I}.
\]
In these coordinates, the Green's function for $2$-forms, evaluated with one argument fixed at $x$, has components
\[
G_{I}^{J}(0,y)=|y|^{2-N} H_{I}^{J}(y),
\qquad I = ( i_{1},i_{2}), \quad J = (j_{1},j_{2})
\]
where $H_{J}^{I}$ is a Lipschitz function in $y$ for each $I$ and $J$,
see \cite{BdR} and Proposition $4.12$ from \cite{Au},
$(\gR)_{ij}$ and $\gR^{ij}$ denote, respectively, the
metric tensor and its inverse with respect to these coordinates.
Due to our choice of $r$,
$G^{e}$ has support in the domain of this coordinate system, and so the above
discussion gives
\begin{align}\label{CoordinateGreens}
\widetilde{\psi}_{1,t}^{e}(x,t)&=
\sum_{I,J}\biggl[\int_{B_{32r}(0)}\!{}l(|y|)|y|^{2-N}
H_{I}^{J}(y)\frac{\pp\tveep}{\pp{}x_{j_{1}}}
\times\frac{\pp\tveep}{\pp{}x_{j_{2}}}\chi\sqrt{|\gR(y)|}\mrmd{}y
\biggr]\mrmd\bar{y}^{I}\\
&\eqqcolon\sum_{I}\biggl[\int_{B_{32r}(0)}\!{}
a_{I}(y)l(|y|)|y|^{2-N}\mrmd{}y\biggr]\mrmd\bar{y}^{I}    \nonumber
\end{align}
where we have set $|\gR(y)| = \det((\gR)_{ij}(y))$.
We now estimate each of the summands from \eqref{CoordinateGreens}.
Following the proof of Lemma $3.12$ of \cite{BOS2} and applying
Fubini's Theorem we obtain, for all $k\in{}L^{1}(M)$, that
\beq\label{IntegralRepresentationIdentity}
\int_{B_{32r}(0)}\!{}l(|y|)|y|^{2-N}k(y)\mrmd{}y
=\int_{\ep^{\alpha}r}^{16r}\!{}s^{-1}\mathcal{J}_{s}^{k}\mrmd{}s
+\frac{1}{N-2}\Bigl[\mathcal{J}_{16r}^{k}
-\mathcal{J}_{\ep^{\alpha}r}^{k}\Bigr]
\eeq
\[
\mathcal{J}_{s}^{k}\ceq{}s^{2-N}\int_{B_{2s}(0)}\!{}
k(y)h(|y|,s)\mrmd{}y
\]
and
\[
h(u,s)\ceq\frac{(N-1)(N-2)}{2^{N-1}-1}\cdot\begin{cases}
1& 0\le{}u\le{}s,\\
2-\frac{u}{s}& \text{if }s\le{}u\le{}2s,\\
0& \text{if }u\ge{}2s.
\end{cases}
\end{equation*}
We refer the reader to A.4.4.1 of \cite{Col} for more details.
We then use \eqref{IntegralRepresentationIdentity} with $k=a_{I}$.
As in \cite{BOS2}, we exploit the 
Jacobian structure of $\frac{\pp\tveep}{\pp{}x_{j_{1}}}
\times\frac{\pp\tveep}{\pp{}x_{j_{2}}}$ by applying Theorem $2.1$ of \cite{JS3} to
$\mathcal{J}_{s}^{a_{I}}$ to obtain
\beq\label{JacobianIntegralEstimate}
\sup_{s\in[\ep^{\alpha}r,16r]}
\bigl\{\mathcal{J}_{s}^{a_{I}}\bigr\}
\le{}C(\delta_{0},r)
\biggl(\frac{\Eepgr(\veep,(x_{T},1),1)+C_{7}R_{1}E_{0}}
{\mathopen{}\left|\log(\ep)\right|\mathclose{}}+\ep^{\beta}\biggr)
\eeq
for some $\beta>0$.
We refer the reader to Lemma A.4.10 of \cite{Col} for more details.
Combining \eqref{CoordinateGreens}, \eqref{IntegralRepresentationIdentity}, and
\eqref{JacobianIntegralEstimate} leads to
\beq\label{Psi1teEstimate}
\left\|\widetilde{\psi}_{1,t}^{e}\right\|_{L^{\infty}(\wedge^{2}M)}
\le{}C(\delta_{0},r)
\Bigl(\Eepgr(\veep,(x_{T},1),1)+C_{7}RE_{0}+1\Bigr).
\eeq
Observe that by \eqref{Psi1teEstimate}, \eqref{JacobianEstimate},
manipulations of Gaussian functions that use that the support of $\chi$ is
$B_{4r}(x_{T})$, in addition to using the assumption that $t\in\Theta_{1}$ we
have
\[
|\eqref{TildePsi1te:Pair}|
\le{}C(\delta_{0},r)
\Bigl(\Eepgr(\veep,(x_{T},1),1)+C_{7}RE_{0}+1\Bigr)
[RM_{0}+\eta].
\]
\newpage{}

\begin{flushleft}
{\it Step $4$: Auxiliary parabolic problem}\\
\end{flushleft}
\par{}Since we only have control over the $L^{1}$-norm of
$\mrmd[\tveep\times\mrmd\tveep]\chi$ then we fall slightly short of using
H\"older's inequality since the estimates \eqref{CompactGreenEstimate},
\eqref{GradientCompactGreenEstimate} only permit us to obtain the $L^{2}$
estimate \eqref{WidetildePsi1ti:L2}.
As a result, as in \cite{BOS2}, we introduce $\psi_{1}^{*}$ solving a parabolic
PDE to obtain better regularity through parabolic estimates.\\

We introduce $\psi_{1}^{*}$ solving
\beq\label{ComparisonPDE}
\begin{cases}
\pp_{t}\psi_{1}^{*}-\Delta\psi_{1}^{*}=\mrmd[\tveep\times\mrmd\tveep]\chi&
\text{on }M\times(0,\infty)\\
\psi_{1}^{*}(\cdot,0)=0&
\text{on }M\times\{0\}.
\end{cases}
\eeq
By appealing to standard parabolic techniques, the monotonicity formula, Gaffney's inequality, Lemma \ref{WeightedEnergyComparison}, as well
as Proposition \ref{HeatEstimate} and Proposition \ref{MonotonicityFormula}
we can show that
\begin{align}
\left\|\psi_{1}^{*}(\cdot,1-\delta^{2})\right\|_{L^{\infty}(\wedge^{2}M)}
&\le{}C(\delta_{0},r)
\Bigl(\Eepgr(\veep,(x_{T},1),1)+C_{7}R_{1}E_{0}\Bigr)    \label{Psi1StarL2}\\
\left\|D\psi_{1}^{*}\right\|_{L^{2}(\wedge^{2}(M\times[0,1-\delta_{0}^{2})])}^{2}
&\le{}C(\delta_{0},r)
\Bigl(\Eepgr(\veep,(x_{T},1),1)+C_{7}R_{1}E_{0}\Bigr).
\label{Psi1StarGradient}
\end{align}
We refer the reader to Lemma A.4.12 of \cite{Col} for more details.
Next, by using \eqref{PohozaevLocalizationFirstInequalityProof},
Proposition \ref{Averaging}, and Proposition \ref{MonotonicityFormula} as well as
its proof we obtain
\beq\label{vtEstimate}
\int_{M\times[0,1-\delta_{0}^{2}]}\!{}|\pp_{t}\veep|^{2}
e^{-\frac{(d_{+}(x,x_{T}))^{2}}{4(1-t)}}
\le{}K_{M}\Bigl(\Eepgr(\veep,(x_{T},1),1)+C_{7}R_{1}E_{0}\Bigr).
\eeq
We refer the reader to A.4.4.2 of \cite{Col} for additional details.
Finally, we argue that we can find $t\in[1-4\delta_{0}^{2},1-\delta_{0}^{2}]$ for which
\beq\label{ComparisonPsiAveraging}
\int_{M\times\{t\}}\!{}|\pp_{t}\psi_{1}^{*}|^{2}
\le{}C(\delta_{0},r)\ep^{-1}
\Bigl(\widetilde{E}_{\ep,\gR}(\veep,(x_{T},1),1)
+C_{7}R_{1}E_{0}\Bigr).
\eeq
We introduce the notation $\Theta_{2}$ to refer to
\beq\label{Theta2}
\Theta_{2}\ceq
\bigl\{t\in[1-4\delta_{0}^{2},1-\delta_{0}^{2}]:
\eqref{ComparisonPsiAveraging}\text{ holds at }t\bigr\}.
\eeq
To show \eqref{ComparisonPsiAveraging} we proceed as in \cite{BOS2}.
Taking the inner product of \eqref{ComparisonPDE} with
$\pp_{t}\psi_{1}^{*}$, integrating over
$M\times[0,1-\delta_{0}^{2}]$, and integrating by parts leads to
\beq\label{StartingPsi1starIdentity}
\int_{M\times[0,1-\delta_{0}^{2}]}\!{}|\pp_{t}\psi_{1}^{*}|^{2}
=-\frac{1}{2}\int_{M\times\{1-\delta_{0}^{2}\}}\!{}\bigl\{|\mrmd\psi_{1}^{*}|^{2}+|\mrmd^{*}\psi_{1}^{*}|^{2}\bigr\}
+\int_{M\times[0,1-\delta_{0}^{2}]}\!{}
\mathopen{}\left<\pp_{t}\psi_{1}^{*},\mrmd[\tveep\times\mrmd\tveep]\chi\right>\mathclose{}.
\eeq
Next introducing normal coordinates, $y$, centred at $x_{T}$
and expressing $\psi_{1}^{*}$ in these coordinates,
similar to \eqref{CoordinateForm}, as
\[
\psi_{1}^{*}=\sum_{I}\psi_{1,I}^{*}(y)\mrmd{}y^{I}
\]
we may write
$\mathopen{}\left<\mrmd[\tveep\times\mrmd\tveep]\chi,
\pp_{t}\psi_{1}^{*}\right>\mathclose{}$ as
\[
\sum_{I}\sum_{j_{1}<j_{2}}\gR^{IJ}(y)\pp_{t}\psi_{1,I}^{*}(y)
\bigl[\pp_{j_{1}}(\tveep(y)\times\pp_{j_{2}}\tveep(y))
-\pp_{j_{2}}(\tveep(y)\times\pp_{j_{1}}\tveep(y))\bigr]
\chi(y)\sqrt{|\gR(y)|}
\]
where $I=(i_{1},i_{2})$, $J=(j_{1},j_{2})$, $|\gR(y)|$ is as in
\eqref{CoordinateGreens}, and we have set
\[
\gR^{IJ}(y)\ceq
\det
\begin{pmatrix}
\gR^{i_{1}j_{1}}& \gR^{i_{1}j_{2}}\\
\gR^{i_{2}j_{1}}& \gR^{i_{2}j_{2}}
\end{pmatrix}
\]
where $\gR^{ij}$ denotes the $i,j$ component of the inverse of
metric tensor, $\gR$, in these coordinates.
By our choice of $\chi$ and $r>0$ we see that the support of
$\chi$ is contained in the domain of this coordinate system.
Setting $\chi_{\gR}^{IJ}\ceq\chi\gR^{IJ}\sqrt{|\gR|}$,
integrating by parts repeatedly as in \cite{BOS2},
and using that $\chi_{\gR}^{IJ}$ is supported on $B_{4r}(0)$
we can write
\[
\int_{M\times[0,1-\delta_{0}^{2}]}\!{}
\mathopen{}\left<\pp_{t}\psi_{1}^{*},\mrmd[\tveep\times\mrmd\tveep]\chi\right>\mathclose{}
=T_{1}+T_{2}+T_{3}+T_{4}
\]
where
\begin{align*}
T_{1}&\ceq2\sum_{I}\sum_{j_{1}<j_{2}}\int_{B_{4r}(0)\times[0,1-\delta_{0}^{2}]}\!{}
\bigl[\pp_{j_{1}}\psi_{1,I}^{*}(\pp_{t}\tveep\times\pp_{j_{2}}\tveep)
-\pp_{j_{2}}\psi_{1,I}^{*}(\pp_{t}\tveep\times\pp_{j_{1}}\tveep)\bigr]\chi_{\gR}^{IJ}\\
T_{2}&\ceq-\sum_{I}\sum_{j_{1}<j_{2}}\int_{B_{4r}(0)\times[0,1-\delta_{0}^{2}]}\!{}
(\tveep\times\pp_{t}\tveep)
\bigl[\pp_{j_{1}}\psi_{1,I}^{*}\pp_{j_{2}}\chi_{\gR}^{IJ}
-\pp_{j_{2}}\psi_{1,I}^{*}\pp_{j_{1}}\chi_{\gR}^{IJ}\bigr]\\
T_{3}&\ceq-\sum_{I}\sum_{j_{1}<j_{2}}\int_{B_{4r}(0)\times[0,1-\delta_{0}^{2}]}\!{}
\pp_{t}\psi_{1,I}^{*}
\bigl[\pp_{j_{1}}\chi_{\gR}^{IJ}(\tveep\times\pp_{j_{2}}\tveep)
-\pp_{j_{2}}\chi_{\gR}^{IJ}(\tveep\times\pp_{j_{1}}\tveep)\bigr]\\
T_{4}&\ceq-\sum_{I}\sum_{j_{1}<j_{2}}\int_{B_{4r}(0)\times\{1-\delta_{0}^{2}\}}\!{}
\bigl[\pp_{j_{1}}\psi_{1,I}^{*}(\tveep\times\pp_{j_{2}}\tveep)
-\pp_{j_{2}}\psi_{1,I}^{*}(\tveep\times\pp_{j_{1}}\tveep)\bigr]\chi_{\gR}^{IJ}.
\end{align*}
We estimate $T_{1}$, $T_{2}$, $T_{3}$, and $T_{4}$ as in
\cite{BOS2} by using
\eqref{Psi1StarL2}, \eqref{Psi1StarGradient}, \eqref{vepEstimates},
\eqref{vtEstimate}, and Proposition \ref{MonotonicityFormula}.
The only change required is in the estimate of $T_{4}$ in which
an additional appeal to Gaffney's inequality and $L^{2}$
estimates of $\psi_{1}^{*}$ obtained from the proof of
\eqref{Psi1StarGradient} are applied.
Proceeding in this way we obtain
\beq\label{Averagedpsi1stt}
\int_{M\times[0,1-\delta_{0}^{2}]}\!{}
|\pp_{t}\psi_{1}^{*}|^{2}
\le{}C(\delta_{0},r)\ep^{-1}
\biggl(\widetilde{E}_{\ep,\gR}(\veep,(x_{T},1),1)
+C_{7}R_{1}E_{0}\biggr).
\eeq
An application of Chebyshev's inequality then allows us
to find $t\in[1-4\delta_{0}^{2},1-\delta_{0}^{2}]$ for
which \eqref{ComparisonPsiAveraging} holds.\\

\begin{flushleft}
{\it Step $5$: Estimate of $\widetilde{\psi}_{1,t}^{i}$}\\
\end{flushleft}
\par{}We assume that $t\in\Theta_{1}\cap\Theta_{2}$.
Using \eqref{ComparisonPDE} we can write
\[
\int_{M\times\{t\}}\!{}\mathopen{}\left<\widetilde{\psi}_{1,t}^{i}
,\mrmd[\tveep\times\mrmd\tveep]\chi\right>\mathclose{}   
=\int_{M}\!{}\mathopen{}\left<\widetilde{\psi}_{1,t}^{i},
\pp_{t}\psi_{1}^{*}\right>\mathclose{}
+\int_{M}\!{}\mathopen{}\left<\widetilde{\psi}_{1,t}^{i}
,-\Delta\psi_{1}^{*}\right>\mathclose{}.
\]
The first term can be estimated using \eqref{WidetildePsi1ti:L2} and
\eqref{ComparisonPsiAveraging} to obtain
\beq\label{Psi1teIntegralEstimate}
\biggl|\int_{M\times\{t\}}\!{}
\mathopen{}\left<\widetilde{\psi}_{1,t}^{i},
\pp_{t}\psi_{1}^{*}\right>\mathclose{}\biggr|
\le{}C(\delta_{0},r)\ep^{\alpha-\frac{1}{2}}
\Bigl(\Eepgr(\veep,(x_{T},1),1)+C_{7}R_{1}E_{0}\Bigr).
\eeq
The second term can be estimated using \eqref{LaplacianGreenEstimate},
the proof of Proposition \ref{HeatEstimate}, and \eqref{WeightedVEnergy} to obtain
\beq\label{Psi1tiIntegralEstimate}
\biggl|\int_{M\times\{t\}}\!{}
\mathopen{}\left<\widetilde{\psi}_{1,t}^{i},
-\Delta\psi_{1}^{*}\right>\mathclose{}\biggr|
\le{}C(\delta_{0},r)
\Bigl(\Eepgr(\veep,(x_{T},1),1)+C_{7}R_{1}E_{0}\Bigr)
[RM_{0}+\eta].
\eeq
Combining estimates \eqref{Psi2tEstimate}, \eqref{KernelRepresentation},
\eqref{KernelHarmonicEstimate}, \eqref{Psi1teIntegralEstimate}, and
\eqref{Psi1tiIntegralEstimate} with \eqref{Psi1tEstimate} gives
\begin{align}\label{Psi1tFinalEstimate}
\int_{M\times\{t\}}\!{}&\bigl\{|\mrmd\psi_{1,t}|^{2}+|\mrmd^{*}\psi_{1,t}|^{2}\bigr\}
\le{}C(\delta_{0},r)\ep^{\alpha-\frac{1}{2}}
\Bigl(\Eepgr(\veep,(x_{T},1),1)+C_{7}R_{1}E_{0}\Bigr)\\
&+C(\delta_{0},r)
\Bigl(\Eepgr(\veep,(x_{T},1),1)+C_{7}R_{1}E_{0}+1\Bigr)
(RM_{0}+\eta+[RM_{0}+\eta]^{2}).    \nonumber
\end{align}
Finally, combining \eqref{XiEstimate} \eqref{vphiEstimate}, \eqref{Psi2tEstimate},
\eqref{psi1treduction}, the estimate of \eqref{TildePsi1te:Pair},
and \eqref{Psi1tFinalEstimate} and choosing the parameters
sufficiently small completes the proof of Proposition
\ref{Theorem1Reduction}.\newline

\section{Energy Decompositions}\label{Sec::Decompositions}

\hspace{15pt}In this section we present the proof of Theorem \ref{Theorem3}.
Compared to \cite{BOS2}, there are 
new considerations related to the homology of $M$.
More specifically, when applying the Hodge de Rham
decomposition we must, since we impose no homological restrictions
on $M$, consider the harmonic part.
In particular, these considerations are responsible for the
presence of $u_{h,\vep}$ in the conclusions of
Theorem \ref{Theorem3}.

We start by stating a local energy decomposition for solutions of
\ref{PGLOriginal}, valid in a region where the modulus is bounded
away from zero.

\begin{theorem}\label{Theorem2}
Suppose that $0<R<\inj_{g}(M)$, $T>0$, and $\Delta{}T>0$ are given.
Consider the cylinder
\[
\Lambda\ceq{}B_{R}(x_{0})\times[T,T+\Delta{}T].
\]
There exists a constant $0<\sigma\le\frac{1}{2}$ and $\beta>0$ depending only on $N$, such that if
\begin{equation}\label{LowerBoundAssumption}
|\uvep|\ge1-\sigma\hspace{10pt}\text{on }\Lambda,
\end{equation}
then
\beq\label{UniformEstimate}
\evep(\uvep)(x,t)\le{}C(\Lambda)\int_{\Lambda}\!{}\evep(\uvep),
\eeq
for any
$(x,t)\in\Lambda_{\frac{1}{2}}$.
Moreover,
\begin{equation}\label{EnergyMeasureDecomposition}
\evep(\uvep)=\frac{|\nabla\Phi_{\vep}|^{2}}{2}+\kappa_{\vep}\hspace{10pt}\text{in }\Lambda_{\frac{1}{2}},
\end{equation}
where the functions $\Phi_{\vep}$ and $\kappa_{\vep}$ are defined on $\Lambda_{\frac{1}{2}}$ and verify
\begin{align}\label{EnergyDecompositionHeatEquation}
\pp_{t}\Phi_{\vep}-\Delta\Phi_{\vep}=0&\hspace{10pt}\text{in }\Lambda_{\frac{1}{2}},\\
\|\kappa_{\vep}\|_{L^{\infty}\bigl(\Lambda_{\frac{1}{2}}\bigr)}\le{}C(\Lambda)\vep^{\beta},&\hspace{10pt}
\|\nabla\Phi_{\vep}\|_{L^{\infty}\bigl(\Lambda_{\frac{1}{2}}\bigr)}^{2}\le{}C(\Lambda)M_{0}\mathopen{}\left|\log(\vep)\right|\mathclose{}.
\label{DecompositionEstimates}
\end{align}
In addition, it follows from our choice of $\Phi_{\vep}$ that if $\uvep=\rho_{\vep}e^{i\vphi_{\vep}}$ on $\Lambda$ then
\begin{equation}\label{ArgumentConvergence}
\left\|\nabla\Phi_{\vep}-\nabla\vphi_{\vep}\right\|_{L^{\infty}(\Lambda_{\frac{1}{2}})}\le{}C(\Lambda)\vep^{\beta}.
\end{equation}
\end{theorem}

This is an adaptation to the present setting of Theorem $2$ in
\cite{BOS2}. Since the analysis is entirely local,
and because it does not involve any delicate properties of 
test functions adapted to the metric, the proof ends up being
essentially identical in the Riemannian case.
This being the case, we omit all details here.
An interested reader can consult \cite{BOS2}, or A.6.0.1 of \cite{Col},
where it is verified in detail that the arguments of \cite{BOS2}
remain valid on a manifold.\\

As was done in \cite{BOS2}, we record a straightforward
consequence obtained by combining the results of Theorem
\ref{Theorem1} with Proposition
\ref{ManifoldEnergyLocalizedToBall}.

\begin{proposition}\label{Proposition2}
Let $\uvep$ be a solution of \ref{PGLOriginal} verifying assumption \eqref{H0} and $\sigma>0$ be given.
Let $x_{T}\in{}M$, $T>0$,
and $0<2\vep<R^{2}<R(\sigma)$ where $R(\sigma)$ is as in Theorem \ref{Theorem1}.
There exists a positive continuous function $\lambda$ defined on $(0,\infty)$ such that, if
\[
\check{\eta}(x_{T},T,R)\ceq\frac{1}{(4\pi)^{\frac{N}{2}}R^{N-2}\mathopen{}\left|\log(\vep)\right|\mathclose{}}
\int_{B_{\lambda(T)R}(x_{T})}\!{}\evep(\uvep(\cdot,T))\le\frac{\eta_{1}(\sigma)}{2}
\]
then
\[
|\uvep(x,t)|\ge1-\sigma\hspace{10pt}\text{for }t\in[T+T_{0},T+T_{1}]\hspace{5pt}\text{and }x\in{}B_{\frac{R}{2}}(x_{T}).
\]
Here $T_{0}$ and $T_{1}$ are defined by
\[
T_{0}\ceq\mathopen{}\left(\frac{2\check{\eta}}{\eta_{1}}\right)^{\frac{2}{N-2}}\mathclose{}R^{2},\hspace{10pt}T_{1}\ceq{}R^{2}.
\]
In particular, a more precise estimate shows that we can find $\lambda$ defined on $(0,\infty)\times(0,\infty)$ satisfying
\[
\lambda(T,R)\sim\sqrt{\Biggl|\frac{8}{c_{*}^{2}}\log\Biggl(\frac{(4\pi)^{\frac{N}{2}}}{M_{0}e^{C_{2}}}\biggl[\frac{2}{T+2R^{2}}\biggr]^{\frac{N-2}{2}}\Biggr)\Biggr|}
\]
as $(T,R)\to(0,0)$.
In particular, $\lambda(T,R)R$ is bounded as $R\to0^{+}$ for any $T>0$.
\end{proposition}

Following \cite{BOS2} we also record the following consequence of
Proposition \ref{Proposition2} combined with Theorems
\ref{Theorem1} and \ref{Theorem2} for future use.

\begin{proposition}\label{Proposition4}
For each $\sigma>0$ there exists positive constants $\eta_{2}(\sigma)$ and $R(\sigma)$ as well as a positive function $\lambda$
defined on $(0,\infty)$ such that if, for $x\in{}M$, $t>0$, and $\sqrt{2\vep}<r<R(\sigma)$ we have
\begin{equation*}
\int_{B_{\lambda(t)r}(x)}\!{}e_{\vep}(u_{\vep})\le\eta_{2}r^{N-2}\mathopen{}\left|\log(\vep)\right|\mathclose{},
\end{equation*}
then
\begin{equation*}
e_{\vep}(u_{\vep})=\frac{|\nabla\Phi_{\vep}|^{2}}{2}+\kappa_{\vep}
\end{equation*}
in $\Lambda_{\frac{1}{4}}(x,t,r)\ceq{}B_{\frac{r}{4}}(x)\times\bigl[t+\frac{15}{16}r^{2},t+r^{2}\bigr]$, where $\Phi_{\vep}$ and
$\kappa_{\vep}$ are as in Theorem \ref{Theorem2}.
In particular,
\begin{equation*}
\mu_{\vep}=\frac{e_{\vep}(u_{\vep})}{\mathopen{}\left|\log(\vep)\right|\mathclose{}}\le{}C(t,r)\hspace{10pt}\text{on }\Lambda_{\frac{1}{4}}(x,t,r).
\end{equation*}
\end{proposition}

We use the remainder of this section to prove Theorem \ref{Theorem3}.
We begin by introducing some notation.
We let $\Omega\ceq{}M\times(t_{1},t_{2})$, where $0<t_{1}<t_{2}<\infty$
and use $\delta$ to denote the exterior derivative on
$M\times(0,\infty)$.
In addition we let $\delta^{*}$ denote its formal
adjoint with respect to the natural product metric.
If 
$\eta$ is a $k$-form on $M\times (0,\infty)$ and $\Sigma$ is a smooth hypersurface,
we will write $\eta_T$ to denote the $k$-form on $\Sigma$ defined by
\[
\eta_T \ceq i^*\eta = \mbox{ the tangential part of $\eta$  on $\Sigma$},
\]
where $i\colon\Sigma\to M\times(0,\infty)$ is the inclusion map. We also write
\[
\eta_N := \eta - \eta_T = \mbox{ the normal part of $\eta$ on $\Sigma$.}
\]

We note that if $\uvep$ solves \ref{PGLOriginal}
and satisfies \eqref{H0} then standard parabolic
estimates give, for sufficiently small $\vep$,
that
\begin{align}
\int_{M\times\{t\}}\!{}\evep(\uvep)&\le{}M_{0}
\mathopen{}\left|\log(\vep)\right|\mathclose{}\hspace{10pt}
\forall{}t>0,  \label{EnergyLogarithmicBoundt}\\
|\uvep(x,t)|&\le3\hspace{50pt}
\forall(x,t)\in\Omega.    \label{UniformUpperBound}
\end{align}
In particular, \eqref{EnergyLogarithmicBoundt} allows us to
conclude that
\begin{align}
\int_{M\times[0,t_{2}]}\!{}\evep(\uvep)&\le{}
C(\Omega)M_{0}\mathopen{}\left|\log(\vep)\right|\mathclose{}
\label{EnergyLogarithmicBound}\\
\int_{\pp\Omega}\!{}\evep(\uvep)&\le{}
2M_{0}\mathopen{}\left|\log(\vep)\right|\mathclose{}.
\label{BoundaryEnergyLogarithmicBound}
\end{align}
The next result is the main decomposition tool used in the proof of Theorem
\ref{Theorem3}.
\begin{proposition}\label{ControlledHodgeDeRham}
Assume that $\uvep$ is a solution to \ref{PGLOriginal} on $M\times(0,\infty)$ that
satisfies \eqref{H0}.
Then there is a smooth $1$-form $\gamma$ dependent only on the initial data of
$\uvep$ such that, on $\Omega$, there exists a smooth function $\Phi$, a smooth
$1$-form $\zeta$, and a smooth $2$-form $\Psi$ for which
\beq\label{ControlledHodgeDeRhamDecomposition}
\uvep\times\delta\uvep=\delta\Phi+\delta^{*}\Psi+\gamma+\zeta,
\hspace{10pt}
\delta\Psi=0\text{ in }\Omega,
\hspace{10pt}
\Psi_{T}=0\text{ on }\pp\Omega,
\eeq
and
\beq\label{LogarithmicControlGradientHodgeDeRhamDecomposition}
\left\|\Phi\right\|_{W^{1,2}\left(\Omega\right)}+\left\|D\Psi\right\|_{L^{2}\left(\Omega\right)}
+\left\|\gamma\right\|_{L^{2}\left(\Sigma\right)}\le{}C(\Omega)
\sqrt{(M_{0}+1)\left|\log(\vep)\right|}.
\eeq
In addition, we have that $\gamma$ is constant in time, independent of $\mrmd{}t$,
a harmonic $1$-form on $M$ for all $t>0$, and there is a time independent
$\bbS^{1}$-valued function $u_{h,\vep}$ such that $ju_{h,\vep}=\gamma$.
Moreover, for any $1\le{}p<\frac{N+1}{N}$,
\beq\label{DelicateLpEstimates}
\begin{cases}
\left\|D\Psi\right\|_{L^{p}(\Omega)}\le{}C(p,\Omega)(M_{0}+1),&\\
\left\|\zeta\right\|_{L^{p}(\Omega)}\le{}C(p,\Omega)(M_{0}+1)\vep^{\frac{1}{2}},
\end{cases}
\eeq
where $C(p,\Omega)$ is a constant depending only on $p$ and $\Omega$.
\end{proposition}
\begin{proof}
As in \cite{BOS2} we split the proof into two steps.
We begin by dealing with $\Sigma\ceq\pp\Omega$.
Notice that $\Sigma=\left(M\times\left\{t_{1}\right\}\right)\sqcup\left(M\times\left\{t_{2}\right\}\right)$.\\

\emph{Step $1$: \emph{HdR} decompositions on $\Sigma$.}
Since $\pp\Sigma=\varnothing$ then a standard Hodge-de Rham decomposition 
applied to the tangential part of $\uvep\times\delta\uvep$ 
allows us to write
\beq\label{BoundaryHodgeDecompositionSigma}
(\uvep\times\delta\uvep)_{T}=\uvep\times\mrmd\uvep
=\mrmd\Phi_{\vep}^{i}+\mrmd^{*}\Psi_{\vep}^{i}+\gamma_{\vep}^{i}
\hspace{5pt}\text{on }M\times\{t_{i}\},
\eeq
for $i=1,2$, with 
\begin{align}
\mrmd\Psi_{\vep}^{i}=0&\hspace{5pt}\text{ on }M\times\{t_{i}\}\text{ for }i=1,2, \label{PsiepsiloniClosed}\\
\mrmd\gamma_{\vep}^{i}=0=\mrmd^{*}\gamma_{\vep}^{i}&\hspace{5pt}\text{ on }M\times\{t_{i}\}
\text{ for }i=1,2,\\
\int_{M\times\{t_{i}\}}\!{}\Phi_{\vep}^{i}\dvol_{g}=0
&\hspace{5pt}\text{ for }i=1,2 .  \label{ZeroAverage}
\end{align}
See for example Theorem $5$ of Section $5.2.5$ of \cite{GMS}, which also shows that 
\beq\label{LogarithmicBound}
\left\|\Phi_{\vep}^{i}\right\|_{W^{1,2}(M\times\{t_{i}\})}^{2}
+\left\|\Psi_{\vep}^{i}\right\|_{W^{1,2}(M\times\{t_{i}\})}^{2}
+\left\|\gamma_{\vep}^{i}\right\|_{L^{2}(M\times\{t_{i}\})}^{2}\le{}KM_{0}\mathopen{}\left|\log(\vep)\right|\mathclose{}
\eeq
for $i=1,2$.
Next observe that by applying $\mrmd$ to \eqref{BoundaryHodgeDecompositionSigma} at $t=t_{i}$ for $i=1,2$ and using \eqref{PsiepsiloniClosed} we obtain
\beq\label{JacobianHarmonicPDE}
-\Delta_{M}\Psi_{\vep}^{i}=J_{M}\uvep\hspace{10pt}\text{on }M\times\{t_{i}\}
\eeq
for $i=1,2$ where $J_{M}\uvep\ceq\frac{1}{2}\mrmd[\uvep\times\mrmd\uvep]$ and
$\Delta_{M}$ is the Laplacian on $M$.
Thus, by Theorem $2.1$ of \cite{JS3}, the Sobolev
Embedding Theorem, duality, and elliptic regularity, see Lemma $2.9$ of \cite{BBO} and
Propositions $5.17$ and $6.5$ of \cite{ISS}, we have that for all
$q>N$, $p\ceq\frac{q}{q-1}$, and $\alpha\ceq1-\frac{N}{q}$ that
\begin{align}\label{W1pBoundForPsi}
\left\|\Psi_{\vep}^{i}\right\|_{W^{1,p}(M\times\{t_{i}\})}
&\le{}C(p,M)\left\|J_{M}\uvep\right\|_{W^{-1,p}(M\times\{t_{i}\})}\\
&\le{}C(p,M)\left\|J_{M}\uvep\right\|_{\left[C^{0,\alpha}(M\times\left\{t_{i}\right\})\right]^{*}}
\le{}C(p,M)(M_{0}+1)  \nonumber
\end{align}
for $i=1,2$.\\

Next we provide an approximation to the harmonic parts from \eqref{BoundaryHodgeDecompositionSigma} that stores most of the energy.
This is a new ingredient needed to extend the corresponding result of \cite{BOS2}
to our setting.
As a result, we go over the associated estimates in more detail.\\

We consider a collection of closed curves,
$\{c_{j}\}_{j=1}^{\beta_{1}(M)}$ where $\beta_{1}(M)$ is the
first Betti number of $M$,
generating the first homology group $H_{1}(M)$.
It follows from item $(ii)$ of Theorem $4$ of Section $5.3.2$ and Theorem $6$ of Section $5.2.5$ of \cite{GMS} that, associated to these curves, we can find
a basis $\{c^{k}\}_{k=1}^{\beta_{1}(M)}$ for $H^{1}(M)$, the space of harmonic $1$-forms on $M$,
satisfying
\[
\int_{c_{j}}\!{}c^{k}=2\pi\delta_{jk}.
\]
Using this basis we can express $\gamma_{\vep}^{0}$, the harmonic part of $\uvep\times\mrmd\uvep$ at $t=0$, as
\beq\label{HarmonicRepresentation}
\gamma_{\vep}^{0}=\sum_{k=1}^{\beta_{1}(M)}a_{k}^{0}(\vep)c^{k}
\eeq
where $a_{k}^{0}(\vep)$ may depend on $\vep$.
From the representation \eqref{HarmonicRepresentation} we may define
\beq\label{HarmonicFormApproximation}
\lfloor\gamma_{\vep}^{0}\rfloor\ceq\sum_{k=1}^{\beta_{1}(M)}\lfloor{}a_{k}^{0}(\vep)\rfloor{}c^{k}
\eeq
which is a harmonic $1$-form on $M$.
Notice that we may extend \eqref{HarmonicFormApproximation} to $M\times(0,\infty)$, in particular to $\Omega$, by being constant in time to obtain
\beq\label{HarmonicFormApproximationExtension}
\gamma_{\vep}(x,t)\ceq\lfloor\gamma_{\vep}^{0}\rfloor(x).
\eeq
We observe that this extension has no term corresponding to $\mrmd{}t$.
We also note that by construction we have
\beq\label{EstimateAtZero}
\left\|\gamma_{\vep}^{0}-\lfloor\gamma_{\vep}^{0}\rfloor\right\|_{L^{2}(\wedge^{1}M)}
=\left\|\sum_{k=1}^{\beta_{1}(M)}\bigl(a_{k}^{0}(\vep)-\lfloor{}a_{k}^{0}(\vep)\rfloor\bigr)c^{k}\right\|_{L^{2}(\wedge^{1}M)}
\le\sum_{k=1}^{\beta_{1}(M)}\left\|c^{k}\right\|_{L^{2}(\wedge^{1}M)}
\eeq
where the rightmost quantity is not dependent on $\vep$.
Next we establish that $\gamma_{\vep}^{1}$ and $\gamma_{\vep}^{2}$ are not too far from $\lfloor\gamma_{\vep}^{0}\rfloor$.
We only demonstrate this for $\gamma_{\vep}^{1}$ as the proof is similar for $\gamma_{\vep}^{2}$.
To do this we first extend $\gamma_{\vep}^{0},\gamma_{\vep}^{1}$ to $M\times(0,\infty)$, in particular $\Omega$,
by being constant in time and, respectively, use
$\Gamma_{\vep}^{0},\Gamma_{\vep}^{1}$ to denote this.
Next we define the $2$-form $\eta$ by
\beq\label{ComparisonHarmonicForm}
\eta\ceq\bigl(\Gamma_{\vep}^{1}-\Gamma_{\vep}^{0}\bigr)\wedge\mrmd{}t.
\eeq
Observe that
\[
\delta^{*}\eta=-\bigl[\mrmd^{*}\bigl(\Gamma_{\vep}^{1}-\Gamma_{\vep}^{0}\bigr)\bigr]\wedge\mrmd{}t=0
\]
where we have identified $\Gamma_{\vep}^{1}-\Gamma_{\vep}^{0}$ with an element of $H^{1}(M)$ since this $1$-form is
independent of $t$ and $\mrmd{}t$.
From this computation we can now see that, after integrating by parts, we obtain
\begin{align}\label{HarmonicDifferenceEstimate}
2\int_{M\times[0,t_{1}]}\!{}\left<J\uvep,\eta\right>_{M\times[0,t_{1}]}
&=\int_{M\times[0,t_{1}]}\!{}\left<\delta\left(\uvep\times\delta\uvep\right),\eta\right>_{M\times[0,t_{1}]}\\
&=\int_{M\times\{t_{1}\}}\!{}(\uvep\times\mrmd{}\uvep)\wedge\star\eta_{N}
+\int_{M\times\{0\}}\!{}(\uvep\times\mrmd\uvep)\wedge\star\eta_{N}  \nonumber
\end{align}
where $J\uvep\ceq\frac{1}{2}\delta[\uvep\times\delta\uvep]$.
By noting that $\eta_{N}=(\gamma_{\vep}^{1}-\gamma_{\vep}^{0})\wedge\mrmd{}t$ at $M\times\{t_{1}\}$ and
$\eta_{N}=-(\gamma_{\vep}^{1}-\gamma_{\vep}^{0})\wedge\mrmd{}t$ at $M\times\{0\}$ we can rewrite this last
expression as
\begin{align}\label{HarmonicRewriting}
&\int_{M\times\{t_{1}\}}\!{}(\uvep\times\mrmd\uvep)\wedge\star\eta_{N}
+\int_{M\times\{0\}}\!{}(\uvep\times\mrmd\uvep)\wedge\star\eta_{N}\\
=&(-1)^{N-1}\biggl[\int_{M\times\{t_{1}\}}\!{}
\left<\uvep\times\mrmd\uvep,\gamma_{\vep}^{1}-\gamma_{\vep}^{0}\right>_{M}
-\int_{M\times\{0\}}\!{}
\left<\uvep\times\mrmd\uvep,\gamma_{\vep}^{1}-\gamma_{\vep}^{0}\right>_{M}\biggr] \nonumber\\
=&(-1)^{N-1}\biggl[\int_{M}\!{}
\left<\gamma_{\vep}^{1},\gamma_{\vep}^{1}-\gamma_{\vep}^{0}\right>_{M}
-\int_{M}\!{}
\left<\gamma_{\vep}^{0},\gamma_{\vep}^{1}-\gamma_{\vep}^{0}\right>_{M}\biggr] \nonumber\\
=&(-1)^{N-1}\int_{M}\!{}\left|\gamma_{\vep}^{1}-\gamma_{\vep}^{0}\right|_{M}^{2}  \nonumber
\end{align}
where we have used \eqref{BoundaryHodgeDecompositionSigma} in the third line.
Putting \eqref{HarmonicDifferenceEstimate} and \eqref{HarmonicRewriting} together and using the Jerrard-Soner estimate, Theorem $2.1$ of
\cite{JS3}, together with equivalence of norms on $H^{1}(M)$ gives
\begin{align*}
\left\|\gamma_{\vep}^{1}-\gamma_{\vep}^{0}\right\|_{L^{2}(\wedge^{1}M)}^{2}
&=2\biggl|\int_{M\times[0,t_{1}]}\!{}\left<J\uvep,\eta\right>_{M\times[0,t_{1}]}\biggr|
\le2\left\|J\uvep\right\|
_{\left(C^{0,\alpha}(\wedge^{2}M\times[0,t_{1}])\right)^{*}}\left\|\eta\right\|_{C^{0,\alpha}(\wedge^{2}M\times[0,t_{1}])}\\
&\le{}C(\alpha,\Omega)\left\|\eta\right\|_{L^{2}(\wedge^{2}M\times[0,t_{1}])}
\Biggl[\frac{\int_{M\times[0,t_{1}]}\!{}\evep(\uvep)}{{\mathopen{}\left|\log(\vep)\right|\mathclose{}}}
+1\Biggr]    \nonumber\\
&=C(\alpha,\Omega)\left\|\gamma_{\vep}^{1}-\gamma_{\vep}^{0}\right\|_{L^{2}(\wedge^{1}M)}
\Biggl[\frac{\int_{M\times[0,t_{1}]}\!{}\evep(\uvep)}{{\mathopen{}\left|\log(\vep)\right|\mathclose{}}}
+1\Biggr]    \nonumber\\
&\le{}C(\alpha,\Omega)(M_{0}+1)\left\|\gamma_{\vep}^{1}-\gamma_{\vep}^{0}\right\|_{L^{2}(\wedge^{1}M)}.  \nonumber
\end{align*}
Thus, we obtain
\beq\label{InitialTimeEstimate}
\left\|\gamma_{\vep}^{1}-\gamma_{\vep}^{0}\right\|_{L^{2}(\wedge^{1}M)}\le{}C(\alpha,\Omega)(M_{0}+1).
\eeq
It then follows from \eqref{EstimateAtZero} and \eqref{InitialTimeEstimate} that for $i=1,2$
\beq\label{FinalTimeEstimate}
\left\|\gamma_{\vep}^{i}-\lfloor\gamma_{\vep}^{0}\rfloor\right\|_{L^{2}(\wedge^{1}M)}
\le{}C(\alpha,\Omega)(M_{0}+1).
\eeq
Next we notice that since the integral of $\gamma_{\vep}$ over every closed loop in $M$ has a value in $2\pi\Z$ then there
exists $u_{h,\vep}\colon{}M\to\bbS^{1}$ such that
\beq\label{HarmonicCurrent1}
u_{h,\vep}\times\mrmd{}u_{h,\vep}=\gamma_{\vep}.
\eeq
We may extend $u_{h,\vep}$ to be constant in time to obtain $u_{h,\vep}\colon\Omega\to\bbS^{1}$ such that
\beq\label{HarmonicCurrent2}
u_{h,\vep}\times\delta{}u_{h,\vep}=\gamma_{\vep},
\eeq
$u_{h,\vep}$ is independent of $t$ and $u_{h,\vep}\times\delta{}u_{h,\vep}$ is independent of $\mrmd{}t$.
We also consider the linear extension $\Phi_{\vep}^{1,2}$ of $\Phi_{\vep}^{1}$ to $\Phi_{\vep}^{2}$ in $\Omega$ defined by
\[
\Phi_{\vep}^{1,2}(x,t)\ceq\biggl(\frac{t_{2}-t}{t_{2}-t_{1}}\biggr)\Phi_{\vep}^{1}(x)
+\biggl(\frac{t-t_{1}}{t_{2}-t_{1}}\biggr)\Phi_{\vep}^{2}(x).
\]
Note that by \eqref{LogarithmicBound} this extension satisfies
\beq\label{LogarithmicBoundOnExtension}
\left\|\Phi_{\vep}^{1,2}\right\|_{W^{1,2}(\Omega)}\le{}K(\Omega)
\sqrt{M_{0}\mathopen{}\left|\log(\vep)\right|\mathclose{}}.
\eeq

\emph{Step $2$: ``Gauge transformation'' of $\uvep$.}
On $\Omega$ we consider the map $w_{\vep}$ defined by
\[
w_{\vep}\ceq{}\uvep{}e^{-i\Phi_{\vep}^{1,2}}\overline{u_{h,\vep}}\hspace{10pt}\text{in }\Omega.
\]
Notice that $|w_{\vep}|=|\uvep|$.
Moreover, one can show
\beq\label{TransferenceIdentity}
w_{\vep}\times\delta{}w_{\vep}
=\uvep\times\delta{}\uvep-\delta\Phi_{\vep}^{1,2}-\gamma_{\vep}
+(1-|\uvep|^{2})(\delta\Phi_{\vep}^{1,2}+\gamma_{\vep}).
\eeq
Since $|\uvep|\le3$ then
\beq\label{GradientOfModifiedFunction}
|\nabla_{x}w_{\vep}|\le|\nabla_{x}\uvep|+3|\nabla_{x}\Phi_{\vep}^{1,2}|
+3|\gamma_{\vep}|
\eeq
and hence
\beq\label{LogarithmicEnergyBoundModifiedFunction}
\left\|\nabla{}w_{\vep}\right\|_{L^{2}(\Omega)}^{2}
+\vep^{-2}\left\|1-|w_{\vep}|^{2}\right\|_{L^{2}(\Omega)}^{2}
\le{}KM_{0}\mathopen{}\left|\log(\vep)\right|\mathclose{}.
\eeq
By H\"{o}lder's inequality, \eqref{UniformUpperBound}, \eqref{EnergyLogarithmicBound}, \eqref{LogarithmicBound}, and
\eqref{LogarithmicBoundOnExtension} we have that for $1\le{}p<2$
\begin{align}
\left\|(1-|\uvep|^{2})\delta\Phi_{\vep}^{1,2}\right\|_{L^{p}(\Omega)}^{p}
&\le{}K(\Omega)M_{0}\vep^{2-p}\mathopen{}\left|\log(\vep)\right|\mathclose{}  \label{SmallErrorTermEstimates1}\\
\left\|(1-|\uvep|^{2})\gamma_{\vep}\right\|_{L^{p}(\Omega)}^{p}
&\le{}K(\Omega)M_{0}\vep^{2-p}\mathopen{}\left|\log(\vep)\right|\mathclose{}  \label{SmallErrorTermEstimates2}
\end{align}
and similarly
\begin{align}
\left\|(1-|\uvep|^{2})\mrmd\Phi_{\vep}^{i}\right\|_{L^{p}(M\times\{t_{i}\})}^{p}
&\le{}K(\Omega)M_{0}\vep^{2-p}\mathopen{}\left|\log(\vep)\right|\mathclose{}  \label{SmallErrorBoundaryEstimates1}\\
\left\|(1-|\uvep|^{2})\gamma_{\vep}^{i}\right\|_{L^{p}(M\times\{t_{i}\})}^{p}
&\le{}K(\Omega)M_{0}\vep^{2-p}\mathopen{}\left|\log(\vep)\right|\mathclose{}  \label{SmallErrorBoundaryEstimates2}
\end{align}
for $i=1,2$.
Next, by Corollary $5.6$ of \cite{ISS}, we have the following Hodge decomposition of $w_{\vep}\times\delta{}w_{\vep}$ on $\Omega$:
\beq\label{HodgeDecompositionOverOmega}
\begin{cases}
w_{\vep}\times\delta{}w_{\vep}=\delta\overline{\Phi}_{\vep}+\delta^{*}\Psi_{\vep}+\eta& \text{in }\Omega,\\
\delta\Psi_{\vep}=0& \text{in }\Omega,\\
\left(\overline{\Phi}_{\vep}\right)_{T}=0,\,\left(\Psi_{\vep}\right)_{T}=0,\,\eta_{T}=0& \text{on }\Sigma\\
\delta\eta=\delta^{*}\eta=0&\text{on }\Omega
\end{cases}
\eeq
also satisfying
\beq\label{OmegaHodgeLogarithmicBound}
\left\|\overline{\Phi}_{\vep}\right\|_{W^{1,2}(\Omega)}+\left\|\Psi_{\vep}\right\|_{W^{1,2}(\Omega)}
\le{}K(\Omega)\sqrt{M_{0}\mathopen{}\left|\log(\vep)\right|\mathclose{}}.
\eeq
Next we include another new estimate needed to extend the argument of \cite{BOS2}
to the setting of a Riemannian manifold.
As this represents an addition to the argument from \cite{BOS2} we provide a more
detailed discussion. 
We will write
\[
H^{1}_{T}(\Omega) := \{ \mbox{ $1$-forms $\eta$ on $\Omega$} : \delta\eta = \delta^*\eta = 0\mbox{ in }\Omega, \ \eta_T=0\mbox{ on }\partial \Omega\}.
\]
It follows from the discussion in Lemma $10$ of Section $5.3$ of \cite{BJOS} that $H^1_T(\Omega)$ is a real vector space of dimension  $(\#\mbox{ of components of }\pp \Omega)-1 = 1.$ 
Since $H^1_T(\Omega)$ clearly includes all $1$-forms of the form $\eta = {const}\,\mrmd t$,
we deduce that
\beq\label{HodgeDecompositionOverOmegaUpdated}
\eta=a_{\vep}\mrmd{}t
\eeq
where $a_{\vep}\in\R$ that may depend on $\vep$.
Next we observe that
\begin{equation}\label{HarmonicCoefficientBound}
a_{\vep}\cdot\vol_{g}(M)(t_{2}-t_{1})=\int_{\Omega}\!{}\left<j_{\Omega}w_{\vep},\mrmd{}t\right>_{\Omega}
\end{equation}
where we have used the abbreviation $j_{\Omega}w_{\vep}\ceq{}w_{\vep}\times\delta{}w_{\vep}$.
By \eqref{TransferenceIdentity} and the fact that $\gamma_{\vep}$ is independent of $\mrmd{}t$ we can rewrite
\eqref{HarmonicCoefficientBound} as
\begin{align*}
\int_{\Omega}\!{}\left<j_{\Omega}w_{\vep},\mrmd{}t\right>_{\Omega}
&=\int_{\Omega}\!{}\left<j_{\Omega}\uvep,\mrmd{}t\right>_{\Omega}
-\int_{\Omega}\!{}\left<\delta\Phi_{\vep}^{1,2},\mrmd{}t\right>_{\Omega}
+\int_{\Omega}\!{}\left<\bigl(1-\left|\uvep\right|^{2}\bigr)\delta\Phi_{\vep}^{1,2},\mrmd{}t\right>_{\Omega}\\
&\eqqcolon(A)+(B)+(C).
\end{align*}
Observe that since $\uvep$ solves \ref{PGLOriginal} then we have
\begin{align}\label{(A)Estimate}
(A)&=\int_{t_{1}}^{t_{2}}\!\!\int_{M}\!{}\uvep\times\pp_{t}\uvep\dvol_{g}(x)\mrmd{}t
=\int_{t_{1}}^{t_{2}}\!\!\int_{M}\!{}\uvep\times\Delta_{M}\uvep\dvol_{g}(x)\mrmd{}t\\
&=-\int_{t_{1}}^{t_{2}}\!\left[\int_{M}\!{}\left<\mrmd{}^{*}(\uvep\times\mrmd\uvep),1\right>_{M}\right]\mrmd{}t
=-\int_{t_{1}}^{t_{2}}\!\left[\int_{M}\!{}\left<\uvep\times\mrmd\uvep,\mrmd(1)\right>_{M}\right]\mrmd{}t
=0  \nonumber
\end{align}
where we integrated by parts over $M$.
Next, observe that by \eqref{ZeroAverage}
\beq\label{(B)Estimate}
(B)=\int_{M}\!\int_{t_{1}}^{t_{2}}\!{}\pp_{t}\Phi_{\vep}^{1,2}\mrmd{}t\,\dvol_{g}(x)
=\int_{M}\!{}\left[\Phi_{\vep}^{2}-\Phi_{\vep}^{1}\right]\dvol_{g}(x)=0.
\eeq
Finally, observe that by Cauchy-Schwarz, \eqref{EnergyLogarithmicBound}, and \eqref{LogarithmicBoundOnExtension}
\beq\label{(C)Estimate}
|(C)|\le\int_{t_{1}}^{t_{2}}\!\!\!\int_{M}\!{}\bigl|1-|\uvep|^{2}\bigr|\bigl|\pp_{t}\Phi_{\vep}^{1,2}\bigr|
\le{}C(\Omega)M_{0}\vep\mathopen{}\left|\log(\vep)\right|\mathclose{}.
\eeq
Combining \eqref{(A)Estimate}, \eqref{(B)Estimate}, and \eqref{(C)Estimate} with \eqref{HarmonicCoefficientBound} we obtain
\beq\label{HarmonicCoefficientEstimate}
|a_{\vep}|\le{}C(\Omega)M_{0}\vep\mathopen{}\left|\log(\vep)\right|\mathclose{}.
\eeq
Next, note that $\Psi_{\vep}$ satisfies
\beq\label{JacobianEllipticProblem}
\begin{cases}
-\Delta_{\Omega}\Psi_{\vep}=\omega_{\vep}\ceq{}Jw_{\vep}& \text{in }\Omega,\\
\left(\Psi_{\vep}\right)_{T}=0& \text{on }\pp\Omega\\
\left(\delta^{*}\Psi_{\vep}\right)_{T}=A_{\vep}
\ceq\mrmd^{*}\Psi_{\vep}^{i}+\bigl(\gamma_{\vep}^{i}-\lfloor\gamma_{\vep}^{0}\rfloor\bigr)
+\bigl(1-|\uvep|^{2}\bigr)\bigl(\mrmd\Phi_{\vep}^{i}+\lfloor\gamma_{\vep}^{0}\rfloor\bigr)& 
\text{on }M\times\{t_{i}\}
\end{cases}
\eeq
for $i=1,2$ where $\Delta_{\Omega}$ is the Laplacian on $\Omega$.
By \eqref{W1pBoundForPsi}, \eqref{InitialTimeEstimate}, \eqref{FinalTimeEstimate}, \eqref{SmallErrorBoundaryEstimates1}, and
\eqref{SmallErrorBoundaryEstimates2} we have, for $i=1,2$, $1\le{}p<\frac{N+1}{N}$,
and $q=\frac{p}{p-1}$ that
\beq\label{BoundaryCodifferentialBound}
\left\|A_{\vep}\right\|_{\bigl(W^{1-\frac{1}{q},q}\left(\pp\Omega\right)\bigr)^{*}}
\le\left\|A_{\vep}\right\|_{\left[L^{q}\left(\pp\Omega\right)\right]^{*}}
=\left\|A_{\vep}\right\|_{L^{p}(\pp\Omega)}\le{}C(p,\Omega)(M_{0}+1).
\eeq
Arguing as in \eqref{W1pBoundForPsi} we also have
\[
\left\|\omega_{\vep}\right\|_{\left[W^{1,q}\left(\Omega\right)\right]^{*}}
\le{}C(p,\Omega)\left\|\omega_{\vep}\right\|_{\left[C^{0,\alpha}\left(\Omega\right)\right]^{*}}
\le{}C(\alpha,\Omega)(M_{0}+1).
\]
Thus, by elliptic regularity, obtained by a Stampacchia duality argument obtained by
combining Proposition $A.2$ of \cite{BBO} and Corollary $5.6$ of \cite{ISS}, we have
\beq\label{PsiOmegaEstimate}
\left\|\Psi_{\vep}\right\|_{W^{1,p}(\Omega)}\le{}C(p,\Omega)(M_{0}+1).
\eeq
We refer the reader to \cite{BOS2} as well as Proposition
\ref{ControlledHodgeDeRham} for more details regarding this estimate.
We set
\[
\Psi\ceq\Psi_{\vep},\hspace{10pt}
\Phi\ceq\Phi_{\vep}^{1,2}+\overline{\Phi}_{\vep},\hspace{10pt}
\gamma\ceq\gamma_{\vep},\hspace{10pt}
\zeta\ceq-\bigl(1-|\uvep|^{2}\bigr)\bigl(\delta\Phi_{\vep}^{1,2}+\gamma_{\vep}\bigr)+\eta.
\]
Then
\begin{align*}
\uvep\times\delta\uvep
&=w_{\vep}\times\delta{}w_{\vep}+|\uvep|^{2}\bigl(\delta\Phi_{\vep}^{1,2}+\gamma_{\vep}\bigr)\\
&=\delta\overline{\Phi}_{\vep}+\delta^{*}\Psi_{\vep}+\eta+\delta\Phi_{\vep}^{1,2}+\gamma_{\vep}
-\bigl(1-|\uvep|^{2}\bigr)\bigl(\delta\Phi_{\vep}^{1,2}+\gamma_{\vep}\bigr)\\
&=\delta\Phi+\delta^{*}\Psi+\gamma+\zeta.
\end{align*}
The conclusion follows from \eqref{HodgeDecompositionOverOmega}, \eqref{LogarithmicBoundOnExtension},
\eqref{SmallErrorTermEstimates1}, \eqref{SmallErrorTermEstimates2}, \eqref{OmegaHodgeLogarithmicBound}, \eqref{HarmonicCoefficientEstimate}, and
\eqref{PsiOmegaEstimate}.
\end{proof}
Next we demonstrate, following \cite{BOS2}, that the phase portion, $\Phi$, of $\uvep$ is close to
satisfying the heat equation.
\begin{lemma}\label{EvolutionOfThePhase}
Suppose $\uvep$ satisfies \ref{PGLOriginal} on $M\times(0,\infty)$ and
\eqref{H0}, and suppose, for $0<t_{1}<t_{2}<\infty$,
we set $\Omega=M\times(t_{1},t_{2})$.
For $\vep>0$ sufficiently small we let $\Phi$, $\Psi$, $\gamma$, and
$\zeta$ satisfy the conclusions of Proposition \ref{ControlledHodgeDeRham}.
Then the function $\Phi$ verifies the equation
\begin{equation}\label{CloseToHeat}
\partial_{t}\Phi-\Delta\Phi=-\mrmd^{*}(\delta^{*}\Psi+\zeta-P_{t}(\delta^{*}\Psi+\zeta)\mrmd{}t)
-P_{t}(\delta^{*}\Psi+\zeta)\hspace{5pt}\text{in }\Omega.
\end{equation}
Here, for a $1$-form $\omega$ on $\Omega$, $P_{t}(\omega)$ denotes its $\mrmd{}t$ component.
\end{lemma}
\begin{proof}
By Proposition \ref{ControlledHodgeDeRham} we have
\beq\label{PGLHodgeDecomposition}
\uvep\times\delta\uvep=\delta\Phi+\delta^{*}\Psi+\gamma+\zeta,
\eeq
where $\Phi$, $\Psi$, $\gamma$, and $\zeta$ verify the conclusions of Proposition
\ref{ControlledHodgeDeRham}.
Taking the cross product of \ref{PGLOriginal} with $\uvep$ leads to
\beq\label{CurrentEquation}
\uvep\times\pp_{t}\uvep=-\mrmd^{*}(\uvep\times\mrmd\uvep)\hspace{5pt}
\text{in }\Omega.
\eeq
On the other hand, we also have by \eqref{PGLHodgeDecomposition}
\beq\label{SeparationOfDerivatives}
\begin{cases}
\uvep\times\mrmd\uvep&=\mrmd\Phi+\gamma+(\delta^{*}\Psi+\zeta)-P_{t}(\delta^{*}\Psi+\zeta)\mrmd{}t,\\
\uvep\times\pp_{t}\uvep&=\Phi_{t}+P_{t}(\delta^{*}\Psi+\zeta).
\end{cases}
\eeq
Notice that $\mrmd^{*}\gamma=0$ since $\gamma$ is a harmonic $1$-form on $M\times\{t\}$ for all $t$.
As a result of this last observation along with \eqref{CurrentEquation} and \eqref{SeparationOfDerivatives} we obtain the conclusion.
\end{proof}
With the above ingredients in hand, the proof of  Theorem \ref{Theorem3}
exactly follows  arguments in \cite{BOS2}. We recall some details for the
convenience of the reader.\\[1pt]

\hspace{-15pt}{\it Proof of Theorem \ref{Theorem3}}
\par{}Let $\uvep$ be a solution of \ref{PGLOriginal} verifying \eqref{H0} on
$M\times(0,\infty)$.
Let $\calK$ be a compact subset of $M\times(0,\infty)$.
Choose $0<t_{1}<t_{2}<\infty$ so that $\calK\subset{}M\times(t_{1},t_{2})$.
Let $\Omega\ceq{}M\times(t_{1},t_{2})$ and suppose that $\Phi$, $\Psi$, $\gamma$, and $\zeta$ be as in Proposition \ref{ControlledHodgeDeRham} and
Lemma \ref{EvolutionOfThePhase}.
We choose $t_{3}$ and $t_{4}$ such that $t_{1}<t_{3}<t_{4}<t_{2}$ and so that $\calK\subset{}M\times(t_{3},t_{4})\eqqcolon\Lambda$.
By perhaps perturbing $t_{3}$ and $t_{4}$ we may assume
\beq\label{PhaseBoundaryLogarithmicBound}
\int_{\partial\Lambda}\!{}|\Phi|^{2}+
\int_{\partial\Lambda}\!{}|\nabla_{x,t}\Phi|^{2}\le{}C(\calK)(M_{0}+1)\mathopen{}\left|\log(\vep)\right|\mathclose{}.
\eeq
This is possible because of \eqref{LogarithmicControlGradientHodgeDeRhamDecomposition}.
We split the proof into two steps.\\

\emph{Step }$1$:\emph{ Defining }$\vphi_{\vep}$.
Let $\varphi_{\vep}$ verify the homogeneous heat equation
\beq\label{PhaseHeatEquation}
\begin{cases}
\pp_{t}\vphi_{\vep}-\Delta\vphi_{\vep}=0&\text{in }\Lambda\\
\vphi_{\vep}=\Phi&\text{on }\mathcal{O}_{1}
\end{cases}
\eeq
and define 
\[
w_{\vep}\ceq\uvep{}e^{-i\vphi_{\vep}}\bar{u}_{h,\vep}
\]
where $u_{h,\vep}$ is the $\bbS^{1}$-valued function described in Proposition
\ref{ControlledHodgeDeRham} satisfying
\[
ju_{h,\vep}=\gamma.
\]
From the standard regularity theory for the heat equation, see Theorems $8$ and $9$ of
\cite{Ev}, in addition to
\eqref{LogarithmicControlGradientHodgeDeRhamDecomposition} we have
\beq\label{GradientPhaseRegularity}
\left\|\nabla\vphi_{\vep}\right\|_{L^{\infty}(\calK)}^{2}
\le{}C(\calK)\left\|\Phi\right\|_{W^{1,2}(\mathcal{O}_{0})}^{2}
\le{}C(\calK)(M_{0}+1)\mathopen{}\left|\log(\vep)\right|\mathclose{}.
\eeq

Next, for later use we set $\Phi_{1} = \Phi - \varphi_\vep$. Then $\Phi_1$
solves
\beq\label{HeatProblemCorrection}
\begin{cases}
\pp_{t}\Phi_{1}-\Delta\Phi_{1}=-\mrmd^{*}(\delta^{*}\Psi+\zeta-P_{t}(\delta^{*}\Psi+\zeta)\mrmd{}t)-P_{t}(\delta^{*}\Psi+\zeta)
\hspace{5pt}&\text{in }\Lambda,\\
\Phi_{1}=0&\text{on }\mathcal{O}_{1}.
\end{cases}
\eeq
Since by \eqref{DelicateLpEstimates} we have
\[
\left\|\delta^{*}\Psi+\zeta-P_{t}(\delta^{*}\Psi+\zeta)\mrmd{}t\right\|_{L^{p}(\Lambda)}
+\left\|P_{t}(\delta^{*}\Psi+\zeta)\right\|_{L^{p}(\Lambda)}
\le{}C(p,\calK)(M_{0}+1)
\]
it follows from standard estimates for the non-homogeneous heat equation that
\beq\label{GradientEstimateCorrection}
\left\|\nabla\Phi_{1}\right\|_{L^{p}(\Lambda)}\le{}C(p,\calK)(M_{0}+1).
\eeq

\emph{Step }$2$: $W^{1,p}$\emph{ estimates for }$w_{\vep}$.
First observe that
\[
|w_{\vep}|^{2}|\nabla{}w_{\vep}|^{2}
=|w_{\vep}|^{2}\bigl|\nabla|w_{\vep}|\bigr|^{2}
+|w_{\vep}\times\nabla{}w_{\vep}|^{2},
\]
and hence
\beq\label{LargeLevelSetsEstimate}
\int_{\calK\cap\{|\uvep|\ge\frac{1}{2}\}}\!{}|\nabla{}w_{\vep}|^{p}
\le{}C(p)\biggl[\int_{\calK}\!{}|w_{\vep}\times\delta{}w_{\vep}|^{p}
+\int_{\calK}\!{}\bigl|\nabla|w_{\vep}|\bigr|^{p}\biggr].
\eeq
On the other hand, by standard estimates for \ref{PGLOriginal},
\eqref{GradientPhaseRegularity},
\eqref{LogarithmicControlGradientHodgeDeRhamDecomposition},
and equivalence of norms for $\gamma$ we have
\[
|\nabla{}w_{\vep}|\le|\nabla\uvep|+3|\nabla\vphi_{\vep}|+3|\gamma|
\le{}C(\calK)M_{0}\vep^{-1},
\]
where we have used that since $|u_{h,\vep}|=1$ then $|\nabla{}u_{h,\vep}|=|ju_{h,\vep}|=|\gamma|$.
As a result, we have
\beq\label{SmallLevelSetsEstimate}
\int_{\calK\cap\{|\uvep|\le\frac{1}{2}\}}\!{}|\nabla{}w_{\vep}|^{p}
\le{}C(p,\calK)M_{0}^{p}\vep^{2-p}\int_{\calK}\!{}V_{\vep}(u_{\vep})
\le{}C(p,\calK)M_{0}^{p+1}.
\eeq
By the definition of $w_{\vep}$ and Proposition \ref{ControlledHodgeDeRham} we have
\beq\label{ModifiedPGLSolutionHodgeDecomposition}
w_{\vep}\times\delta{}w_{\vep}=\delta^{*}\Psi+\delta\Phi_{1}
+\zeta+\bigl(1-|u_{\vep}|^{2}\bigr)(\delta\varphi_{\vep}+\gamma).
\eeq
By H\"{o}lder's inequality we have
\begin{align*}
\left\|\bigl(1-|\uvep|^{2}\bigr)\delta\vphi_{\vep}\right\|_{L^{p}(\calK)}
+\left\|\bigl(1-|\uvep|^{2}\bigr)\gamma\right\|_{L^{p}(\calK)}
&\le{}C(p,\calK)M_{0}\vep^{\frac{2-p}{p}}\mathopen{}\left|\log(\vep)\right|\mathclose{},
\end{align*}
and hence, when combined with \eqref{ModifiedPGLSolutionHodgeDecomposition}, Proposition \ref{ControlledHodgeDeRham}, and
\eqref{GradientEstimateCorrection}, we have
\beq\label{ModifiedCurrentPartEstimate}
\int_{\calK}\!{}|w_{\vep}\times\delta{}w_{\vep}|^{p}\le{}C(p,\calK)(M_{0}+1).
\eeq
The proof for the gradient of the modulus remains the same as in \cite{BOS2} except we use a cutoff function in time, $\chi_{\calK}$, and work over a
set $\calK'\ceq{}M\times[t_{4},t_{5}]\subset\Omega$ containing $\calK$.
Following this procedure we have, for $1\le{}p<2$, that
\beq\label{ModGrad:Estimate}
\int_{B_{\calK}}\!{}\bigl|\nabla|w_\vep| \bigr|^{p}
\le{}C(\calK)(M_{0}+1)\vep^{1-\frac{p}{2}}
\mathopen{}\left|\log(\vep)\right|\mathclose{}.
\eeq
We refer the reader to \cite{BOS2}  for additional details.
Combining \eqref{ModGrad:Estimate} with \eqref{LargeLevelSetsEstimate},
\eqref{SmallLevelSetsEstimate}, and \eqref{ModifiedCurrentPartEstimate} completes
the proof.
\hfill\qedsymbol\\

\section{Analysis of Limiting Measures}\label{Sec::Limiting}
\hspace{15pt}In this section we complete the proof of Theorem \ref{BOSTheorem}.
To do this we will, as in \cite{BOS2}, combine the results of
Theorems \ref{Theorem1}, \ref{Theorem2}, \ref{Theorem3}
as well as their consequences
and apply a detailed analysis of the limiting energy measure.
Much of the corresponding proof used in \cite{BOS2} carries
over to the general setting with minor variations.
However, new ingredients are needed in the globalization of
$\Phi_{*}$ due to the presence of $u_{h,\vep}$ from
Theorem \ref{Theorem3}.
We refer the reader to Section A.7 of \cite{Col} for more detail.\\

We fix solutions $\{\uvep\}_{0<\vep<1}$  of
\ref{PGLOriginal} satisfying assumption \eqref{H0},
and we define Radon measures over
$M\times[0,\infty)$ and its time slices by
\[
\begin{aligned}
\mu_{\vep}(x,t) &\ceq\frac{\evep(\uvep(x,t))}{\mathopen{}\left|\log(\vep)\right|\mathclose{}}\dvol_{g}(x)\mrmd{}t
\\
\mu_{\vep}^{t}(x) &\ceq\frac{\evep(\uvep(x,t))}{\mathopen{}\left|\log(\vep)\right|\mathclose{}}\dvol_{g}(x).
\end{aligned}
\]
As a result of assumption \eqref{H0} and standard estimates for
\ref{PGLOriginal}, together with well-known arguments from \cite{Brak, Ilmanen}, there is a
subsequence $\vep_{n}\to0^{+}$ and Radon measures $\mus$ and $\mus^t$, 
defined on $M\times[0,\infty)$ and on $M$ respectively, such that
\beq\label{WeakStarMeasureLimit}
\begin{aligned}
&\mu_{\vep_{n}}\rightharpoonup\mus\hspace{5pt}\text{as measures,}\\
&\mu_{\vep_{n}}^{t}\rightharpoonup\mus^{t}\text{ as measures for all
$t>0$, \quad where }\mus=\mus^{t}\mrmd{}t.
\end{aligned}
\eeq
We will write $\vep$ instead of $\vep_{n}$ when this is not misleading.
We also identify the measure $\mus^{t}$ with a measure on $M\times\{t\}$, and we will sometimes identify $M$ with $M\times\{t\}$.
We record a consequence of the monotonicity formula on the limit measures.
\begin{lemma}\label{MonotonicityOfLimitMeasure}
For each $t>0$ and $x\in{}M$, the function $r\mapsto\mathscr{G}_{\mu}((x,t),\cdot)$ defined on $(0,\infty)$ by
\[
\mathscr{G}((x,t),r)
\ceq\frac{e^{C_{2}r}}{(4\pi)^{\frac{N}{2}}r^{N-2}}\int_{M}\!{}e^{-\frac{(d_{+}(x,y))^{2}}{4r^{2}}}\mrmd\mus^{t-r^{2}}(y)
+C_{1}M_{0}r,
\]
where $C_{1}$ and $C_{2}$ are determined by Proposition \ref{MonotonicityFormula}, is non-decreasing for
$0<r<\min\{\sqrt{t},1\}$.
\end{lemma}
Next, we record an important consequence of the previous analysis.
\begin{theorem}\label{Theorem5}
There exists an absolute constant $\eta_{2}>0$, and a positive continuous function $\lambda$ defined on $(0,\infty)$ such that if,
for $x\in{}M$, $t>0$, and $r>0$ sufficiently small,
and
\beq\label{MeasureDensityBound}
\mus^{t}\left(B_{\lambda(t)r}(x)\right)<\eta_{2}r^{N-2},
\eeq
then for every $s\in\bigl[t+\frac{15}{16}r^{2},t+r^{2}\bigr]$, $\mus^{s}$ is absolutely continuous with respect to the volume measure on the ball
$B_{\frac{r}{4}}(x)$.
More precisely,
\[
\mus^{s}=\frac{|\nabla\Phi_{*}|^{2}}{2}\dvol_{g}(x)\hspace{5pt}\text{on }B_{\frac{r}{4}}(x),
\]
where $\Phi_{*}$ satisfies the heat equation in
$\Lambda_{\frac{1}{4}}=\Lambda_{\frac{1}{4}}(x,t,r,r^{2})$.
\end{theorem}
\begin{proof}
The proof is the same as in \cite{BOS2} and is a straightforward
consequence of Proposition \ref{Proposition4}.
\end{proof}

\subsection{Densities and the concentration set}\label{Densities and the concentration set}
We begin by introducing some notation for measure densities.
\begin{definition}\label{LowerAndUpperMeasureDensities}
Let $\nu$ be a Radon measure on $M$.
For $m\in\N$, the $m$-dimensional lower and upper densities of $\nu$ at the point $x$, denoted $\Theta_{*,m}(\nu,x)$ and
$\Theta_{m}^{*}(\nu,x)$ respectively, are defined by
\[
\Theta_{*,m}(\nu,x)\ceq\liminf_{r\to0^{+}}\frac{\nu(B_{r}(x))}{\omega_{m}r^{m}},
\hspace{10pt}
\Theta_{m}^{*}(\nu,x)\ceq\limsup_{r\to0^{+}}\frac{\nu(B_{r}(x))}{\omega_{m}r^{m}}
\]
where $\omega_{m}$ denotes the volume of the $m$-dimensional Euclidean unit ball in the standard metric.
When both quantities coincide, $\nu$ admits an $m$-dimensional density $\Theta_{m}(\nu,x)$ at the point $x$, defined as the common value.
\end{definition}
Next, following \cite{BOS2}, we record a lemma regarding upper
bounds on measure densities.
\begin{lemma}\label{UpperBoundDensityBound}
For all $x\in{}M$ and for all $t>0$,
\[
\Theta_{*,N-2}(\mus^{t},x)\le\Theta_{N-2}^{*}(\mus^{t},x)
\le\frac{M_{0}e^{\frac{1}{4}}}{\omega_{N-2}}\Bigl[e^{C_{2}t}t^{\frac{2-N}{2}}+(4\pi)^{\frac{N}{2}}C_{1}\sqrt{t}\Bigr].
\]
\end{lemma}
\begin{proof}
The first inequality follows from the definition of lower and upper densities while the second inequality follows
from the fact that $d$ agrees with $d_{+}$ on $B_{r}(x)$ for
each $x\in{}M$ and
$0<r<\min\Bigl\{1,\frac{\inj_{g}(M)}{2}\Bigr\}$
combined with
Lemma \ref{MonotonicityOfLimitMeasure} and \eqref{H0}.
\end{proof}
Proceeding as in \cite{BOS2} we introduce  a suitable notion
(not the usual one in this context) of parabolic $m$-dimensional
lower density of a Radon measure $\nu$.
\begin{definition}\label{ParabolicLowerDensity}
Let $\nu$ be a Radon measure on $M\times[0,\infty)$ such that $\nu=\nu^{t}\mrmd{}t$.
For $t>0$ and $m\in\N$, the parabolic $m$-dimensional density of $\nu$ at the point $(x,t)$ is defined by
\begin{equation*}
\Theta_{m}^{P}\left(\nu,(x,t)\right)
\ceq\lim_{r\to0^{+}}\frac{1}{(4\pi)^{\frac{N}{2}}r^{m}}\int_{M}\!{}e^{-\frac{(d_{+}(x,y))^{2}}{4r^{2}}}\mrmd\nu^{t-r^{2}}(y)
\end{equation*}
when it exists.
\end{definition}
Observe that since $r\mapsto\mathscr{G}_{\mu}((x,t),r)$ is non-decreasing then $\Theta_{N-2}^{P}(\mu^{*},(x,t))$ is defined everywhere in $M\times(0,\infty)$.
Next, analogously to \cite{BOS2}, we will relate the parabolic
density to the lower $(N-2)$-dimensional measure density.
\begin{lemma}\label{ParabolicDensityAndLowerDensity}
Suppose $x\in{}M$ and $t>0$.
Then, there exists $K_{M}>0$, depending on $M$, such that
\beq\label{ParabolicDensityRelation}
\Theta_{N-2}^{P}(\mu^{*},(x,t))\ge{}K_{M}\Theta_{*,N-2}(\mus^{t},x)).
\eeq
\end{lemma}
\begin{proof}
The proof is the same as found in Subsection $6.2$ of
\cite{BOS2}.
We refer the reader to \ref{UpperSemiContinuityParabolicDensity} for additional details.\\

Let $(x,t)\in{}M\times(0,\infty)$ be given.
Let $0<r<\min\Bigl\{t,1,\frac{\inj_{g}(M)}{2}\Bigr\}$ be fixed.
Similar to the proof of Lemma \ref{UpperBoundDensityBound}, we conclude from  Lemma \ref{MonotonicityOfLimitMeasure} that
\[
\frac{\mus^{t}(B_{r}(x))}{r^{N-2}}
\le\frac{e^{\frac{1}{4}+C_{2}r}}{(r^{2}+r)^{\frac{N-2}{2}}}\int_{M}\!{}e^{-\frac{(d_{+}(x,y))^{2}}{4(r^{2}+r)}}\mrmd\mus^{t-r}(y)
+(4\pi)^{\frac{N}{2}}e^{\frac{1}{4}}C_{1}M_{0}\sqrt{r^{2}+r}.
\]
Observe that on $B_{\frac{\inj_{g}(M)}{2}}(x)$ that
\[
e^{-\frac{(d_{+}(x,y))^{2}}{4(r^{2}+r)}}
=e^{-\frac{(d_{+}(x,y))^{2}}{4r}}e^{\frac{(d_{+}(x,y))^{2}}{4(r+1)}}
\le{}e^{\frac{(\inj_{g}(M))^{2}}{16}}e^{-\frac{(d_{+}(x,y))^{2}}{4r}}.
\]
On $M\setminus{}B_{\frac{\inj_{g}(M)}{2}}(x)$ we have
\[
\int_{M\setminus{}B_{\frac{\inj_{g}(M)}{2}}(x)}\!{}e^{-\frac{(d_{+}(x,y))^{2}}{4(r^{2}+r)}}\mrmd\mus^{t-r}(y)
\le{}e^{-\frac{[\inj_{g}(M)]^{2}}{16(r^{2}+r)}}M_{0}.
\]
Putting these together we obtain
\begin{align*}
\frac{\mus^{t}(B_{r}(x))}{r^{N-2}}
&\le{}\frac{e^{\frac{1}{4}+\frac{(\inj_{g}(M))^{2}}{16}+C_{2}r}}{r^{\frac{N-2}{2}}}
\int_{M}\!{}e^{-\frac{(d_{+}(x,y))^{2}}{4r}}\mrmd\mus^{t-r}\\
&+\frac{e^{\frac{1}{4}+C_{2}r-\frac{[\inj_{g}(M)]^{2}}{16(r^{2}+r)}}}{(r^{2}+r)^{\frac{N-2}{2}}}M_{0}
+(4\pi)^{\frac{N}{2}}e^{\frac{1}{4}}C_{1}M_{0}\sqrt{r^{2}+r}.
\end{align*}
Letting $r\to0^{+}$ gives the conclusion.
\end{proof}
Just as in \cite{BOS2} we define
\begin{align}
\Sigma_{\mu}&\ceq\bigl\{(x,t)\in{}M\times(0,\infty):\Theta_{N-2}^{P}(\mus,(x,t))>0\bigr\},    \label{VortexFilamentSet}
\\
\Sigma_{\mu}^{t}&\ceq\Sigma_{\mu}\cap(M\times\{t\}).
\hspace{5pt}\text{for $t>0$.}
\label{VortexFilamentSetSlice}
\end{align}
A consequence of Lemma \ref{ParabolicDensityAndLowerDensity} is
\beq\label{LowerDensityConsequence}
\Theta_{*,N-2}\left(\mus^{t},x\right)\equiv0\hspace{5pt}\text{on }M\setminus\Sigma_{\mu}^{t}.
\eeq
Next we record, for future use, that the function  $(x,t)\mapsto{}\Theta_{N-2}^{P}(\mus,(x,t))$ is upper semi-continuous on
$M\times(0,\infty)$.
We note that the proof of this is the same as in \cite{BOS2}.
More detail regarding its extension can be found in A.7.1.1 of \cite{Col}.
\begin{lemma}\label{UpperSemiContinuityParabolicDensity}
The map $(x,t)\mapsto\Theta_{N-2}^{P}(\mus,(x,t))$ is upper semi-continuous on $M\times(0,\infty)$.
\end{lemma}

\subsection{First properties of \texorpdfstring{$\Sigma_{\mu}$}{}}
\hspace{15pt}We begin this subsection by demonstrating a lower bound estimate on the $(N-2)$-dimensional lower density over the set $\Sigma_{\mu}$.
The proof follows \cite{BOS2} closely so we refer the reader to
A.7.2.1 of \cite{Col} for more details.
\begin{lemma}\label{LowerDensityLowerBound}
Suppose $0<r<\sqrt{t}$ and $x\in{}M$.
Then, if $(x,t)\in\Sigma_{\mu}$ it follows that
\[
r^{2-N}\mus^{t-r^{2}}\left(B_{\lambda(t-r^{2})r}(x)\right)>\eta_{2},
\]
where $\eta_{2}$ is the constant in Theorem \ref{Theorem5}.
\end{lemma}
\begin{proof}
We proceed by proving the contrapositive statement.
Suppose there is $(x,t)\in{}M\times(0,\infty)$ and $0<r<\sqrt{t}$
for which
\[
r^{2-N}\mus^{t-r^{2}}(B_{\lambda(t-r^{2})r}(x))\le\eta_{2}.
\]
By Theorem \ref{Theorem5},  for all $\tau\in\bigl[t-\frac{r^{2}}{16},t\bigr]$ we have
\[
\mus^{\tau}=\frac{|\nabla\Phi_{*}|^{2}}{2}\dvol_{g}(x)\hspace{15pt}B_{\frac{r}{4}}(x)
\]
where $\Phi_{*}$ is smooth.
Straightforward computations then show that
$\Theta_{*}^{P}(\mus,(x,t))=0$.
\end{proof}
Next we prove a clearing out lemma related to the set $\Sigma_{\mu}$.
\begin{theorem}\label{Theorem6}
There exists a positive continuous function $\eta_{3}$ defined on $(0,\infty)$, such that for any $(x,t)\in{}M\times(0,\infty)$ and any
$0<r<\sqrt{t}$, if
\[
\mathscr{F}_{\mu}((x,t),r)\ceq\frac{1}{r^{N-2}}\int_{M}\!{}e^{-\frac{(d_{+}(x,y))^{2}}{4r^{2}}}\mrmd\mus^{t-r^{2}}(y)
\le\eta_{3}(t-r^{2})
\]
then $(x,t)\notin\Sigma_{\mu}$.
\end{theorem}
\begin{proof}
The proof extends to our setting without change to the argument
from the proof of Theorem $6$ of \cite{BOS2}.
We refer the reader to \cite{BOS2} or A.7.2.2 of \cite{Col} for
additional details.
\end{proof}
Following \cite{BOS2} we note that a consequence of Theorem
\ref{Theorem6}, for which details can be found in A.7.2.3 of \cite{Col},
is the following:
\begin{corollary}\label{Theorem6Corollary}
For any $(x,t)\in\Sigma_{\mu}$,
\[
\Theta_{N-2}^{P}(\mus,(x,t))\ge\eta_{3}(t).
\]
\end{corollary}
Next we provide a decomposition for $\mus^{t}$ and demonstrate a few properties of $\Sigma_{\mu}^{t}$ and $\Sigma_{\mu}$.
The proof of this proposition is the same as in \cite{BOS2}
with the exception that we rescale the metric instead of the
function in the argument for \eqref{SecondPropertySigmaSet}.
More details can be found in A.7.2.4 of \cite{Col}.
\begin{proposition}\label{Proposition6}\hspace{5pt}\vspace{2pt}
\begin{enumerate}
\item\label{FirstPropertySigmaSet}
The set $\Sigma_{\mu}$ is closed in $M\times(0,\infty)$.
\item\label{SecondPropertySigmaSet}
For any $t>0$ we have
\[
\calH^{N-2}(\Sigma_{\mu}^{t})\le{}KM_{0}<\infty.
\]
\item\label{DecompositionOfMeasureProperty}
For any $t>0$, the measure $\mus^{t}$ can be decomposed as
\[
\mus^{t}=g(x,t)\calH^{N}+\Theta_{*}(x,t)\calH^{N-2}\rest\Sigma_{\mu}^{t},
\]
where $g$ is some smooth function defined on $[M\times(0,\infty)]\setminus\Sigma_{\mu}$ and $\Theta_{*}$ verifies the bound
$\Theta_{*}(x,t)\le{}K_{M}M_{0}\bigl[e^{C_{M}t}t^{\frac{2-N}{2}}+D_{M}\sqrt{t}\bigr]$ for $C_{M},D_{M},K_{M}>0$ depending on $M$.
\end{enumerate}
\end{proposition}

\subsection{Regularity of \texorpdfstring{$\Sigma_{\mu}^{t}$}{}}
\hspace{15pt}Next we record that the $(N-2)$-dimensional parabolic density of $\mus$ is controlled by $\Theta_{*,N-2}(\mus^{t},x)$ for most $t$ and $x$.
This gives the reverse relationship illustrated in Lemma \ref{ParabolicDensityAndLowerDensity}.
The proof is very similar to the corresponding one from
\cite{BOS2} the only exceptions are that we invoke the
Besicovitch-Federer Covering
Theorem, see Theorem $2.8.14$ of \cite{Fed}, and we do
not restrict our analysis to a finite region of time.
As a result, we refer the reader to A.7.3.1 of \cite{Col} for more details.
\begin{proposition}\label{Proposition7}
For $\calL^{1}$-almost every $t>0$, the following inequality holds:
\beq\label{ReverseDensityInequality}
\Theta_{*,N-2}(\mus^{t},x)\ge{}K\Theta_{N-2}^{P}(\mus,(x,t))
\eeq
for $\calH^{N-2}$-almost every $x\in{}M$.
\end{proposition}
Next we show that a lower density bound holds on $\Sigma_{\mu}^{t}$ for most points.
\begin{corollary}\label{DensityLowerBound}
For $\calL^{1}$-almost every $t\ge0$
\begin{equation}\label{LowerDensityBoundInequality}
\Theta_{*,N-2}(\mus^{t},x)\ge{}K\eta_{3}(t)
\end{equation}
for $\calH^{N-2}$-almost every $x\in\Sigma_{\mu}^{t}$.
\end{corollary}
\begin{proof}
The corollary follows from Corollary \ref{Theorem6Corollary} and
Proposition \ref{Proposition7}.
Details can be found in A.7.3.2 of \cite{Col}.
\end{proof}
Finally, we show that for $\calL^{1}$-almost every $t>0$ and $\calH^{N-2}$-almost every $x\in\Sigma_{\mu}^{t}$ the upper and lower densities of
$\mus^{t}$ agree.
As a result, for $\calL^{1}$-almost every $t>0$ the set $\Sigma_{\mu}^{t}$ is $(N-2)$-rectifiable.

\begin{proposition}\label{Proposition8}
For $\calL^{1}$-almost every $t>0$,
\begin{equation*}
\Theta_{*,N-2}(\mus^{t},x)=\Theta_{N-2}^{*}(\mus^{t},x)\ge{}K\eta_{3}(t)
\end{equation*}
for $\calH^{N-2}$-almost every $x\in\Sigma_{\mu}^{t}$.
Consequently, for $\calL^{1}$-almost every $t>0$ the set $\Sigma_{\mu}^{t}$ is $(N-2)$-rectifiable.
\end{proposition}
\begin{proof}

The proof essentially follows ideas from \cite {BOS2}. 
One begins by defining the vector space, $F$, for a fixed $(x,t)\in\Omega_{\omega}$ by
\[
F\ceq\biggl\{g\in{}L^{\infty}\left((0,\infty);\R\right):I(g)\ceq\lim_{r\to0^{+}}I_{r}(g)\text{ exists and is finite}\biggr\}
\]
where for $r>0$,
\[
I_{r}(y)\ceq\frac{1}{r^{N-2}}\int_{M}\!{}g\biggl(\frac{d_{+}(x,y)}{r}\biggr)\mrmd\mus^{t}(y).
\]
The same definition appears in \cite{BOS2}, with the Euclidean distance in place of $d_{+}$.\\

To prove the proposition, it suffices to show that
$\chi_{[0,1]}\in{}F$.
The starting point is the fact that, if we write $e_s(\ell) = e^{-s\ell^2}$, then $e_{1/4}\in F$; this is established in the proof of Proposition \ref{Proposition7}.
One can then proceed using the
same technique as in \cite{BOS2}, which involves a number of
steps which we now outline.\\

It is now shown that if $g\in{}F$ then for $s>0$ the rescaling
$g_{s}\colon\ell\mapsto{}g(\ell\sqrt{s})$ belongs to $F$ as well.
Since $e_{1/4}\in{}F$ this shows that $e_{s}\in{}F$ for all
$s$.
Next, we proceed to inductively demonstrate that functions of
the form $\ell\mapsto\ell^{2k}e^{-\ell^{2}}$ for
$k\in\N\cup\{0\}$ are member of $F$.
We then show that $g\in{}C_{c}^{2}((0,\infty))$ satisfying
$g'(0)=0$ are also members of $F$ by appealing to Hermite
polynomials and an approximation argument.
Finally, we use members of $C_{c}^{2}((0,\infty))$ to
approximate $\chi_{[0,1]}$ and show that $\chi_{[0,1]}\in{}F$.\\

We refer the reader to the proof of Proposition $8$ and simply note that the proof presented there only depends on functions over
the real line.
\end{proof}

\subsection{Globalizing \texorpdfstring{$\Phi_{*}$}{}}
\hspace{15pt}In this subsection we demonstrate that the function $\Phi_{*}$ has a globally defined differential and partial derivative in $t$ even though its
construction was merely local.
\begin{lemma}\label{GlobalPhi}
The locally defined function $\Phi_{*}$ from Theorem \ref{Theorem5} extends to a function
$\Phi_{*}\colon{}M\times(0,\infty)\to\R\slash2\pi\Z$.
In particular, $\Phi_{*}$ has a differential, $\mrmd\Phi_{*}$, that is globally defined and satisfies
$\mrmd\Phi_{*}=\mrmd\phi_{*}+\gamma_{*}$ where $\phi_{*},\gamma_{*}$ are globally defined so that $\phi_{*}$ solves the heat equation on
$M\times(0,\infty)$ and $\gamma_{*}$ is a harmonic $1$-form on
$M\times(0,\infty)$ that is only a function of $x$ and
has no term corresponding to $\mrmd{}t$.
In addition, $\pp_{t}\Phi_{*}$ is globally defined and equal to $\pp_{t}\phi_{*}$.
\end{lemma}
\begin{proof}
For $m\in\N\setminus\{1\}$ we set $\calK_{m}=M\times\bigl[\frac{1}{m},m\bigr]$, so that
$\bigcup_{m\ge2}\calK_{m}=M\times(0,\infty)$.
Applying Theorem \ref{Theorem3} to $\calK=\calK_{m}$ we may write, for $\vep$ sufficiently small,
\beq\label{ExpressionOnmSet}
\uvep=e^{i\phi_{\vep}^{m}}w_{\vep}^{m}u_{h,\vep}\hspace{10pt}\text{on }\calK_{m},
\eeq
where $\phi_{\vep}^{m}$ solves the heat equation on $\calK_{m}$, $u_{h,\vep}\times\mrmd{}u_{h,\vep}=\gamma_{\vep}$
is a harmonic $1$-form on $M$ not dependent, as a function, on $t$ and $m$ and has no component corresponding to $\mrmd{}t$.
Theorem \ref{Theorem3} yields the estimates
\begin{align}
\left\|\nabla\phi_{\vep}^{m}\right\|_{L^{\infty}(\calK_{m})}
+\left\|\nabla{}u_{h,\vep}\right\|_{L^{\infty}(\calK_{m})}
&\le{}C(m)\sqrt{(M_{0}+1)\mathopen{}\left|\log(\vep)\right|\mathclose{}}\label{PhaseandHarmonicGradientBound}
\\
\label{AdditionalRegularity}
\left\|\nabla{}w_{\vep}^{m}\right\|_{L^{p}(\calK_{m})}
&\le{}C(m,p)\hspace{5pt}\text{for any }1\le{}p<\frac{N+1}{N}.
\end{align}
For fixed $m$,  we may pass to a further subsequence $\{\vep_{n}\}_{n\in\N}$ such that
\begin{align}
\frac{\phi_{\vep}^{m}}{\sqrt{\mathopen{}\left|\log(\vep)\right|\mathclose{}}} &\to\phi_{*}^{m}\hspace{5pt}\text{in }C^{2}(\calK_{m-1})
\label{ConvergentPhase}
\\
\frac{\gamma_{\vep}}{\sqrt{\mathopen{}\left|\log(\vep)\right|\mathclose{}}}
&\to{}\gamma_{*}
\hspace{5pt}\text{in }C^{2}(\calK_{m-1})
\label{HarmonicCoverging}
\end{align}
where $\phi_{*}^{m}$ also satisfies the heat equation on $\calK_{m-1}$, and $\gamma_*$ is a harmonic $1$-form. We have used the fact that the space of harmonic forms is  finite dimensional. Note also that 
$\gamma_*$  does not depend on $m$ or $t$ and has no component corresponding to
$\mrmd{}t$.

Next, let $x_{0}\in\Omega_{\mu}\ceq(M\times(0,\infty))\setminus\Sigma_{\mu}$.
By \eqref{FirstPropertySigmaSet} of Proposition \ref{Proposition6} we have that $\Omega_{\mu}$ is open.
Thus, we can find a set $\Lambda_{x_{0}}=B_{R}(x_{0})\times[t_{0},t_{1}]$ contained in $\Omega_{\mu}$.
For $m_{0}$ large enough we will have, for $m\ge{}m_{0}$, that $\Lambda_{x_{0}}\subset\calK_{m}$.
For $\vep$ sufficiently small we have
\beq\label{BoundedAwayFromZero}
|\uvep|\ge1-\sigma\ge\frac{1}{2}\hspace{5pt}\text{on }\Lambda_{x_{0}}
\eeq
where $\sigma$ is the constant in Theorem \ref{Theorem2}.
This lower bound on the norm allows us to write
\beq\label{LocalRepresentation}
\uvep=\rho_{\vep}e^{i\vphi_{\vep}}
\eeq
for some real-valued function $\vphi_{\vep}:M\times(0,\infty)\to\R\slash2\pi\Z$.
By \eqref{BoundedAwayFromZero} we may apply \eqref{ArgumentConvergence} of Theorem \ref{Theorem2} to demonstrate that there exists a solution
$\Phi_{\vep}$ of the heat equation on $\Lambda_{x_{0}}$ such that
\beq\label{HeatSolutionEstimate}
\left\|\nabla\Phi_{\vep}-\nabla\vphi_{\vep}\right\|_{L^{\infty}\bigl((\Lambda_{x_{0}})_{\frac{1}{2}}\bigr)}
\le{}C\vep^{\beta}.
\eeq
On the other hand, since $|w_{\vep}^{m}|=|\uvep|$ we may write, for $m\ge{}m_{0}$
\beq\label{LocalRepresentationw}
w_{\vep}^{m}=\rho_{\vep}e^{i\psi_{\vep}^{m}}\hspace{5pt}\text{on }\Lambda_{x_{0}}
\eeq
where $\psi_{\vep}^{m}:M\times(0,\infty)\to\R\slash2\pi\Z$.
Combining \eqref{ExpressionOnmSet}, \eqref{LocalRepresentation}, and \eqref{LocalRepresentationw} we obtain
\beq\label{GradientEquality}
\mrmd\vphi_{\vep}=\mrmd\phi_{\vep}^{m}+\gamma_{\vep}+\mrmd\psi_{\vep}^{m}.
\eeq
By \eqref{HeatSolutionEstimate}  for fixed $m$ we have
\[
\biggl|\frac{\mrmd\phi_{\vep}^{m}+\gamma_{\vep}-\mrmd\Phi_{\vep}}{\sqrt{\mathopen{}\left|\log(\vep)\right|\mathclose{}}}\biggr|
\le\biggl|\frac{\mrmd\psi_{\vep}^{m}}{\sqrt{\mathopen{}\left|\log(\vep)\right|\mathclose{}}}\biggr|+C\vep^{\beta}
\hspace{5pt}\text{on }(\Lambda_{x_{0}})_{\frac{1}{2}}.
\]
By \eqref{AdditionalRegularity} we obtain
\[
\left\|\frac{\mrmd\phi_{\vep}^{m}}{\sqrt{\mathopen{}\left|\log(\vep)\right|\mathclose{}}}
+\frac{\gamma_{\vep}}{\sqrt{\mathopen{}\left|\log(\vep)\right|\mathclose{}}}
-\frac{\mrmd\Phi_{\vep}}{\sqrt{\mathopen{}\left|\log(\vep)\right|\mathclose{}}}\right\|_{L^{p}\bigl((\Lambda_{x_{0}})_{\frac{1}{2}}\bigr)}\to0
\hspace{5pt}\text{as }\vep\to0^{+}.
\]
Since $\frac{\mrmd\phi_{\vep}^{m}}{\sqrt{\mathopen{}\left|\log(\vep)\right|\mathclose{}}}\to\phi_{*}^{m}$ and
$\frac{\gamma_{\vep}}{\sqrt{\mathopen{}\left|\log(\vep)\right|\mathclose{}}}\to\gamma_{*}$ from \eqref{ConvergentPhase} and
\eqref{PhaseandHarmonicGradientBound} then we deduce that $\frac{\mrmd\Phi_{\vep}}{\sqrt{\mathopen{}\left|\log(\vep)\right|\mathclose{}}}\to\mrmd\Phi_{*}$ on
$(\Lambda_{x_{0}})_{\frac{1}{2}}$ and
\[
\mrmd\Phi_{*}=\mrmd\phi_{*}^{m}+\gamma_{*}\hspace{5pt}\text{on }(\Lambda_{x_{0}})_{\frac{1}{2}}.
\]
Observe that since $\gamma_{*}$ and $\Phi_{*}$ are independent of $m$ then by changing $\phi_{*}^{m}$ by a constant we may assume that all
$\phi_{*}^{m}$ coincide on $(\Lambda_{x_{0}})_{\frac{1}{2}}$.
By analyticity, for each $n\ge{}m_{0}$ the functions $\{\phi_{*}^{m}\}_{m\ge{}n}$ coincide on $\calK_{m}$.
Letting $n$ go to infinity, we define their common value $\phi_{*}$ on $M\times(0,\infty)$.
We then have
\beq\label{GlobalGradient}
\mrmd\Phi_{*}=\mrmd\phi_{*}+\gamma_{*}
\eeq
on $(\Lambda_{x_{0}})_{\frac{1}{2}}$.
Since the right-hand-side is globally defined we can then extend $\Phi$.
We also note that $\pp_{t}\Phi_{*}$ is globally defined and equal to $\pp_{t}\phi_{*}$.
\end{proof}

\subsection{Mean Curvature Flows}
\hspace{15pt}The goal of this subsection is to prove \eqref{BOSTheorem:Item5}
from Theorem \ref{BOSTheorem}.
In particular, we focus on studying the properties of the
singular parts of $\{\mu_{*}^{t}\}_{t>0}$, denoted
$\{\nu_{*}^{t}\}_{t>0}$, which for each $t>0$ satisfy
\beq\label{Def:nust}
\nu_{*}^{t}=\Theta_{*}(x,t)\calH^{N-2}\rest\Sigma_{\mu}^{t}
\eeq
where $\Theta_{*}$ and $\Sigma_{\mu}^{t}$ are as in
\eqref{Decomposition}.
As in \cite{BOS2}, and following the same proof,
we will study limiting behaviour of
\begin{equation}\label{TimeDerivative}
\omega_{\vep}^{t}\ceq\frac{|\pp_{t}\uvep|^{2}}{\mathopen{}\left|\log(\vep)\right|\mathclose{}}\dvol_{g}(x)
\end{equation}
and
\begin{equation}\label{MixedDerivative}
\sigma_{\vep}^{t}\ceq\frac{-\pp_{t}\uvep\cdot\nabla\uvep}{\mathopen{}\left|\log(\vep)\right|\mathclose{}}\dvol_{g}(x).
\end{equation}

\toclesslab\subsubsection{Convergence of $\sigma_{\vep}^{t}$}{Subsec:Sigmavept}
\hspace{15pt}By the Cauchy-Schwarz inequality $\sigma_{\vep}$ is uniformly bounded on $M\times[0,T]$ for every $T>0$.
By perhaps passing to a further subsequence, we may assume that $\sigma_{\vep}\rightharpoonup\sigma_{*}$ as measures.
The Radon-Nikodym derivative of $|\sigma_{\vep}|$ with respect to $\mu_{\vep}$ verifies
\begin{equation*}
\frac{\mrmd|\sigma_{\vep}|}{\mrmd\mu_{\vep}}
=\frac{|\pp_{t}\uvep\cdot\nabla\uvep|}{\evep(\uvep)}
\le\frac{\sqrt{2}|\pp_{t}\uvep|\sqrt{\evep(\uvep)}}{\evep(\uvep)}
=\sqrt{2}\frac{|\pp_{t}\uvep|}{\sqrt{\evep(\uvep)}}.
\end{equation*}
On the other hand,
\begin{align*}
\left\|\frac{|\pp_{t}\uvep|}{\sqrt{\evep(\uvep)}}\right\|
_{L^{2}\left(M\times[0,T],\mrmd\mu_{\vep}\right)}^{2}
&=\int_{M\times[0,T]}\!{}\frac{|\pp_{t}\uvep|^{2}}{\evep(\uvep)}\mrmd\mu_{\vep}\mrmd{}t\\
&=\int_{M\times[0,T]}\!{}\frac{|\pp_{t}\uvep|^{2}}{\evep(\uvep)}\cdot
\frac{\evep(\uvep)}{\mathopen{}\left|\log(\vep)\right|\mathclose{}}\dvol_{g}(x)\mrmd{}t\\
&=\int_{M\times[0,T]}\!{}\frac{|\pp_{t}\uvep|^{2}}{\mathopen{}\left|\log(\vep)\right|\mathclose{}}\dvol_{g}(x)\mrmd{}t\\
&\le{}M_{0}
\end{align*}
where we used standard energy estimates for \ref{PGLOriginal} and assumption \eqref{H0} for the last inequality.
We conclude that $\frac{\mrmd|\sigma_{\vep}|}{\mrmd\mu_{\vep}}$ is uniformly bounded in
$L^{2}(M\times[0,T],\mrmd\mu_{\vep})$.
Arguing as in Theorem $2.2$ of \cite{BF}, but adapting to the case of a compact Riemannian manifold without boundary, it follows that $\sigma_{*}$ is
absolutely continuous with respect to $\mus$.
Therefore, we may write
\begin{equation*}
\sigma_{*}=h\mus^{t}\mrmd{}t
\end{equation*}
where $h\in{}L^{2}(M\times[0,T],\mus^{t}\mrmd{}t)$.
We use \eqref{Decomposition} from Theorem \ref{BOSTheorem}
to decompose $\mus^{t}$ into its absolutely continuous part with
respect to $\dvol_{g}$ and its singular part $\nu_{*}^{t}$
satisfying \eqref{Def:nust}.
Arguing as in Proposition $3.1$ of \cite{BOS1} combined with
Theorem \ref{Theorem2} and Lemma \ref{GlobalPhi} we see that
the part of $\sigma_{*}^{t}$ absolutely continuous with
respect to $\dvol_{g}$ has density
$-\pp_{t}\Phi_{*}\cdot\nabla\Phi_{*}$.
We now have
\begin{lemma}\label{MeasureDecompositionLemma}
The measure $\sigma_{*}$ decomposes as $\sigma_{*}=\sigma_{*}^{t}\mrmd{}t$, where for $\calL^{1}$-almost every $t\ge0$,
\begin{equation*}
\sigma_{*}^{t}=-\pp_{t}\Phi_{*}\cdot\nabla\Phi_{*}\dvol_{g}(x)+h\nu_{*}^{t}.
\end{equation*}
\end{lemma}
Next we observe that for every $t\ge0$, by appealing to the ideas found in Lemmas $7.5$ and $7.6$ of \cite{PP2},
we have for all smooth vector fields, $X$, that
\begin{align}\label{MatrixMeasureIdentity}
\frac{1}{\mathopen{}\left|\log(\vep)\right|\mathclose{}}\int_{M\times\{t\}}\!{}
[\evep(\uvep)I-\nabla\uvep\otimes\mrmd\uvep]:DX\dvol_{g}(x)
&=\int_{M}\!{}
\mathopen{}\left<X,\frac{\pp_{t}\uvep\cdot\nabla\uvep}{\mathopen{}\left|\log(\vep)\right|\mathclose{}}\right>\mathclose{}\dvol_{g}(x)\\
&=-\int_{M}\!{}\mathopen{}\left<X,\sigma_{\vep}^{t}\right>\mathclose{}  \nonumber
\end{align}
where $I$ is the identity operator,
$\nabla\uvep\otimes\mrmd\uvep=\nabla\uvep^{1}\otimes\mrmd\uvep^{1}+\nabla\uvep^{2}\otimes\mrmd\uvep^{2}$,
$DX$ is the $(1,1)$-tensor field defined at a
point $p\in{}M$ by
\beq\label{Covariant:Def}
DX_{p}\colon{}v\in{}T_{p}M\to{}D_{v}X,
\eeq
and we use the notation
$A:B$ to denote the inner product of $(1,1)$-tensor fields
on $T_{x}M$ defined by
\beq\label{Def:EndoProduct}
A:B\ceq\sum_{i=1}^{N}\sum_{j=1}^{N}
\mathopen{}\left<A(e_{i}),e_{j}\right>\mathclose{}
\mathopen{}\left<B(e_{i}),e_{j}\right>\mathclose{}
\eeq
where $\{e_{1},e_{2},\ldots,e_{N}\}$ is any orthonormal basis
for $T_{x}M$.
Following \cite{BOS2} we use \eqref{MatrixMeasureIdentity}
as motivation to analyze the weak limit of
\begin{equation*}
\alpha_{\vep}^{t}=\biggl(I-\frac{\nabla\uvep\otimes\mrmd\uvep}{\evep(\uvep)}\biggr)
\mrmd\mu_{\vep}^{t}.
\end{equation*}
Since $|\alpha_{\vep}^{t}|\le{}KN\mu_{\vep}^{t}$ then we may assume that, by perhaps passing to a subsequence, that
\begin{equation*}
\alpha_{\vep}^{t}\rightharpoonup\alpha_{*}^{t}\equiv{}A\cdot\mus^{t}
\end{equation*}
where $A$ is a symmetric $(1,1)$-tensor field and where a $(1,1)$-tensor field is symmetric if for each $x\in{}M$ and each
$u,v\in{}T_{x}M$ we have
\begin{equation*}
\mathopen{}\left<A_{x}(u),v\right>\mathclose{}=\mathopen{}\left<u,A_{x}(v)\right>\mathclose{}.
\end{equation*}
We also recall that a symmetric $(1,1)$-tensor is referred to
as positive semi-definite if for each $x\in{}M$ and each
$u\in{}T_{x}M$ we have
\begin{equation*}
\mathopen{}\left<A_{x}(u),u\right>\mathclose{}\ge0.
\end{equation*}
Finally, notice that if $A,B$ are symmetric $(1,1)$-tensor
fields then we write
\begin{equation*}
A\le{}B
\end{equation*}
if $B_{x}-A_{x}$ is positive semi-definite for each
$x\in{}M$.
We now notice that since $\nabla\uvep\otimes\mrmd\uvep$ is a
positive semi-definite $(1,1)$-tensor field then
\begin{equation}\label{UpperBoundOnOperator}
A\le{}I.
\end{equation}
On the other hand, computing in normal coordinates about a point $x\in{}M$, we have, at $x$, that
\begin{equation*}
\text{tr}_{g}[\{\evep(\uvep)I-\nabla\uvep\otimes\mrmd\uvep\}]
=(N-2)\evep(\uvep)+2\Vep(\uvep).
\end{equation*}
Therefore, since the trace is a linear operation, passing to the limit we obtain
\begin{equation}\label{TaceIdentity}
\text{tr}_{g}(A)=(N-2)+2\frac{\mrmd{}V_{*}}{\mrmd\mus}
\end{equation}
where $\frac{\mrmd{}V_{*}}{\mrmd\mus}$ is the non-negative limiting measure, obtained after passing to a subsequence, of
$\frac{\Vep(\uvep)}{\evep(\uvep)}$.
Taking the limit $\vep\to0^{+}$ in \eqref{MatrixMeasureIdentity},
decomposing $\mus^{t}$ using \eqref{Decomposition} of
Theorem \ref{BOSTheorem},
and using the pointwise estimates provided by Theorem
\ref{Theorem2} we obtain for $\calL^{1}$-almost every $t\ge0$
\begin{align}\label{LimitingTraceIdentity}
\int_{M}\!{}A:DX\mrmd\nu_{*}^{t}
&+\int_{M}\!{}\biggl[\frac{|\nabla\Phi_{*}|^{2}}{2}I-\nabla\Phi_{*}\otimes\mrmd\Phi_{*}\biggr]:DX\dvol_{g}(x)\\
&=-\int_{M}\!{}\mathopen{}\left<X,h\right>\mathclose{}\mrmd\nu_{*}^{t}
-\int_{M}\!{}\mathopen{}\left<X,\pp_{t}\Phi_{*}\nabla\Phi_{*}\right>\mathclose{}\dvol_{g}(x).  \nonumber
\end{align}
Since $\Phi_{*}$ solves the heat equation then we also have, by multiplying the heat equation by
$\mathopen{}\left<X,\nabla\Phi_{*}\right>\mathclose{}$ and arguing in
coordinates similar to Lemmas $7.5$ and $7.6$ of \cite{PP2}, that
\begin{equation}\label{HeatEquationPhaseIdentity}
\int_{M}\!{}\biggl[\frac{|\nabla\Phi_{*}|^{2}}{2}I-\nabla\Phi_{*}\otimes\mrmd\Phi_{*}\biggr]:DX\dvol_{g}(x)
=-\int_{M}\!{}\mathopen{}\left<X,\pp_{t}\Phi_{*}\cdot\nabla\Phi_{*}\right>\mathclose{}\dvol_{g}(x).
\end{equation}
Combining \eqref{LimitingTraceIdentity} and \eqref{HeatEquationPhaseIdentity} now gives the following
\begin{lemma}\label{AMatrixIdentity}
For $\calL^{1}$-almost every $t\ge0$ and for every smooth vector field $X$ we have
\begin{equation}\label{MatrixIntegralIdentity}
\int_{M}\!{}A:DX\mrmd\nu_{*}^{t}=-\int_{M}\!{}\mathopen{}\left<X,h\right>\mathclose{}\mrmd\nu_{*}^{t}.
\end{equation}
\end{lemma}
We see that the conclusion of Lemma \ref{AMatrixIdentity} is
close to \eqref{FirstVariationAbsolutelyContinuous}.
Thus, if we can show that $A$ is the orthogonal projection
operator from $T_{x}M$ onto
$T_{x}\Sigma_{\mu}^{t}$ then we will have shown that
$\nu_{*}^{t}$ has first variation with mean curvature
$h$.
Following \cite{BOS2} we proceed in this direction by
first demonstrating that $A$ is perpendicular to normal
vectors to $T_{x}\Sigma_{\mu}^{t}$.
\begin{lemma}\label{KernelOfOperator}
For $\calL^{1}$-almost every $t\ge0$ and $\calH^{N-2}$-almost every $x\in\Sigma_{\mu}^{t}$ we have
\begin{equation}\label{KernelIdentity}
A_{x}\biggl[\int_{T_{x}\Sigma_{\mu}^{t}}\!{}\nabla\chi(y)\mrmd\calH^{N-2}(y)\biggr]=0
\end{equation}
where $\chi$ is a compactly supported smooth function
on $T_{x}M$ where we use the exponential map to
identify a neighbourhood of zero in $T_{x}M$ with subsets of
$M$.
\end{lemma}
\begin{proof}
As in the corresponding proof from \cite{BOS2} we choose $t\ge0$ for which \eqref{MatrixIntegralIdentity} holds and $x\in\Sigma_{\mu}^{t}$
such that $T_{x}\Sigma_{\mu}^{t}$ exists and such that $x$ is a Lebesgue point for $\Theta_{*}$, with respect to $\calH^{N-2}$, and of $A$ with
respect to $\nu_{*}^{t}$.
We now consider a smooth function $\chi$ with support contained in a normal coordinate neighbourhood centred at $x$.
We then consider, written in normal coordinates centred at $x$, the vector field defined by
$X_{r,l}(y)\ceq\chi\bigl(\frac{-y}{r}\bigr)\frac{\pp}{\pp{}x^{l}}$ for $l\in\left\{1,2,\ldots,N\right\}$.
Inserting $X_{r,l}$ into \eqref{MatrixIntegralIdentity}, taking the limit $r\to0^{+}$, and appealing to the difference of homogeneity of the
right-hand side as in Theorem $3.8$ of \cite{AS}, we conclude that
\begin{equation*}
A_{x}\biggl[\int_{T_{x}\Sigma_{\mu}^{t}}\!{}\nabla\chi(y)\mrmd\calH^{N-2}(y)\biggr]=0.
\end{equation*}
\end{proof}
We have, due to the arguments of Section $6$ of \cite{AS}, the following consequence:
\begin{corollary}\label{TangentPerpKernel}
For $t$ and $x$ as in Lemma \ref{KernelOfOperator},
\begin{equation*}
(T_{x}\Sigma_{\mu}^{t})^{\perp}\subset{\rm{ker}}(A_{x}).
\end{equation*}
\end{corollary}
We now show that $A_{x}=P$ where $P$ is the orthogonal projection of $T_{x}M$ onto $T_{x}\Sigma_{\mu}^{t}$.
\begin{corollary}\label{AIsProjection}
For $t$ and $x$ as in Lemma \ref{KernelOfOperator}, $A_{x}=P$
is the orthogonal projection onto the tangent space
$T_{x}\Sigma_{\mu}^{t}$.
\end{corollary}
\begin{proof}
By \eqref{UpperBoundOnOperator} we have $A_{x}\le{}I_{x}$ for
each $x\in{}M$, and therefore all the eigenvalues of $A_{x}$ are
less than or equal to $1$.
By \eqref{TaceIdentity}, $\text{tr}_{g}(A_{x})\ge{}N-2$ so that
the sum of the eigenvalues of $A_{x}$ is at least $N-2$.
By Corollary \ref{TangentPerpKernel} and our choice of $x$ and
$t$ we know that $A_{x}$ has at least two zero eigenvalues.
Combining the above information allows us to conclude that $A_{x}$
has precisely two zero eigenvalues and $(N-2)$ eigenvalues equal
to $1$.
In particular, since the kernel is $(T_{x}\Sigma_{\mu}^{t})^{\perp}$ then $A_{x}$ is the orthogonal projection onto
$T_{x}\Sigma_{\mu}^{t}$.
\end{proof}
Combining Lemma \ref{AMatrixIdentity} and Corollary
\ref{AIsProjection} we obtain:
\begin{proposition}\label{Proposition9}
For $\calL^{1}$-almost every $t\ge0$, $\nu_{*}^{t}$ has a first variation and
\begin{equation*}
\delta\nu_{*}^{t}=h\nu_{*}^{t}.
\end{equation*}
That is, $h$ is the mean curvature of $\nu_{*}^{t}$.
\end{proposition}
Next, following \cite{BOS2}, we demonstrate the semi-continuity
of $\omega_{\vep}^{t}$ defined in \eqref{TimeDerivative}.
First, we introduce the bundle $B$
whose fiber over $x\in{}M$ is the space of linear maps $T_{x}M\to \R^{2}$, which we identify with $(T_{x}M)^{2}$.
On $B$ we define the measure
\begin{equation*}
\tilde{\omega}_{\vep}^{t}
\ceq\delta_{p_{\vep}(x)}\frac{|\pp_{t}\uvep\cdot{}p_{\vep}|^{2}}{\mathopen{}\left|\log(\vep)\right|\mathclose{}}\dvol_{g}(x)
\end{equation*}
where $p_{\vep}\ceq\frac{\nabla\uvep}{|\nabla\uvep|}$.
By perhaps passing to a further subsequence, we may assume that $\tilde{\omega}_{\vep}^{t}\mrmd{}t\rightharpoonup\omega_{*}$ as measures.
We deduce from the decomposition provided by Theorem
\ref{Theorem2} and the Portmanteau Theorem that:
\begin{lemma}\label{MixedMeasureDecomposition}
The measure $\tilde{\omega}_{*}$ decomposes as $\tilde{\omega}_{*}=\tilde{\omega}_{*}^{t}\mrmd{}t$, and for $\calL^{1}$-almost every $t\ge0$
\begin{equation*}
\tilde{\omega}_{*}^{t}=\Pi_{*,x}^{t}(p)|\pp_{t}\Phi_{*}|^{2}\dvol_{g}(x)+\calM_{*}^{t},
\end{equation*}
where $\Pi_{*,x}^{t}$ is a probability measure on $\left(T_{x}M\right)^{2}$
with support on the unit ball and
$\calM_{*}^{t}=\tilde{\omega}_{*}^{t}\rest\Sigma_{\mu}^{t}$.
\end{lemma}
We borrow the following proposition, after adapting it to the case of a manifold, from Section $6$ of \cite{AS}
\begin{proposition}\label{Proposition10}
For $\calL^{1}$-almost every $t\ge0$ and every smooth function $\chi$ we have
\begin{equation*}
\int_{B}\!{}\chi(x)\calM_{*}^{t}(x,p)
\ge\int_{M}\!{}\chi\left|h\right|^{2}\mrmd\nu_{*}^{t}.
\end{equation*}
\end{proposition}
We are now ready to prove \eqref{BOSTheorem:Item5} of Theorem
\ref{BOSTheorem}.
\begin{proof}
We begin by using Lemma \ref{DerivativeOfEnergy}, integating over $[T_{0},T_{1}]$, and dividing by $\mathopen{}\left|\log(\vep)\right|\mathclose{}$.
Next we let $\vep\to0^{+}$.
Then by combining Lemma \ref{MeasureDecompositionLemma}, Proposition \ref{Proposition9}, Lemma \ref{MixedMeasureDecomposition}, and
Theorem \ref{Theorem2} we obtain
\begin{align}
\nu_{*}^{T_{1}}-\nu_{*}^{T_{1}}
&+\int_{M\times\{T_{1}\}}\!{}\chi\frac{|\nabla\Phi_{*}|^{2}}{2}\dvol_{g}(x)
-\int_{M\times\{T_{0}\}}\!{}\chi\frac{|\nabla\Phi_{*}|^{2}}{2}\dvol_{g}(x)    \label{GeneralizedBrakkeFlow}\\
&\le-\int_{M\times[T_{0},T_{1}]}\!{}\chi|h|^{2}\mrmd\nu_{*}
+\int_{M\times[T_{0},T_{1}]}\!{}\mathopen{}\left<\nabla\chi{},P(h)\right>\mathclose{}\mrmd\nu_{*}  \nonumber\\
&-\int_{M\times[T_{0},T_{1}]}\!{}\chi|\pp_{t}\Phi_{*}|^{2}\dvol_{g}(x)\mrmd{}t
+\int_{M\times[T_{0},T_{1}]}\!{}\mathopen{}\left<\nabla\chi,\pp_{t}\Phi_{*}\nabla\Phi_{*}\right>\mathclose{}.  \nonumber
\end{align}
Since $\Phi_{*}$ solves the heat equation, we have the identity
\begin{align}
&\int_{M\times\{T_{1}\}}\!{}\chi\frac{|\nabla\Phi_{*}|^{2}}{2}\dvol_{g}(x)
-\int_{M\times\{T_{0}\}}\!{}\chi\frac{|\nabla\Phi_{*}|^{2}}{2}\dvol_{g}(x)  \label{HeatEquationIdentity}\\
&=\int_{M\times[T_{0},T_{1}]}\!{}\chi|\pp_{t}\Phi_{*}|^{2}\dvol_{g}(x)\mrmd{}t
+\int_{M\times[T_{0},T_{1}]}\!{}\mathopen{}\left<\pp_{t}\Phi_{*}\cdot\nabla\Phi_{*},\nabla\chi\right>\mathclose{}\dvol_{g}(x)\mrmd{}t.  \nonumber
\end{align}
Combining \eqref{GeneralizedBrakkeFlow} and \eqref{HeatEquationIdentity} gives
\begin{equation*}
\nu_{*}^{T_{1}}-\nu_{*}^{T_{1}}
\le-\int_{M\times[T_{0},T_{1}]}\!{}\chi|h|^{2}\mrmd\nu_{*}
+\int_{M\times[T_{0},T_{1}]}\!{}\mathopen{}\left<\nabla\chi{},P(h)\right>\mathclose{}\mrmd\nu_{*}.
\end{equation*}
Applying Theorem $4.4$ of \cite{AS}, whose proof extends to the case of a compact Riemannian manifold, completes the proof of \eqref{BOSTheorem:Item5} of Theorem \ref{BOSTheorem}.
\end{proof}

\bibliographystyle{acm}
\bibliography{ms}

\end{document}